\documentclass[10pt,reqno]{amsart}

 \usepackage{amsmath,amsfonts,amssymb,amscd,amsthm,amsbsy,bbm, epsf,calc,enumerate,color,graphicx,verbatim,cite,import}
\usepackage{color}
\usepackage{datetime,appendix}
\usepackage{chapterbib,epstopdf}

\usepackage{float}

\usepackage{graphicx}

\usepackage{ifpdf}
\ifpdf
    \usepackage[colorlinks,citecolor=black,linkcolor=black]{hyperref}
\fi

\usepackage[mathlines]{lineno}
\theoremstyle{plain}
\newtheorem{definition}{Definition}[section]
\newtheorem{thm}{Theorem}[section]
\newtheorem{theorem}[thm]{Theorem}

\newtheorem{lemma}[thm]{Lemma}

\newtheorem{corollary}[thm]{Corollary}
\newtheorem{prop}[thm]{Proposition}
\newtheorem{proposition}[thm]{Proposition}
\theoremstyle{definition}

\theoremstyle{remark}                  %% For unnumbered Remarks, etc.

\newtheorem{remark}[thm]{Remark}

\definecolor{darkgreen}{rgb}{0,0.4,0}

\newcommand{\e}{\eps}
\newcommand{\R}{\mathbb{R}}
\def \l {\left(}
\def\r {\right)}

\newcommand{\calL}{{\mathcal L}}

\newcommand{\calG}{{\mathcal G}}

\newcommand{\eps}{\varepsilon}

\numberwithin{equation}{section}

\def\XXint#1#2#3{{\setbox0=\hbox{$#1{#2#3}{\int}$ }
\vcenter{\hbox{$#2#3$ }}\kern-.6\wd0}}

\title{Porous Medium Equation with A Drift: Free boundary Regularity}
 
\author[Inwon Kim and Yuming Paul Zhang]{\bfseries Inwon Kim and  Yuming Paul Zhang}

\address{
(I. Kim) Department of Mathematics \\ % \hfill (Received 00 00 2010)\\
University of California   \\ %\hfill (Revised  00 00 2010)\\
Los Angeles\\
USA}
\email{ikim@math.ucla.edu}

\address{
(Y. Zhang) Department of Mathematics \\ % \hfill (Received 00 00 2010)\\
University of California   \\ %\hfill (Revised  00 00 2010)\\
Los Angeles\\
USA}
\email{yzhangpaul@math.ucla.edu}

\begin{document}

%{\begin{flushleft}\baselineskip9pt\scriptsize
%PUBLICATIONS DE L'INSTITUT MATH\'EMATIQUE\newline
%Nouvelle s\'erie, tome 91(105) (2012), od--do \hfill DOI:
%\end{flushleft}}
\vspace{18mm} \setcounter{page}{1} \thispagestyle{empty}

\begin{abstract}
We study regularity properties of the free boundary for solutions of the porous medium equation with the presence of drift.   We show the $C^{1,\alpha}$ regularity of the free boundary, when the solution is directionally monotone in space variable in a local neighborhood. The main challenge lies in establishing a local non-degeneracy estimate (Theorem 1.3 and Proposition 1.5), which appears new even for the zero drift case.
\end{abstract}

\maketitle

%\vspace{.1cm}
%\noindent{\small {\bf Keywords:}  free boundaries, degenerate diffusion, smooth drift, propagation/formation of singularities.}

%\vspace{.1cm}
%\noindent{\small  {\bf 2010 Mathematics Subject Classification}: 35R35, 35K65,	35A21.}

%\bigskip

%\tableofcontents
\section{Introduction}

Let us consider the drift-diffusion equation
\begin{equation}\label{main}
\varrho_t=\Delta \varrho^m+\nabla\cdot (\varrho\, {\vec{b}}) \quad \hbox{ in } Q:=\mathbb{R}^d\times (0,\infty),
\end{equation}
with a smooth vector field $\vec{b}:Q \to \mathbb{R}^d$, a non-negative initial data $\varrho(\cdot,0)=\varrho_0$ and $m>1$.  
The nonlinear diffusion term in \eqref{main} represents an anti-congestion effect (\cite{BGHP,TBL,HW,W}). 
\medskip

 Our interest is on the regularity of the {\it free boundary}  $\partial\{\varrho>0\}$, which is present at all times if starting with a compactly supported initial data. We are motivated by the intriguing fact that the free boundary regularity is open even for the travelling wave solutions in two space dimensions, with a smooth and laminar drift $\vec{b}(x_1,x_2) = (\sin x_2 ,0)$ (see \cite{traveling}). Our analysis provides a starting point of the discussion in a general framework, but the full answer to this question remains open (see Theorem~\ref{thm:trav} and the discussion below). The presence of the drift generates several significant challenges that are new to the problem, as we will discuss below.

\medskip

To illustrate the regularizing mechanism of the interface, let us write \eqref{main} in the form of continuity equation,
$$
\varrho_t - \nabla\cdot ((\nabla u + \vec{b})\varrho)=0,
$$
where
\begin{equation}
    \label{rho u}
    u=\frac{m}{m-1}\varrho^{m-1}.
\end{equation}

Hence formally the normal velocity for the free boundary can be written as 
\begin{equation} \label{with drift speed}
   V = - (\nabla u+\vec{b})\cdot \vec{n} = |\nabla u| -\vec{b}\cdot \vec{n} \quad \hbox{ on } (x,t) \in \Gamma:=\partial\{u>0\},
\end{equation}
where $\vec{n}=\vec{n}_{x,t}$ is the outward normal vector at given boundary points. Given that $\varrho$ solves a diffusion equation, it would be natural to expect that the free boundary is regularized by the pressure gradient $|\nabla u|$ if $\vec{b}$ is smooth, as long as $u$ stays non-degenerate near the free boundary and topological singularities are ruled out. In general neither can be guaranteed even with zero drift.  Below we discuss our main results and new challenges in the context of literature. We will always assume that
\begin{equation}
    \label{maincondition}
    \vec{b}\in C^{3,1}_{x,t}(Q)\quad \text{ and }\quad \varrho_0\in L^1(\mathbb{R}^d)\cap L^\infty(\mathbb{R}^d).
\end{equation} 
%\textcolor{We will }

\medskip

\textbf{Literature}

\medskip

Let us first discuss the case $\vec{b}=0$, in which case  our problem \eqref{main} corresponds to the well-known {\it Porous Medium Equation} $(PME)$. In this case a vast amount of literature is available: we refer to the book \cite{book}. What follows is a briefly discussion of  several prominent results that are relevant to our results. Aronson and Benilan \cite{fundamentalest} showd the semi-convexity estimate $\Delta u> -\infty \;\text{ for } t>0$ which played a fundamental role in the regularity theory of $(PME)$.  In general there can be a {\it waiting time} for degenerate initial data, where the free boundary does not move and regularization is delayed. When the initial data $u_0=u(\cdot,0)$ has super-quadratic growth at the free boundary, Caffarelli and Friedman \cite{CFregularity} showed that  there is no waiting time and the support of solution strictly expands in time. There an expansion rate of the support was obtained, by showing that its free boundary can be represented as $t=S(x)$ where $S$ is H\"{o}lder continuous. To discuss further regularity results, it is natural to require some geometric properties of the solution to rule out topological singularities such as merging of two fingers. The $C^{1,\alpha}$ regularity of the free boundary is established by Caffarelli and Wolanski \cite{C1alpha}, under the assumption of non-degeneracy and Lipschitz continuity of solutions. Their assumptions are shown to hold after a finite time $T_0>0$ by Caffarelli, Vazquez and Wolanski \cite{CVWlipschitz}, where $T_0$ is the first time the support of solution expands to contain its initial convex hull.  More recently, Kienzler explored the stability of solutions that are close to the flat traveling wave fronts to $(PME)$ \cite{kienzler}. Later Kienzler, Koch and Vazquez \cite{flatness}  improved this result and showed that solutions that are locally close to the traveling waves are smooth: see further discussion on their result in comparison to ours below Theorem~\ref{thm 1.3}.

\medskip

 When $\vec{b} \neq 0$, few results are available on the free boundary regularity of \eqref{main}. With the exception of  the particular choice $\vec{b}=x$, {in general} there appears to be no change of coordinates that eliminates the drift dependence in  \eqref{main}. Numerical experiments in \cite{numerics} present the interesting possibility that an initially planar solution with smooth drift could develop corners without topological changes. However the non-degeneracy of pressure or the free boundary regularity is unknown even for  traveling wave solutions  in $\R^2$ (see \cite{traveling}).  By comparison, well-posedness and regularity theory for the solutions of \eqref{main} has been much better understood. Existence and uniqueness results are shown in \cite{Bertsch} and \cite{dib83} for weak solutions and in \cite{kimlei} for viscosity solutions. Asymptotic convergence to equilibrium of \eqref{main} is shown in \cite{carrillo2001entropy} using energy dissipation when $\vec{b}$ is the gradient of a convex potential. Recently \cite{IKYZ,hwang2019continuity} proved H\"{o}lder continuity of solutions for uniformly bounded, but possibly non-smooth drifts.

\bigskip

 {\bf Discussion of main results and difficulties}

\medskip

For our analysis, we will consider the pressure variable \eqref{rho u} and the equation it satisfies:
\begin{equation}\label{premain}
u_t=(m-1) u\,\Delta u+|\nabla u|^2+\nabla u\cdot{\vec{b}}+(m-1)u \,\nabla\cdot{\vec{b}}
\end{equation}
in $Q= \mathbb{R}^d\times (0,\infty)$. 

\medskip

We first show the semi-convexity (Aronsson-Benilan) estimate through a simple but novel barrier argument on $\Delta u$. This is where we use the $C^3_x$ norm of $\vec{b}$. 

\begin{thm}{\rm [Theorem~\ref{fundamental estimate}]} %Suppose $\vec{b} \in C^{3,0}_{x,t}$. 
Let $\rho$ solve \eqref{main} in $Q$ with \eqref{maincondition}, and let $u$ be the corresponding pressure variable given by \eqref{rho u}.
Then for some $\sigma>0$, $\Delta u>-\frac{\sigma}{t}-\sigma$ in the sense of distribution for all $t> 0$.
\end{thm}

%For the following discussions on local regularity, let us assume$\Delta u\geq -C_0$ and $\|\vec{b}\|_{C_x^2}+\|\partial_t\vec{b}\|_{C_x^1}\leq L$ for some $C_0$ and $L$.

%Next we discuss a weak non-degeneracy property. With zero drift this corresponds to the \textcolor{waiting times} and strict expansion \textcolor{phenomena} of the positive set, see section 14 \cite{book}. 

Next we discuss a weak non-degeneracy property in the event of zero initial waiting time. With zero drift this corresponds to the strict expansion property of the positive set, see section 14 \cite{book}. 
In our case this property needs to be understood in terms of the  \textit{streamlines}, defined as
\begin{equation}\label{ode}
X(t):=X(x_0,t_0;t) \hbox{ is the unique solution of the ODE }\,\, \left\{\begin{aligned}
    &\partial_t X(t)=-\vec{b}(X(t),t_0+t), \quad t\in \R,\\
    &X(0)=x_0.
    \end{aligned}\right.
\end{equation}

%We will use the notation $\Omega:=\{(x,t),\, u(x,t)>0\}$ and $\Omega_t:=\{x,\,u(\cdot,t)>0\}$. 

While the streamlines are a natural coordinate for us to measure the strict expansion of the positive set over time, it does not cope well with the diffusion term in the equation. The most delicate scenario occurs with degenerate pressure, where the time range we need to observe is much larger than the space range. To deal with such case we need to carefully localize $\vec{b}$.

\begin{thm}\label{intro 1.3}
{\rm [Theorem~\ref{strictly expanding}] }
Let $u$ be as given in Theorem 1.1, and fix $(x_0,t_0) \in \Gamma:= \partial\{u>0\} \cap \{t>0\}$. Then either of the following holds:
\begin{itemize}
\item[] (Type one)  $X(-s):=X({x_0,t_0};-s) \in \Gamma$ for $s\in[0,t_0]$;\\
\item[] (Type two) {there exist} $C_*,\beta>1$ and $h>0$ such that for $s\in (0,h)$
\begin{align*}
   % u(x,t)=0 \text{ if }|x-X(t-t_0)|\leq C(t_0-t)^\beta \text{ for }t_0-h<t<t_0,\\
   % u(x,t)>0 \text{ if }|x-X(t-t_0)|\leq C(t-t_0)^\beta \text{ for }t_0<t<t_0+h.
   & u(x,t_0-s)=0\quad \text{ if }|x-X(-s)|\leq C_* s^\beta,\\
   &u(x,t_0+s)>0\quad \text{ if }|x-{X(s)}|\leq C_* s^\beta.
\end{align*}
\end{itemize}
%Here $\beta$ only depends on $m,d,C_0$ and $\|\nabla \vec{b}\|_\infty$.
Moreover, {if $u_0$ satisfies the near-boundary growth estimate}
\begin{equation}
\label{intro less quar grow}
u_0(x)\geq \gamma({d}(x,\Omega_0^C))^{2-\varsigma} \,\, \hbox{ for some } \gamma, \varsigma>0  ,
\end{equation}
then any point on $\Gamma$ is of type two.% with $C_*,h$ only depending on $\gamma, \varsigma, t$ and $\|\nabla\vec{b}\|_\infty$. 

\end{thm}

 The growth condition in \eqref{intro less quar grow} is optimal, since there is a stationary solution to \eqref{main} with a corner on its free boundary and with quadratic growth (see Theorem~\ref{prop corner}).

\medskip

Next we proceed to show the non-degeneracy property of $u$, as it is essential for the regularity of its free boundary. This step presents the most challenging and novel part of our analysis.  To illustrate the difficulties, let us briefly go over the main components of the celebrated arguments in \cite{CVWlipschitz}, which provides non-degeneracy of solutions  for (PME) for times $t>T_0$.
One key ingredient in their analysis was the scale invariance of the equation under the transformation
$$u_{\e, A}(x,t) := \frac{1+A\e}{(1+\e)^2} u((1+\e)x, (1+A\e)t +B)\quad \hbox{ for any }A,B, \e>0,
$$ 
 In \cite{CVWlipschitz}  $u_{\e, A}$ was compared to $u$ to obtain the space-time directional monotonicity  
  \begin{equation}
     \label{space-time mono}
     x\cdot \nabla u + (At +B) u_t \geq 0\quad \text{ on }\Gamma.
 \end{equation}
 Applying \eqref{with drift speed} with $\vec{b}=0$ we then have
 $$
 |\nabla u| = V=   \frac{u_t}{|\nabla u|} \geq \frac{1}{(At+B)} \nu \cdot (\frac{x}{|x|})\quad \hbox{ on } \Gamma,
 $$ 
where the first equality is from \eqref{with drift speed}, the second equality is due to the level set formulation of the normal velocity, and the last inequality is due to \eqref{space-time mono} and the fact that $\nabla u$ is parallel to the negative normal $-\nu$ on the free boundary.  Thus the non-degeneracy follows if we know that the free boundary is a Lipschitz graph with respect to the radial direction. This was shown in \cite{CVWlipschitz} for $t>T_0$ by the celebrated moving planes arguments, and thus we can conclude.

\medskip

For nonzero drift, neither scaling invariance nor the moving planes method is available due to the inhomogeneity in $\vec{b}$. In fact it is not reasonable to expect consistent free boundary behavior for large times, except possibly when $\vec{b}$ is a potential vector field. Still, it is reasonable to expect that, without topological singularites and waiting time, the diffusive nature of the equation \eqref{premain} regularizes the free boundary.  With this in mind we show a local non-degeneracy result under the assumption of directional monotonicity and zero waiting time.

\medskip

Let us define the spatial cone of directions
 \begin{equation}\label{cone}
 W_{\theta,\mu} := \left\{y\in\R^d: \left|\frac{y}{|y|}-\mu\right| \leq  2\sin\frac{\theta}{2}\right\} \quad\hbox{ with axis }\mu \in \mathcal{S}^{d-1} \hbox{ and } \theta \in (0,\pi/2].
 \end{equation}
We say $u$ is {\it monotone} with respect to $W_{\theta,\mu}$ if $u(\cdot,t)$ is non-decreasing along  directions in  $W_{\theta, \mu}$. We also denote $Q_r:= \{|x|\leq r\} \times (-r,r)$.

\begin{thm} {\rm [Local Non-degeneracy, Corollary~\ref{cor nondeg}] } \label{thm 1.3}   Let $\varrho$ be a weak solution to \eqref{main} in $Q_2 $, where $\Gamma$ is of type two, and let $u$ be the pressure. Suppose in $Q_2$, $\Delta u>-\infty$ and $u$ is monotone with respect to $W_{\theta, \mu}$ for some $\theta$ and $\mu$. Then 
there exists $\kappa_*>0$ such that 
\begin{equation*}\label{thm rs wk nondeg}
\liminf_{\e\to 0^+ }  \dfrac{u (x+\e\mu, t)}{\e} \geq \kappa_* \quad \hbox{ for } (x,t)\in \Gamma\cap Q_1.
\end{equation*}
\end{thm}

{For the proof we adopt a local barrier argument introduced in \cite{choijerisonkim} in the context of the Hele-Shaw flow. Heuristically speaking the barrier argument illustrates the fact that the nondegeneracy property of positive level sets propagates to the free boundary as the positive set expands out in diffusive free boundary problems.}

\medskip

As mentioned above, in the zero drift case \cite{flatness} considered solutions that are locally close to a planar traveling wave solution. Their assumption in particular endows a discrete small-scale flatness and non-degeneracy. It was shown there that over time the flatness improves in its scale to yield the smoothness of the solutions. It was conjectured there whether a cone monotonicity assumption could replace proximity to the planar travelling waves. While we do not pursue improvement of flatness in scale, our result yields a positive partial answer to this question.

\medskip

Building on the above non-degeneracy result, we proceed to study the free boundary regularity. To prevent sudden changes in the evolution caused by changes in the far-away region, we assume that, in the weak sense,
\begin{equation}
    \label{cond ii}   u_t \leq  A\,(\mu\cdot \nabla u + u + 1) \quad \hbox{  in } Q_1 \text{  for some } A>0.
\end{equation}

\begin{thm}{\rm [Theorem~\ref{col C1alpha}]}\label{thm 1.4} Let $u$ be given as in Theorem \ref{thm 1.3}. If in addition \eqref{cond ii} holds, then  $u$ is Lipschitz continuous and $\Gamma$ is $C^{1, \alpha}$ in $Q_{1/2}$.
\end{thm}

The proof of above theorem is given in Section 6. The novel ingredient in this section is the following result, which propagates the non-degeneracy of the solution at the free boundary to nearby positive level sets.

\begin{prop}\label{prop:main} [Propagation of non-degeneracy, Proposition~\ref{condlem nondeg'}]
Under the assumption of Theorem~\ref{thm 1.4},  there exist $\delta<\frac{1}{2}$ and $c_1>0$ such that
%\begin{equation}\label{local non deg}\frac{1}{C}\leq |\nabla u(x,t)|\leq C,\end{equation}
\begin{equation*}
    \nabla_{{{\mu}}} u(x,t)\geq c_1\quad \text{  in $\{u>0\}\cap Q_\delta$.}
\end{equation*} 
\end{prop}

From here, the proof of Theorem 1.4 largely follows the iterative argument given in \cite{C1alpha}, which compares in different scales the solution with its shifted version. For nonzero drifts \eqref{premain} changes under coordinate shifts, and thus a notable modification is necessary in the iteration procedure. See Remark~\ref{remark}.

\medskip

Now we address the traveling wave solutions discussed earlier in the introduciton. 

\begin{thm}\label{thm:trav}
Let $\alpha: \R\to \R$ be a smooth and bounded function. Let $u$ solve \eqref{premain} in $Q=\mathbb{R}^2\times (0,\infty)$
with $\vec{b}=(\alpha(x_2),0)$ and the initial data $u_0(x) = u(x,0)= (x_1)_+$, under linear growth condition at infinity.\\
Then $\Gamma$ is locally uniformly $C^{1,\alpha}$ in $Q$.

\end{thm}

In \cite{traveling} the existence of traveling wave solutions are shown with the above choice of $\vec{b}$. We consider the initially planar solution that was used in \cite{numerics} to approximate the traveling waves. Our argument yields an exponentially decaying lower bound on the nondegeracy of $u$. While it rules out the possibility of finite time singularity for the approximate solutions, the free boundary regularity of travelling wave solutions remains open.

\medskip

Lastly we present some examples which illustrate new types of free boundary singularities generated by drifts.

%For the Hele-Shaw problem, corners on the free boundary can be preserved for a finite time if the angles are small enough, see \cite{choijerisonkim}. But for our equation, corners of any angle can be preserved or actually they even shrink due to the force of the drift. Also because of the existence of a stationary solution with non-smooth support, we know that angles on the free boundary can be preserved for all time. As mentioned before, we have the formation of corners and cusps. But for this part we can only find examples with continuous spatial vector fields. This leaves open of the question that whether we have formation of corners/cusps given smooth drifts . 

\begin{thm}\label{intro 1.5} [Theorem~\ref{prop corner} and ~\ref{prop cusp}].
There is  $\vec{b} \in C_x^3(\R^d)$ such that  \eqref{premain} has a stationary profile with a corner on its free boundary. There is a continuous spatial vector field $\vec{b}$ such that an initially smooth solution to \eqref{premain} develops singularity on the free boundary in finite time.
\end{thm}
  \medskip

\textbf{Acknowledgements.} Both authors are partially supported by NSF grant DMS-1566578. We would like to thank Jean-Michel Roquejoffre and Yao Yao for helpful discussions.

\medskip

\section{Preliminaries}

\noindent $\circ$\textbf{ Notations.}

\medskip
\begin{itemize}
\item Throughout the paper we denote $\sigma$ as various \textit{universal constants}, by which we mean constants that only depend on $m$, $d$, $\|\vec{b}\|_{C^{3,1}_{x,t}}$ and $\|\varrho_0\|_{L^1}+\|\varrho_0\|_{L^\infty}$. \\

\item We use $C$ to represent constants which might depend on universal constants and other constants that are given in the assumptions of corresponding theorems. \\

\item For a continuous, non-negative function $u:\mathbb{R}^d\times (0,\infty) \to \R$, we denote
$$\Omega(u):=\{u>0\}, \quad \Omega_t(u):=\{u(\cdot,t)>0\}
$$
and
$$
\Gamma_t(u):=\partial\Omega_t,\quad \Gamma(u):=\bigcup_{t\in(0,\infty)}(\Gamma_t\times\{t\}).
$$
When it is clear from the context we will omit the dependence on $u$.\\

\item  $B(x,r):= \{x\in\mathbb{R}^d: |x|\leq r\}$, $B_r:=B(0,r)$,  $Q=\mathbb{R}^d\times [0,\infty)$ and $Q_r:=B_r\times (-r,r)$.\\

\item $\nabla := \nabla_x$, and $\hat{\nabla}:=(\nabla,\partial_t)$. We also denote  $f_i:=\partial_{x_i} f$, $ f_{ij}:=\partial^2_{x_ix_j} f $.\\

\item For $\nu,\mu \in  \R^d \setminus\{0\}$,  the angle between them are denoted by
\[\langle \nu,\mu\rangle:=\arccos\left(\frac{\nu\cdot\mu}{|\nu||\mu|}\right)\in [0,\pi].\]
For $\mu\in\mathbb{R}^d$, $\nu\in \mathbb{R}^{d+1}$ and $\theta\in [0,\pi/2]$, we define the space and space-time cones by
\begin{equation}\label{spacetime cone}
W_{\theta,\mu} := \{p\in\mathbb{R}^{d}: \,\langle p,\mu\rangle \leq \theta\},   \quad  \widehat{W}_{\theta,\nu}:=\{p\in\mathbb{R}^{d+1}:\, \langle p,\nu\rangle \leq \theta\}.
\end{equation} 

\end{itemize}

\vspace{20pt}

\noindent$\circ$\textbf{ Notions of solutions and their smooth approximations.}

\medskip

Next we recall the notion of weak solutions and their properties, including their smooth approximations that will be used in this paper.

\medskip

\begin{definition}\label{def1.1}
Let $\varrho_0$ be a non-negative function in  $L^\infty(\mathbb{R}^d)\cap L^1(\mathbb{R}^d)$. We say that a non-negative and bounded function $\varrho:\mathbb{R}^d\times [0,T]\to[0,\infty)$ is a subsolution (resp. supersolution) to \eqref{main} with initial data $\varrho_0$ if \begin{equation}\label{definitionsol}
 \varrho\in C([0,T],L^1(\mathbb{R}^d)), \,\varrho\,\vec{b}\in L^2([0,T]\times\mathbb{R}^d)\;\text{ and }\; \varrho^m\in L^2(0,T, \dot{H}^1(\mathbb{R}^d))
 \end{equation}
and
\begin{equation}
    \label{def sol2}
    \int_0^T\int_{\mathbb{R}^d} \varrho\,\phi_tdxdt\geq (\text{resp. } \leq) \int_{\mathbb{R}^d} \varrho_0(x)\phi(0,x)dx+\int_0^T\int_{\mathbb{R}^d}(\nabla \varrho^m+\varrho\,\vec{b})\nabla\phi\; dxdt,
    \end{equation}
     for all non-negative  $\phi\in C_c^\infty(\mathbb{R}^d\times [0,T))$.

We say $\varrho$ is a weak solution to \eqref{main} if it is both sub- and supersolution of \eqref{main}. We also say that $u:= \frac{m}{m-1}\varrho^{m-1}$ is a {\it solution} ({resp.} {\it super/sub solution}) to \eqref{premain} if $\varrho$ is a weak solution ({resp.} super/sub solution) to \eqref{main}.

%, or equivalently, 
%it satisfies for all test function $\phi\in C^\infty_c( \mathbb{R}^d\times [0,T))$, 
%\[\int_0^T\int_{\mathbb{R}^d} \varrho\phi_tdxdt=\int_{\mathbb{R}^d} \varrho_0(x)\phi(0,x)dx+\int_0^T\int_{\mathbb{R}^d}(\nabla \varrho^m+\varrho\vec{b})\nabla\phi\; dxdt.\]
\end{definition}

%{Let $U\subseteq\mathbb{R}^d$. We say a non-negative function $\rho: U\times[0,T]\to[0,\infty)$ is a solution to 
%\[\varrho_t=\Delta \varrho^m+\nabla\cdot (\varrho\, {\vec{b}}) \quad \hbox{ in } U\times (0,T)\]
%with initial data $\varrho_0$, if $\rho$ satisfies \eqref{definitionsol} with $\mathbb{R}^d$ replaced by $U$ and for all $\phi\in C_c^\infty(U\times [0,T))$, the equality \eqref{def sol2} holds. }

%=\frac{m}{m-1}\varrho^{m-1}

The well-posedness result of general degenerate parabolic type equations is established in \cite{alt1983quasilinear} - \cite{Bertsch}.  \cite{di1982continuity,dib83}  proved the H\"{o}lder regularity of solutions.

\begin{thm}[Theorem 1.7, \cite{alt1983quasilinear}]\label{exitence boundedness}
Let $\varrho_0$ {be} as given in Definition~\ref{def1.1}. When $\vec{b}\in C^3(Q)$, then there exists a weak solution $\varrho$ to \eqref{main} with initial data $\varrho_0 $. Moreover $\varrho$ is uniformly bounded for all $t\geq 0$.
\end{thm}

\begin{thm}[Theorem 1, \cite{dib83}]\label{holder}
Suppose $\varrho$ is a non-negative, bounded weak solution to \eqref{main} in $Q_1$. Then $\varrho$ is H\"{o}lder continuous in $Q_{\frac{1}{2}}$.
\end{thm}

\begin{thm}[Theorem 2.2, \cite{alt1983quasilinear}]\label{thm:comp}
Suppose $U$ is an open subset of $\mathbb{R}^d$ and $\vec{b}\in C^{1,0}_{x,t}$. Let $\bar{\varrho},\underline{\varrho}$ be respectively a subsolution and a supersolution of \eqref{main} in $U\times\mathbb{R}^+$ such that $\bar{\varrho}\leq \underline{\varrho}$ a.e. in the parabolic boundary of $U\times \mathbb{R}^+$. Then $\bar{\varrho} \leq\underline{\varrho}$ in $U\times\mathbb{R}^+$.
\end{thm}

\begin{remark}\label{remark 2.0}
Following from Theorem \ref{thm:comp}, we have comparison principle for \eqref{premain}: suppose $\bar{u},\underline{u}$ are respectively a subsolution and a supersolution of \eqref{premain} in $U\times\mathbb{R}^+$ such that $\bar{u}\leq \underline{u}$ a.e. on the parabolic boundary of $U\times \mathbb{R}^+$. Then $\bar{u} \leq\underline{u}$ in $U\times\mathbb{R}^+$.
\end{remark}

\medskip

%\begin{theorem}[Theorem 3.5 \cite{IKYZ}]\label{L1 contraction}Suppose $\|V\|_{L^2}<\infty$ and $1<m< 2$. Let $u_1,u_2$ be two non-negative weak solutions to \eqref{main} with initial datas $u_{1,0},u_{2,0}$ respectively. Assume in addition that \[\partial_t( u_1- u_2)\in L^1(\mathbb{R}^d\times [0,T]).\]Then the following holds:\[\int_{\mathbb{R}^d} (u_1-u_2)_+(t) dx\leq \int_{\mathbb{R}^d} (u_{1,0}-u_{2,0})_+ dx \quad\hbox{ for } 0\leq t\leq T.\]\end{theorem}

%Suppose $u_0$ is compactly supported and $V$ is bounded. By finite propagation property, $u$ is compactly supported (we will only consider such solutions in this paper) for any finite time and we can truncate the vector field $V$ without changing the solution. And the above $L^1$-contraction (Theorem \ref{L1 contraction}) still holds.

%Let us recall the famous \textit{Barenblatt solution}:\[\mathcal{B}(x,t)=\frac{m}{m-1}t^{-\alpha}(C-k|x|^2t^{-2\beta})_+\]where\[\alpha=\frac{d(m-1)}{d(m-1)+2},\quad\beta=\frac{\alpha}{d(m-1)},\quad k=\frac{\alpha}{2md}\]and $C>0$ is arbitrary. It is a fundamental solution to the Porous Medium Equation:\[\mathcal{B}_t=(m-1)\mathcal{B}\Delta \mathcal{B}+|\nabla \mathcal{B}|^2\]and $\mathcal{B}(x,t)\to M(C,m,d)\delta(x)$ (a delta mass) as $t\to 0$.

 %If $\|V\|_\infty+\|D_x V\|_\infty+\|V_t\|<+\infty$, it is straightforward to see\[X(x,s)=x+V(x,0)s+o(|s|).\]

 In our analysis it is often convenient to work with classical solutions of \eqref{main}, which is made possible by the following result. We will rely on this approximation lemma in Theorem \ref{fundamental estimate} and in Section \ref{sec 5}.
 
\begin{lemma}[Section 9.3 \cite{book}]\label{approximation}
Let $U$ be either $B_1$ or $\mathbb{R}^d$, and consider $\varrho_0\in L^1(U)\cap L^\infty(U)\cap C(U)$.  Let $\varrho$ be a weak solution of  \eqref{main} in $U \times [0,1]$ that is in $C(\overline{U}\times [0,1])$ with initial data $\varrho_0$. %{Do you need continuity up to $t=0$?} 
Then there exists a sequence of strictly positive, classical solutions $\varrho_k$ of \eqref{main}  such that $\varrho_k\to\varrho$ locally uniformly in $U\times (0,1]$ as $k\to \infty$. 
\end{lemma}
\begin{proof}

Let us consider $U=B_1$.
Consider $\varrho_{0,k}=\varrho_0+\frac{1}{k}$ and let $\varrho_k$ be the weak solution to \eqref{main} in $U$ with initial data $\varrho_{0,k}$ and Dirichlet boundary condition $\varrho_k=\varrho+\frac{1}{k}$ on $\partial U \times (0,1]$. Note that 
\[\psi(x,t):=\frac{1}{k}\exp(-\|\nabla\cdot\vec{b}\|_\infty t)\]
is a subsolution to \eqref{main} in $U\times (0,1]$ with $\psi\leq \frac{1}{k}$ on the parabolic boundary. Thus from the comparison principle it follows that 
\[\varrho_k(x,t)\geq\psi(x,t)>0.\]
Since $\varrho_k$ is uniformly bounded away from zero in $U \times [0,T]$,   \eqref{main} is uniformly parabolic. In view of the standard parabolic theory, it follows that $\varrho_k$ is smooth in $U\times (0,T]$. The proof for locally uniform convergence of $\varrho_k$ to $\varrho$ is parallel to that of Lemma 9.5 in \cite{book}.
\end{proof}

To end this section, we state the following technical lemma which is used for comparison.

\begin{lemma}\label{lem comp supersl}
Set $U:=B_1$ or $\mathbb{R}^d$. Let $\psi$ be a non-negative continuous function defined in $U \times [0,T]$ such that
\begin{itemize}
    \item[(a)] $\psi$ is smooth in its positive set and in the set {we have} $\psi_t-\Delta \psi^m-\nabla\cdot(\vec{b}\,\psi)\geq 0$,
    \item[(b)] $\psi^\alpha$ is Lipschitz continuous for some $\alpha\in (0,m)$,
    \item[(c)] $\Gamma(\psi)$ has Hausdorff dimension $d-1$.
\end{itemize}
Then 
\[\psi_t-\Delta \psi^m-\nabla\cdot(\vec{b}\,\psi)\geq 0 \text{ in }U\times [0,T]\]
in the weak sense i.e. for all non-negative  $\phi\in C_c^\infty(U\times [0,T))$
\begin{equation}
    \label{wk frm supersl}
    \int_0^T\int_{\mathbb{R}^d} \psi\,\phi_tdxdt \leq \int_{\mathbb{R}^d} \psi(0,x)\phi(0,x)dx+\int_0^T\int_{\mathbb{R}^d}(\nabla \psi^m+\psi\,\vec{b})\nabla\phi\; dxdt.
\end{equation}
\end{lemma}
We postpone the proof to the appendix.

\section{Regularity of the pressure}

{In this section we establish two basic properties for the pressure variable $u$ that we will frequently use in the rest of the paper.} We begin by obtaining the fundamental estimate.

\begin{theorem}\label{fundamental estimate}
Let $u$ be a solution of \eqref{premain} in $Q=\R^d\times [0,\infty)$ with non-negative initial data $u_0$ such that $u_0^{\frac{1}{m-1}}\in L^1(\mathbb{R}^d)\cap L^\infty(\mathbb{R}^d)$. Then there exists a universal constant $\sigma$ such that 
\begin{equation}\label{ineq fund}
    \Delta u>-\frac{\sigma}{\tau} - \sigma \quad \hbox{ in } \mathbb{R}^d\times [\tau,\infty)
\end{equation} 
in the sense of distribution. %Here $\sigma$ only depends on $d,\|\vec{b}\|_\infty,\|\varrho_0\|_1,\|\varrho_0\|_\infty$
%If assuming $\Delta u_0>-\infty$, the estimate holds for some constant $C$ for all $t\geq 0$.
\end{theorem}

\begin{proof}

By Lemma \ref{approximation}, it is enough to consider positive smooth solutions with positive smooth initial data.  If \eqref{ineq fund} holds for the approximated smooth solutions, from the locally uniform convergence of the approximation we can conclude.

\medskip

Assume that $u$ is positive and smooth, and consider $p:=\Delta u$. By differentiating \eqref{premain} twice, we get
\begin{align*}
    p_t&=(m-1)u \Delta p+2m\nabla u\cdot\nabla p+(m-1)p^2+2\Sigma\, u_{ij}u_{ij}\\
    &+\nabla p\cdot{\vec{b}}+2\Sigma\, u_{ij}{b^i_j}+\nabla u\cdot \Delta{\vec{b}}+(m-1)\left(p\nabla\cdot{\vec{b}}+2\nabla u\cdot\nabla(\nabla\cdot\vec{b})+u\Delta(\nabla\cdot {\vec{b}})\right).
\end{align*}

By Young's inequality,
\begin{align*}
    \left|(m-1)p\nabla\cdot{\vec{b}}+2\Sigma\,u_{ij}{b}^i_j\right|&\leq \frac{m-1}{2}p^2+\Sigma |u_{ij}|^2+\sigma m \\
    &\leq  {\left(\frac{m-1}{2}-\frac{1}{d}\right)}p^2+2\Sigma\, |u_{ij}|^2+\sigma m;\\
    \left|\nabla u\cdot \Delta{\vec{b}}+2(m-1)(\nabla u\cdot\nabla(\nabla \cdot{\vec{b}})\right|&\leq {m}|\nabla u|^2+\sigma m;\\
    (m-1)\left(u\Delta(\nabla\cdot {\vec{b}})\right)
    &\leq \sigma m.
\end{align*}
Thus we obtain 
\[p_t-(m-1)u \Delta p-2m\nabla u\cdot \nabla p-\left(\frac{m-1}{2}+\frac{1}{d}\right)p^2-\nabla p\cdot{\vec{b}}+{m}|\nabla u|^2+\sigma m\geq 0.\]
 Viewing $u$ as a known function, we may write the above quasilinear parabolic operator of $p$ as $\mathcal{L}_0(p)$ and so we have $\mathcal{L}_0(p)\geq 0$. Below will construct a barrier for this operator to obtain a lower bound for $p$.

Suppose that $\Delta u(\cdot,0)\geq -\frac{1}{\tau}$ for some $\tau>0$. {By Theorem 1.7 \cite{alt1983quasilinear}, $u$ is uniformly bounded by a universal constant and we denote it as $\sigma_0$.} Let $w:=-\frac{\sigma_1}{t+\tau}+u-\sigma_2$ for some $\sigma_1\geq 1,\sigma_2\geq \sigma_0$ to be determined later. Then $p\geq w$ at $t=0$. 

Direct computation yields
\[\mathcal{L}_0(w)=\frac{\sigma_1}{(t+\tau)^2}+u_t-(m-1)u \Delta u-2m |\nabla u|^2-\left(\frac{m-1}{2}+\frac{1}{d}\right)\l-\frac{\sigma_1}{t+\tau}+u-\sigma_2\r^2\]\[-\nabla u\cdot{\vec{b}}+{m}|\nabla u|^2+\sigma m.\]
Now we use the equation \eqref{premain} to obtain
\begin{align*}
    \mathcal{L}_0(w)&\leq \frac{\sigma_1}{(t+\tau)^2}-\left(m-1\right) |\nabla u|^2-\left(\frac{m-1}{2}+\frac{1}{d}\right)\l-\frac{\sigma_1}{t+\tau}+u-\sigma_2\r^2+\sigma m\\
    &\leq \frac{\sigma_1}{(t+\tau)^2}-\left(\frac{m-1}{2}+\frac{1}{d}\right)\frac{\sigma_1^2}{(t+\tau)^2}-\left(\frac{m-1}{2}+\frac{1}{d}\right)(\sigma_2-u)^2+\sigma m\\
    &\leq 0,
\end{align*}
where the last inequality holds if we choose $\sigma_1:= d$ and $\sigma_2:=\sigma_0+(4d\sigma)^{1/2}$. Hence $\mathcal{L}_0(w)\leq 0\leq \mathcal{L}_0(p)$, and from the comparison principle for $\mathcal{L}_0$ we conclude that 
\[\Delta u=p\geq w \geq -\frac{\sigma_1}{t+\tau}-\sigma_2.\]
After taking $\tau\to 0$, we obtain that \eqref{ineq fund} holds for smooth solutions. We can conclude by Lemma \ref{approximation}.

\end{proof}

\begin{remark}
Using the same barrier in the proof of the lemma, it can be seen that if $\Delta u_0\geq -C_0$ in the sense of distribution, then $\Delta u\geq -\frac{\sigma_1}{t+(1/C_0)}-\sigma_2$ in the distribution sense for all time.
\end{remark}

\medskip

Next we prove a useful property: the consistency of positivity set of a solution along streamlines over time. The proof is parallel to the proof of Lemma 3.5 \cite{kim2017singular} where they used a barrier argument. Recall that we denote ${\Omega}_{t}=\{u(\cdot,t)>0\}$.
%The lemma is stated as a local estimate.

%We will consider for $t>0$, since at $t=0$, $u$ may not even be continuous. Again we are using smooth approximations.

\medskip

\begin{lemma}\label{streamline}
%If $\vec{b}\in C^{3,0}_{x,t}$, then 
Let $u$ solves \eqref{premain} with $\Delta u>-\infty$ in $Q_2$. Then for $X(x,t;s)$ given in \eqref{ode} and for $c_0:=\frac{1}{2(1+\|\vec{b}\|_\infty)}$ the following is true. 
\[\left(X({\Omega_t},t;s)\cap B_1\right)\subseteq {\Omega_{t+s}} \hbox{ for all } t\in (-1,1-c_0) \hbox{ and } s\in (0,c_0].
\]

If $u$ solves \eqref{premain} in $Q=\R^d\times [0,\infty)$ with initial data $u_0$ given as in Theorem~\ref{fundamental estimate}, then 
\[X({\Omega_t},t;s)\subseteq {\Omega_{t+s}} \quad \hbox{ for all } s,t>0. 
\]
\end{lemma}

\begin{proof}
In view of Theorem \ref{fundamental estimate}, the second statement follows easily from the first one. To prove the first statement, 
it is suffices to show that for all $x\in\Omega_t$ and $s\in (0,c_0]$, if $X(x,t;s)\in B_1$ then $u(X(x,t;s))>0$.

If $x\in B_{\frac{3}{2}}^c$, by the choice of $c_0$, 
\[|X(x,t;s)|\geq |x|-\|\vec{b}\|_\infty s> 1\quad \text{ for all }s\in (0,c_0].\]
Thus we take $x\in \Omega_t\cap B_{\frac{3}{2}}$ and then $X(x,t;s)$ is inside the domain $B_2$ for all $s\in (0, c_0]$.  By Theorem \ref{holder}, $u$ is continuous in $Q_2$. Then we can suppose for contradiction that there exists $s_0\in (0, c_0]$ such that
\[u(X(x,t;s),t+s)>0\text{ for all $s\in (0,s_0)$ } \quad \text{ and }\quad u(X(x,t;s_0),t+s_0)=0.\]

Suppose $\Delta u\geq -C_0$ in $Q_2$. Note that \eqref{premain} is uniformly parabolic in any compact subset of $\{u>0\}$, due to the continuity of $u$. Therefore by the standard parabolic theory, $u$ is smooth in $\Omega
\cap Q_2$. It follows from \eqref{premain} that for all $s\in(0,s_0)$,
\begin{align*}
   \partial_s u(X(x,t;s),t+s)&= (u_t +\nabla u\cdot \vec{b})(X(x,t;s),t+s)\\
    &\geq (-C_0(m-1)u+|\nabla u|^2+(m-1)u\nabla\cdot \vec{b})(X(x,t;s),t+s)\\
    &\geq -C u(X(x,t;s),t+s)
\end{align*}
where $C:=(m-1)(C_0+\|\nabla\cdot \vec{b}\|_\infty)$.
This yields
\begin{equation}\label{mono stream}
    u(X(x,t;s),t+s)\geq e^{-Cs}u(x,t)>0,
\end{equation}
which, after taking $s\to s_0<1$, contradicts with the assumption that $u(X(x,t;s_0),t+s_0)=0$.

\end{proof}

\section{Regularity of the Free Boundary}

%In this section, we only consider time independent vector fields. 

    %& \Delta u\geq -C_0,\; \| \vec{b}\|_{C^{3,1}_{x,t}}\leq L \text{ in the range of consideration.}

In this section we study {finer properties on expansion of the positive set} $\{u>0\}$ along the streamlines associated with the drift $\vec{b}$.
\begin{comment}
 The central idea in \cite{CFregularity} was to measure the time the free boundary moves away from a given point by distance $R$, in terms of the average pressure in a ball of size $R$, making it sufficient to track the size of the pressure average over time instead of the free boundary movement.
 \end{comment}
  We largely follow the ideas in \cite{CFregularity} applied to the zero drift case, and obtain corresponding statements (Lemma 4.1 and 4.2) for our problem.

\medskip

\begin{lemma}\label{lem part1}
Let $u$ be given as in Theorem \ref{fundamental estimate}. For any $t_0\geq \eta_0>0$  there {exist} $\tau_0,c_0$ depending only on $\eta_0$ and universal constants 
such that the following holds. For any $R>0$ and $\tau\in (0,{\tau_0})$, if 
\begin{equation}\label{if}
u(\cdot,t_0)=0 \hbox{ in } B(x_0,R)\quad \hbox{ and }\quad\oint_{B(X(x_0,t_0;\tau),R)}u(x,t_0+\tau)dx\leq \frac{c_0R^2}{\tau},
\end{equation}
then
\begin{equation}\label{then}
u(x,t_0+\tau)=0 \quad\text{ for }x\in B(X(x_0,t_0;\tau),R/6). 
\end{equation}
\end{lemma}
\begin{proof}

For simplicity, suppose $x_0=0,t_0=0$,  and consider the rescaled function
\begin{equation}\label{rescale v}
    \tilde{u}(x,t):=\frac{\tau}{R^2}u(Rx,\tau t)\quad \hbox{ with } \vec{b}'(x,t):=\frac{\tau}{R} \vec{b}(Rx,\tau t), \, \tilde{X}(t):=\frac{1}{R}X(0,0;\tau t).
\end{equation} 
Then $\tilde{u}$ satisfies
\[\tilde{u}_t=(m-1)\tilde{u}\Delta \tilde{u}+|\nabla \tilde{u}|^2+\nabla \tilde{u}\cdot \vec{b}'+(m-1)\tilde{u}\, \nabla \vec{b}'.\]

\medskip

 Theorem~\ref{fundamental estimate} yields  
 \begin{equation}\label{bound00}
 \Delta u\geq -C_0 = -C_0(\eta_0) \hbox{ for }t\geq \eta_0.
 \end{equation}  Set $\eps:=C_0\tau_0$ so that  $\Delta \tilde{u} = \tau \Delta u \geq -\eps$.  From our assumption, it follows that \[\oint_{B(\tilde{X}(1),1)}\tilde{u}(x,1)dx\leq c_0.\]
 Using this and that $\tilde{u}+\eps |x|^2/(2d)$ is subharmonic, we find for $x\in B(\tilde{X}(1),\frac{1}{2}),$
\begin{equation}\label{est v t1}
\begin{aligned}
\tilde{u}(x,1)&\leq -\frac{\eps|x|^2}{2d}+\oint_{B(\tilde{X}(1),\frac{1}{2})}\tilde{u}(y,1)+\frac{\eps|y|^2}{2d}dy\\
&\leq 2^d\oint_{B_1}\tilde{u}(y,1)dy+\sigma\eps 
\leq 2^dc_0+\sigma\eps.
\end{aligned}
\end{equation}

Now consider 
\[v(x,t):=\tilde{u}(x+\tilde{X}(t),t).\]
Then $\Delta v \geq -\eps$.  
Moreover, observe that $v$  is the weak solution of 
$$
\calL(v):= {v}_t - (m-1){v}\Delta {v}- |\nabla {v}|^2-\nabla{v}\cdot (\vec{b}'(x+\tilde{X},t)-\vec{b}'(\tilde{X},t))-(m-1){v}\nabla\cdot \vec{b}'(x+\tilde{X},t)=0.
$$
We used Definition \ref{def1.1} as the notion of weak solutions, where $\vec{b}$ is replaced by $\vec{b}'(x+\tilde{X},t)-\vec{b}'(\tilde{X},t)$. Since the operator $\mathcal{L}$ is locally uniformly parabolic  in its positive set, $v$ is smooth in the set due to the standard parabolic theory.  From the above equation, $v$ satisfies the following in the classical sense in its positive set \begin{equation*}
   % \label{est tilde v prime}
\begin{aligned}
    v_t(x,t)
    &\geq -\eps (m-1){v}+|\nabla {v}|^2-\sigma\tau |\nabla{v}||x| - \sigma\tau {v}\\
    &\geq -\eps (m-1){v}-\sigma\tau{v}-\sigma \tau^2|x|^2,
\end{aligned}
\end{equation*}
where the first inequality is due to the fact that $|\nabla\, \vec{b}'|\leq \tau\sigma$ and the second inequality follows from Young's inequality. Because $v$ is continuous and non-negative, the above estimate also holds weakly in the whole domain.

 Since $\e = C_0\tau\geq \tau$,  we obtain
\begin{equation}
    \label{est tilde v prime}
    {v}_t(x,t)\geq -\sigma \eps\, {v}(x,t) -\sigma \eps^2|x|^2,  %\text{ with }C_L=(m-1)(1+L) .
\end{equation}
and thus{ by Gronwall}
$$
    {v}(x,1)\geq e^{\sigma\eps(t-1)}{v}(x,t)-\sigma(1-e^{\sigma\eps(t-1)})\eps\,|x|^2\geq e^{-\sigma\eps}{v}(x,t)-\sigma \eps \hbox { in } B_{\frac{1}{2}}\times (0,1).
    $$

Using \eqref{est v t1}, we conclude that for all $(x,t)\in B_{\frac{1}{2}} \times (0,1)$ and some $\sigma\geq 1$,
\begin{equation}\label{est v 111}
\begin{aligned}
{v}(x,t)&\leq e^{\sigma\eps}v(x,1)+e^{\sigma\eps}\sigma\eps=e^{\sigma\eps}\tilde{u}(x+\tilde{X}(t),1)+e^{\sigma\eps}\sigma\eps\\
&\leq e^{\sigma\eps}(2^d c_0+2\sigma\eps)\leq \sigma (c_0+\eps).
\end{aligned}
\end{equation}
if $\e$ is sufficiently small.

\medskip

To conclude we proceed with a barrier argument applied to the operator $\mathcal{L}$. Define
\[\varphi(x,t):=\lambda\left(\frac{t}{36}+\frac{(|x|-1/3)}{6}\right)_+\]
and we aim at showing $\calL(\varphi)\geq 0$ weakly.
Using Lemma \ref{lem comp supersl} to $(\frac{m-1}{m}\varphi)^{\frac{1}{m-1}}$, the corresponding density variable of $\varphi$, and the Lipschitz continuity of $\varphi$, we find that to show $\varphi$ is a supersolution of $\calL$, it suffices to prove $\calL(\varphi)\geq 0$ in the positive set of $\varphi$. 

Notice
\[\nabla{\varphi}\cdot (\vec{b}'(x+\tilde{X},t)+\vec{b}'(\tilde{X},t))-(m-1){\varphi}\nabla\cdot \vec{b}'(x+\tilde{X},t)\leq    \sigma\eps|\nabla {\varphi}|\,  |x| + \sigma \eps {\varphi},\]
so direct computations yield that if
\begin{equation}
    \frac{1}{\lambda}\geq \left(\frac{t}{6}+|x|-\frac{1}{3}\right)\left((m-1)(d-1)|x|^{-1}+\frac{\sigma\eps}{\lambda}\right)+1+\frac{\sigma\eps}{\lambda^2},\label{4.1}
\end{equation}
then $\mathcal{L}(\varphi)\geq 0$ for $\frac{1}{3}-\frac{t}{6}< |x|< \frac{1}{2}$ in the classical sense.
The inequality \eqref{4.1} is valid for $t\in (0,1)$ provided that we take $0<\eps <<\lambda<<1$.
With this choice of $\eps,\lambda$, we get $\mathcal{L}(\varphi)\geq 0$ in $|x|< \frac{1}{2}$ weakly.
By the assumption ${v}(x,0) = 0$ in $B_{\frac{1}{2}}$ and thus ${v}\leq \varphi$ on $|x|\leq \frac{1}{2},t=0$. 
On the lateral boundary $|x|=\frac{1}{2}, t\in(0,1)$, by \eqref{est v 111} if $c_0,\eps$ are small enough depending on universal constants we have
\[{v}\leq \sigma(c_0+\eps)\leq \frac{\lambda}{36}\leq \varphi. \]
%Then from here, $c_0,\tau_0$ are selected such that \[c_0\lesssim_{m,d}1, \quad \tau_0\lesssim_{m,d} \frac{1}{C_0(1+L)}.\]
Hence {by comparison principle for the operator $\mathcal{L}$} (see Remark \ref{remark 2.0}) in $B_{\frac{1}{2}} \times (0,1)$  we have
${v}\leq \varphi.$
In particular
\[\tilde{u}(x+\tilde{X}(1),1)={v}(x,1)\leq \varphi(x,1)=0\]
for $|x|<\frac{1}{6}$ and we proved the lemma. %\textcolor{with $\tau_0,c_0$ only depending on $m,d,C_0, L$.}

\end{proof}

\begin{remark}\label{remark part1}
One can check that the conclusion of the lemma also holds in a local setting: If $u$ solves \eqref{premain} with $\Delta u\geq-C_0$ in $Q_1$ for some $C_0$. Then there exist $\tau_0,c_0,\sigma$ such that \eqref{if} implies \eqref{then} for any $R\in (0,\sigma)$ and $\tau\in (0,{\tau_0})$. Here $\tau_0,c_0$ depend only on $C_0$ and universal constants, and $\sigma$ is universal. This local version of the lemma will be used in Lemma \ref{condlem nondeg}.
\end{remark}

\begin{lemma}\label{lem part2 conv}
Let $u$ be as in Theorem \ref{fundamental estimate}. For any $t_0\geq \eta_0>0$ and $c_1>0$, there exist $\lambda, c_2,\tau_0>0$ depending on $c_1$, $\eta_0$ and universal constants such that the following holds. For any $R>0$ and $0<\tau \leq \tau_0$, if 
\begin{equation}
\oint_{B(x_0,R)}u(x,t_0)dx\geq c_1\frac{R^2}{\tau},
\end{equation}
then
\begin{equation}
(X(x_0,t_0;\lambda \tau),t_0+\lambda\tau)\geq c_2\frac{R^2}{\tau}.
\end{equation}
\end{lemma}

\begin{proof}

%\[\eps=C_0\tau_0,\quad c_1\geq C\eps^\gamma+Ce^{C(L+1)\eps\lambda}\lambda c_2^\frac{m}{m-1}+CL\eps\lambda^2,\quad \eps\lambda<<1/(1+L).\]

 Let $C_0$ be as in \eqref{bound00}, and set $(x_0,t_0)=(0,0)$ by shifting coordinates. We consider the corresponding density variable $\varrho(x,t):=(\frac{m-1}{m}u(x, t))^{\frac{1}{m-1}}$ and its rescaled version
\[\tilde{\varrho}(x,t):=(\frac{\tau}{R^2})^{\frac{1}{m-1}}\varrho(Rx,\tau t).\]
Let $\vec{b}'$, $\tilde{X}$ be as in \eqref{rescale v} and let $\eps=C_0\tau$ as in the proof of Lemma 4.1. Then $\tilde{\varrho}$ solves the re-scaled density equation
\[{\tilde{\varrho}}_t=\Delta {\tilde{\varrho}}^m+\nabla\cdot({\tilde{\varrho}}\,\vec{b}'). \]
The fundamental estimate on $u$ implies that $\Delta \tilde{\varrho}^m\geq -\eps\tilde{\varrho}$ in the sense of distribution.

 Let us define ${\xi}(x,t):=\tilde{\varrho}(x+\tilde{X},t)$ and $Y(t) := \int_{B_1} {\xi}^m(x,t) dx$. {Below we study properties on the growth rate of $Y$ using properties of $\tilde{\varrho}$,  namely we derive \eqref{Y lower} and \eqref{est Y}. We then use these estimates to argue by a contradiction to prove our main statement. }

\medskip

First let us show that $Y(\lambda)$ stays sufficiently positive if {$\eps\lambda$} is small.  Since $\tilde{X}(0)=0$,  our assumption yields that
\begin{align*}
    Y(0)&=\oint_{{{B_1}}}{\xi}^m(x,0)dx=\sigma(\frac{\tau}{R^2})^{\frac{m}{m-1}}\oint_{B(0,R)}\varrho^m(x+\tilde{X}(0),0)dx\\
    &=\sigma\oint_{B(0,R)} (\frac{\tau}{R^2}u)^{\frac{m}{m-1}}(x,0)dx\\
    &\geq \sigma\left(\frac{\tau}{R^2} \oint_{B(0,R)}u(x,0)dx\right)^{\frac{m}{m-1}}\geq \sigma c_1^\frac{m}{m-1}=:c_1'.
\end{align*}

Due to \eqref{est tilde v prime} and ${v}(x,t)=\frac{m}{m-1}{\xi}^{m-1}(x,t)$, for $\eps$ small enough
\begin{equation}
    \label{est v tilde sq prime}
\begin{aligned}
    ({\xi}^m)_t&\geq -\sigma\eps {\xi}^m-\sigma \eps^2 |x|^2{\xi}\geq -\sigma\eps{\xi}^m-\sigma \eps \hbox{ for } x\in B_1\cap\{\xi>0\}.
\end{aligned}
\end{equation}
%So
%\begin{align}\nonumber    \partial_t w^m(x,t)&\geq C_{L,m}w^m(x,t)-C_L\eps w^m(x,t)-\eps |x|^2\rho^{m-1}(x+\tilde{X},t)\\&\geq -C_L\eps w^m(x+\tilde{X},t)-C\eps.\label{est v tilde sq prime}\end{align}
Consequently
\begin{equation}
    \label{Y lower}Y(t)\geq e^{-\sigma\eps t}Y(0)-\sigma \eps t\geq e^{-\sigma\eps \lambda}c_1'-\sigma \eps \lambda>\frac{c_1'}{2}\sim c_1^{\frac{m}{m-1}}
\end{equation}
for $t\in (0,\lambda]$ if $\eps\lambda<<_\sigma 1$.

\medskip
  
Next we obtain an upper bound for the growth of $Y$ over time.

\medskip

{\bf Claim:}
For some universal constants $\sigma_1, \sigma_2$ and $\gamma$, 
\begin{equation}\label{est Y}
e^{-\sigma_1\eps t}\int_0^t Y(s) ds \leq \sigma_2(\int_0^t {\xi}^m(0,s)ds + \eps^\gamma+ Y^{\frac{1}{m}}).
\end{equation}

{\bf Proof of the Claim}.
As in  \cite{CFregularity}, we introduce the Green's function in a unit ball so that $G$ solves 
\begin{equation}
    \label{prop G}
    \Delta G= -\sigma_d\delta(x)+\sigma_d I_{{{B_1}}}\quad \text{ and }\qquad G = |\nabla G|=0 \text{ on }\partial {{B_1}}.
\end{equation}
Let us only discuss the dimension $d\geq 3$, where $G$ is defined as 
\begin{equation}\label{green}
G(x)=|x|^{2-d}-1-\frac{d-2}{2}(1-|x|^2).
\end{equation}
%For simplicity of notations, we write \[\tilde{X}=\tilde{X}(x,t),\, \tilde{X}=\tilde{X}(0,t),\,r=|\tilde{X}-\tilde{X}|(x,t)\,\text{ and } S=S(t)=B(\tilde{X},1).\]

%By direct computation we have\[\int_{B(\tilde{X},1)}G(y-\tilde{X})\rho(y,t)dy=\int_{\tilde{X}^{-1}(\cdot,t)B(\tilde{X},1)}G(\tilde{X}-\tilde{X}(0,t)){\xi}\left|\frac{\partial \tilde{X}}{\partial x}\right|dx.\]

We want to differentiate $\int_{B(\tilde{X},1)}G(x-\tilde{X})\tilde{\varrho}(x,t)dx$ with respect to $t$. Since 
$G(x-\tilde{X})=0$
on $\partial B(\tilde{X},1)$, %and
%, 

\begin{equation}
    \begin{aligned}
    &\left(\int_{B(\tilde{X},1)}G(x-\tilde{X})\tilde{\varrho}(x,t)dy\right)' =
    \int_{B(\tilde{X},1)}\nabla G(x-\tilde{X}) \cdot\vec{b}'(\tilde{X}){\tilde{\varrho}}\,dx
    +\int_{B(\tilde{X},1)}G(x-\tilde{X})\,{\tilde{\varrho}}_t\,dx\\
    &=\int_{B(\tilde{X},1)}\nabla G(x-\tilde{X})\cdot (\vec{b}'(\tilde{X})-\vec{b}'(x)){\tilde{\varrho}}\,dx
    +\int_{B(\tilde{X},1)}\Delta G(x-\tilde{X})\,{\tilde{\varrho}}^m\,dx=:A_1+A_2. \label{G prime}
    \end{aligned}
\end{equation}
Since ${\nabla}\vec{b}'  \geq -\sigma\eps  I_d$,  
\begin{equation} \label{diffence Vx}
\begin{aligned}
    A_1&= -\int_{B(\tilde{X},1)}(d-2)(|x-\tilde{X}|^{-d}-1)(x-\tilde{X})\cdot(\vec{b}'(\tilde{X})-\vec{b}'(x))\tilde{\varrho}\,dx\\
    &\geq -\sigma\eps\int_{B(\tilde{X},1)}(d-2)(|x-\tilde{X}|^{-d}-1)|x-\tilde{X}|^2{\tilde{\varrho}}\,dx\\
    &\geq-\sigma \eps\int_{B(\tilde{X},1)}G(x-\tilde{X}) {\tilde{\varrho}}\,dx.
\end{aligned}
\end{equation}
As for $A_2$, applying \eqref{prop G}, we obtain \begin{align}
A_2= -\sigma_d\, {\tilde{\varrho}}^m(\tilde{X},t)+\sigma \int_{B(\tilde{X},1)}{\tilde{\varrho}}^m(x,t)\,dx.
\label{G prime 2}
\end{align}

Using \eqref{diffence Vx}, \eqref{G prime 2}, we find for some universal $\sigma>0$
\begin{align*}
    \left(\int_{B(\tilde{X},t)}G(x-\tilde{X})\tilde{\varrho}(x,t)dy\right)'    
    &\geq -\sigma_d\, {\tilde{\varrho}}^m(\tilde{X},t)+\sigma \int_{B(\tilde{X},t)}{\tilde{\varrho}}^m(x,t)\,dx\\
    &\quad -\sigma \eps\int_{B(\tilde{X},t)}G(x-\tilde{X})\,{\tilde{\varrho}}(x,t) dx.
\end{align*}
Hence we derive
\begin{align*}
    e^{\sigma \eps t}\int_{{{B_1}}}G(|x|){\xi}(x,t)dx\geq -\sigma_d\int_0^t e^{\sigma \eps s}{\xi}^m(0,s)ds+\sigma \int_0^t\int_{{{B_1}}}e^{\sigma \eps s}{\xi}^m(x,s)\,dxds,
\end{align*}
which simplifies to
\begin{align}\label{ineq Y G}
  \int_0^t e^{-\sigma \eps t}Y(s)ds\leq  \sigma \int_{{{B_1}}}G(|x|){\xi}(x,t)dx+\sigma \int_0^t {\xi}^m(0,s)ds.
\end{align}

Now following the proof of Lemma 2.3 \cite{CFregularity}, using \eqref{ineq Y G} and the integrability property of $G$, we can obtain the upper bound $\int_{{{B_1}}}G{\xi}\,dx$ to conclude. We omit the computation since it is parallel to \cite{CFregularity}.

\hfill$\Box$.

\medskip

  Going back to the proof of Lemma~\ref{lem part2 conv}, let us suppose that our statement is false, which means $u(X(\lambda\tau),\lambda\tau)< c_2\frac{R^2}{\tau}$ for any choice of $\lambda,c_2,\tau_0$, where $X(t):=X(0,0;t)$. Later we will pick the constants satisfying
  \[\lambda>>1,\quad c_2^{\frac{m}{m-1}}\lambda<<1,\quad \eps\lambda<<1.\]
 In terms of ${\xi}=\tilde{\varrho}(\cdot+\tilde{X},\cdot)$, we have
\[{\xi}^m(0,\lambda)\leq \sigma(m)\, c_2^\frac{m}{m-1}.\]

Since $\eps\lambda<<1$, by \eqref{est v tilde sq prime} again, we obtain 
\[{\xi}^m(0,t)\leq \sigma e^{\sigma\eps \lambda}c_2^\frac{m}{m-1}+\sigma \eps \lambda \hbox{ for } t\in (0,\lambda].
\]
If follows from \eqref{est Y} that for all $t\in (0,\lambda]$ and some $\sigma=\sigma(\sigma_2)$,
\[
e^{-\sigma_1\eps t}\int_0^t Y(s) ds \leq \sigma( e^{\sigma\eps\lambda}c_2^\frac{m}{m-1}\lambda+\eps\lambda^2+ \eps^\gamma+ Y^{\frac{1}{m}}).
\]
Recall \eqref{Y lower}, and we have
\begin{equation}
    \label{cond c2}
\sigma Y^{\frac{1}{m}}\geq\sigma   c_1^\frac{1}{m-1}\geq \sigma(e^{\sigma\eps\lambda}c_2^\frac{m}{m-1}\lambda+\eps\lambda^2+\eps^\gamma).
\end{equation} 
Hence we get for $t\in (0,\lambda]$ and some universal $\sigma>0$,
\[\sigma Y^\frac{1}{m}\geq e^{-\sigma_1\eps t}\int^t_0 Yds.\]

Writing $Z(t):=\int_0^t Y(s) ds$, in view of \eqref{Y lower} we obtain $Z(\frac{\lambda}{2})\geq c_3\lambda$ with
\[c_3:=\frac{1}{2}(e^{-\sigma\eps\lambda}c_1'-\eps\lambda)\geq \sigma c_1^\frac{m}{m-1}>0.\]
Solving the ODE problem
\[\sigma Z'\geq  e^{-\sigma\eps t}Z^m,\quad \text{ with }Z(\frac{\lambda}{2})\geq c_3\lambda\] 
shows that 
\begin{equation}
    \label{est Z}Z(t+\frac{\lambda}{2})\geq \left({(c_3\lambda)^{1-m}-f(t)}\right)^{\frac{1}{1-m}},\quad \text{ for $t\in (0,\frac{\lambda}{2}]$ }
\end{equation}
where 
\[f(t):=\int_{\lambda/2}^{t+\lambda/2}\sigma e^{-\sigma\eps s}ds=\sigma e^{-\sigma\lambda\eps/2}\frac{(e^{\sigma\eps t}-1)}{\sigma\eps}.\]
Since $\sigma \eps << 1$,
\[f(t)\geq \sigma t-\sigma \eps t^2.\]

It is obvious that $f$ is monotone increasing in $t$. Notice the right-hand side of \eqref{est Z} goes to $+\infty$ as 
\[t\to f^{-1}((c_3\lambda)^{1-m})
\] 
which is impossible provided that $f^{-1}((c_3\lambda)^{1-m})\leq \frac{\lambda}{2}$. However if $ \lambda\geq C (c_3,\sigma)$ and $ \eps\lambda <<1$, we indeed have
\begin{equation}
    \label{cond lambda}
   f(\frac{\lambda}{2})\geq \sigma \frac{\lambda}{2}-\sigma\eps \frac{\lambda^2}{4}\geq(c_3\lambda)^{1-m},
    \end{equation}
which leads to a contradiction.

\medskip

We proved that ${\xi}^m(0,\lambda)\leq \sigma(m)\, c_2^\frac{m}{m-1}.$  Since $c_3$ only depends on $c_1,\sigma$, the choices of $\lambda,c_2,\eps$ only depend on $c_1,\sigma$.
We conclude the lemma with $\tau_0£º=\eps/C_0$, $\lambda$ satisfying \eqref{cond lambda}, and $c_2,\eps$ satisfying \eqref{cond c2} and $\eps\lambda<<1$.

\end{proof}

For any $(x_0,t_0)\in\Gamma$, we use the notation
\[{\Upsilon}(x_0,t_0):=\left\{(X(x_0,t_0;-s),t_0-s), \quad s\in (0,t_0)\right\}.\]

\begin{theorem}\label{strictly expanding}
For a given point $(x_0,t_0)\in \Gamma$ with $t_0\geq \eta_0>0$, the following is true:
\begin{flushleft}
 (1) Either (a) ${\Upsilon}(x_0,t_0)\subset \Gamma$ or (b) ${\Upsilon}(x_0,t_0)\cap \Gamma=\emptyset$.\\
(2) If (b) holds,  then there exist positive constants $C_*,\beta,h$ such that for all $s\in (0,h)$
\begin{align*}
    &\varrho(x,t_0-s)=0 \quad\text{ if }\quad |x-X(x_0,t_0;-s)|\leq C_* s^\beta ;\\
    &\varrho(x,t_0+s)>0 \quad\text{ if }\quad |x-X(x_0,t_0;s)|\leq C_*s^\beta.
\end{align*}
\end{flushleft}
Here $\beta$ only depends on $\eta_0$ and universal constants. 
If (b) holds for $(x_0,t_0)\in\Gamma$, we say $(x_0,t_0)$ is ``of the second type" free boundary point.

\end{theorem}

%\begin{definition}\label{strictly expand}
%We call the second case as: the support of the solution is \textit{strictly expanding relatively to the streamlines} at $(x_0,t_0)$. In terms of $X(x_0,t_0;s)$, we have
%\begin{align*}
%&X(x_0,t_0;s)\in \Omega^c(t) &\text{ for }&s<0,\\
%&X(x_0,t_0;s)\in \Omega(t) &\text{ for }&s>0.
%\end{align*}
%\end{definition}

{\bf Sketch of the proof}

The proof is parallel to those for Theorems 3.1-3.2 \cite{CFregularity}, based on the Lemmas \ref{lem part1} and \ref{lem part2 conv}.
Let us only sketch the proof for part (1) below. 

If the assertion of (1) is not true, then we can find $ t_0>t_1>t_2>0$ such that $t_0-t_1>>t_1-t_2$ and
\[x_0\in\Gamma_{t_0},\quad x_1:=X(x_0,t_0;t_1-t_0)\in\Gamma_{t_1},\quad x_2:=X(x_0,t_0;t_2-t_0)\notin\Gamma_{t_2}.\]
Consequently  $u(\cdot,t_2)=0$ in $B(x_2,R)$ for some $R>0$. Since $x_1=X(x_2,t_2;t_1-t_2)$, by Lemma \ref{lem part1}, 
\[\oint_{B(x_1,R)}u(x,t_1)dx\geq \frac{c_0R^2}{t_1-t_2}.\]
Since $t_0-t_1>> (t_1-t_2)$,  Lemma \ref{lem part2 conv} yields
 $u(x_0,t_0)=u(X(x_1,t_1;t_0-t_1),t_0)>0$, which is a contradiction.

\hfill$\Box$

\medskip

 When the initial data grows faster than quadratically near its free boundary, it is possible to characterize the constants $C_*, h$ in above theorem in terms of time variable. By a compactness argument, iteratively using Theorem \ref{strictly expanding} and arguing as in the remark on Theorem 3.2 in \cite{CFregularity}, we have the following theorem.

\begin{theorem}\label{strictly expanding initial condition}
{Suppose \eqref{intro less quar grow}.} Then any point $x_0\in\Gamma_{t_0}$ with $ t_0\leq T$ is of the second type and the constants $C_*,h$ in Theorem \ref{strictly expanding} (2) only depend on $\gamma,\varsigma,t_0$ and universal constants.
\end{theorem}

\section{Monotonicity Implies Non-degeneracy}\label{sec 5}

 In this section we discuss non-degeneracy property of solutions in local settings.  %Note that continuity of weak solutions are guaranteed by Theorem \ref{holder}.  
 We start with the following theorem.

\begin{theorem}\label{prop nondeg}
Let $u$ solve \eqref{premain} in $Q_2$ with $\Delta u\geq -C_0$. Suppose that $\Gamma$ is of type two in $Q_2$,  and that  
\begin{equation}
    \label{e.5.1}
    u\text{ is monotone with respect to $W_{\theta, \mu}$ in $Q_2$ for some $\theta\in (0,\pi/2)$ and $\mu\in \mathcal{S}^{d-1}$}. 
\end{equation}
Then there exist constants $C,\eps_0>0$ such that we have %for all $(x,t)\in Q_1\cap\Gamma$ and $\eps$ small enough 
\begin{equation*}\label{thm 5.1}
    u(X(x,t;C\eps)-\eps{\mu},{t+}C\eps)>0 \text{ for } (x,t) \in \Gamma\cap Q_1 \hbox{ and for } \e<\eps_0.
\end{equation*}
\end{theorem}

\begin{remark}
The constants $C,\eps_0$ in Theorem \ref{prop nondeg} only depend on 
\begin{equation}
    \label{C.5.1}
    C_0, \theta,C_*,h,\beta,\text{ and universal constants},
\end{equation}
where  $C_*,h,\beta$ are constants given in Theorem \ref{strictly expanding}. 
In the global setting, an estimate of $C_0$ can be found in Theorem~\ref{fundamental estimate}.

Let us also mention that  Theorem \ref{holder} allows us to consider continuous local solutions. 

\end{remark}

\medskip

%\textcolor{it is not really outline so I removed that sentence}

%\medskip

{The central ingredient of the proof is a barrier argument motivated from \cite{choijerisonkim} in the context of Hele-Shaw flow. The barrier argument illustrates the fact that in diffusive free boundary problems the nice regularity properties of $u$ propagate from positive level sets to the free boundary as the positive set expands out. This argument in our setting corresponds to the proof of \eqref{claim}. Compared to the Hele-Shaw flow which is driven by a harmonic function, our solutions features a nonlinear diffusion that degenerates near the free boundary and thus it requires more careful arguments. On the other hand, we will benefit from the weak formulation of the problem using the density formula (see $\mathcal{G}$ below.)}

\medskip

For $u$ as given above we consider 
\begin{equation}\label{v}
v{(x,t)}:=u(x+X(t),t),\quad \hbox{ where } X(t):=X(0,0;t) \hbox{ is given in } \eqref{ode}.
\end{equation} Then $v$ is a weak solution of  $\mathcal{L}_2 (\cdot) = 0$, where the operator $\mathcal{L}_2$ is given by 
\begin{equation}\label{cal L 2}
\begin{aligned}
      \mathcal{L}_2 ({f})&:=\partial_t{f} -(m-1){f}\Delta {f}-|\nabla {f}|^2-\nabla {f}\cdot (\vec{b}(x+X(t),t)\\
      &\quad\quad-\vec{b}(X(t),t))-(m-1){f}\nabla\cdot \vec{b}(x+X(t),t).
\end{aligned}
\end{equation}
Since the operator $\mathcal{L}_2$ is the same as in  \eqref{premain} with $\vec{b}$ replaced by $\vec{b}(x+X(t),t)-\vec{b}(X(t),t))$, the notion of sub- and supersolution is given in Definition \ref{def1.1}. 

\medskip

{Below we construct of a supersolution for the operator $\mathcal{L}_2$ for the aforementioned barrier argument, using a inf-convolution construction introduced first by \cite{caffarellipart}. }Since the supersolution to be constructed is a rescaled inf-convolution of $v$ (see \eqref{w}), comparison of the two  functions gives a space-time monotonicity of $v$, yielding the theorem.  To this end, we will use both smooth approximations of $u$ and the density version of the equation $\mathcal{L}_2$.

\medskip

We begin with some basic properties of the inf-convolution of smooth functions.

\medskip

Let $\psi, h \in C^\infty(B_2)$ with $0<\psi<\frac{1}{2}$ and $h\geq 0$. Define
\begin{equation}\label{inf_convolution}
f(x):=\inf_{B(x,\psi(x))}h(y)
\end{equation}
which is  Lipschitz continuous. 
The proofs of the following two lemmas are in the appendix.

\begin{lemma} \label{lem delta} Let $h$ and $f$ be as given in \eqref{inf_convolution}. Furthermore, suppose $\Delta h\geq -C $ for some $C\in\mathbb{R}$ and $\|\nabla\psi\|_\infty\leq 1$. Then there are dimensional constants $\sigma_1>0$ and $\sigma_2\geq 3$ such that if $\psi$ satisfies
\[\Delta\psi\geq \frac{\sigma_1|\nabla\psi|^2}{|\psi|} \quad \hbox{ in } B_2,
\]
we have 
\[\Delta f(\cdot)-(1+\sigma_2 \|\nabla\psi\|_\infty)\Delta h(y(\cdot))\leq {\sigma_2}\|\nabla\psi\|_\infty C\quad \text{ in $B_1$ in the sense of distribution},
\] 
where $y(\cdot)$ satisfies that $f(\cdot)=h(y(\cdot))$ a.e. in $B_1$.
\end{lemma}

%\[ (1+\sigma \|\nabla\psi\|_\infty)\Delta h(y(0))+\sigma \|\nabla\psi\|_\infty C_0.\]

\begin{lemma}\label{lem nabla}

Let $h,f$ be as given in \eqref{inf_convolution}. Then for a.e. $x\in B_1$ we have
$$
\left|\nabla f(x)-\nabla h(y)\right|=  |\nabla h(y)||\nabla\psi(x)|\quad\text{ if }f(x)=h(y) \hbox{ and } y\in B(x,\psi(x)).
$$
\end{lemma}

\medskip

Now for  a weak solution  $u$ to \eqref{premain} {in $Q_2$},  let  $\{u_k\}_k$ be its smooth approximations as given in Lemma \ref{approximation}. In particular $u_k$ is positive in $Q_2$ for each $k$.  Set $v_k(x,t):=u_k(x+X(t),t)$ and introduce the corresponding density variable of $v_k$ as
\begin{equation}\label{e.5.7}
    \xi_k(x,t):=\left(\frac{m-1}{m}v_k(x,t)\right)^{\frac{1}{m-1}}=\left(\frac{m-1}{m}u_k(x+X(t),t)\right)^{\frac{1}{m-1}}.
\end{equation}
We define the density version of the operator $\calL_2$ as $\calG(\xi):=\calL_2(v)$ where $\xi=(\frac{m-1}{m}v)^{\frac{1}{m-1}}$ i.e.
$$
\calG(f):= \partial_t f-\Delta f-\nabla\cdot(f(\vec{b}(x,t)-f(x+X(t),t))),
$$
and thus $\calG(\xi_k)=0$.

Let $\varphi:\R^d\to (0,\infty)$ be a smooth function and $\sigma_1,\sigma_2$ be from Lemma \ref{lem delta}. For some constants $ \alpha,A_0, M_0 \geq 1$ to be determined, we define 
\begin{equation}\label{w_k}
w_k(x,t):=e^{A_0\eps t}\inf_{y\in B(x,R_\eps(x,t))}v_k(y+{ r \eps } {{\mu}},{p_\eps}(t)),
\end{equation}
\begin{equation}\label{w}
w(x,t):=e^{A_0\eps t}\inf_{y\in B(x,R_\eps(x,t))}v(y+{ r \eps } {{\mu}},{p_\eps}(t)),
\end{equation}
and
\begin{equation}\label{e.5.8}
\eta_k(x,t):=e^{A_1\eps t}\inf_{y\in B(x,R_\eps(x,t))}\xi_k(y+{ r \eps } {{\mu}},{p_\eps}(t))\quad\text{ with }A_1:=\frac{A_0}{m-1},
\end{equation}

where
\begin{align}
R_\eps(x,t)&:=\eps\varphi(x)(1- \alpha t)\label{R ep}\\
{p_\eps}(t)&:=(1+{\sigma_2}M_0\eps)\left(\frac{e^{A_0\eps t}-1}{A_0\eps}\right).\label{t hat}    
\end{align}
Then $w_k$ is Lipschitz continuous, and
\[\eta_k(x,t)=\left(\frac{m-1}{m}w_k(x,t)\right)^{\frac{1}{m-1}}.\]
Thus to show that $w_k$ is a supersolution for $\calL_2$, it suffices to show that $\eta_k$ is a supersolution for $\calG$.

\medskip

We will apply Lemmas \ref{lem delta}, \ref{lem nabla} with 
$$h=\xi^{m}_k(\cdot+r\eps\mu,p_\eps)\quad \hbox{ and }\quad \psi=R_\eps(\cdot,t).
$$ 
Based on these lemmas we estimate the density equation $\mathcal{G}(\eta_k)$ in the weak sense, to go around the potential lack of smoothness for inf-convolutions, to conclude.

\medskip

 We will choose  the constants $A_0=A_0(M_0)$ and $\alpha=\alpha(M_0)$ in Proposition \ref{lem supersolution}, the constants
 $M_0,r$ and the function $\varphi$ in the proof of Theorem \ref{prop nondeg}.

\begin{proposition}\label{lem supersolution}
Let $u_k,w_k$ be defined from above, and suppose that $u_k$ satisfies $\Delta u_k\geq -C_0$ in $Q_2$. Fix any $M_0\geq 1$ and consider $\varphi :B_2 \to \R$ such that
\begin{equation}\label{cond varphi}
 \left\{  
  \begin{array}{ll}
&\Delta\varphi= \frac{\sigma_1|\nabla\varphi|^2}{|\varphi|},\\ 
  &\frac{r}{M_0}\leq \varphi(\cdot)\leq rM_0,\quad \|\nabla \varphi\|_\infty \leq M_0 \quad \text{ for some }r\in(0,1).
\end{array}\right.
\end{equation}
Then {there exist} { positive constants}  $A_0,\alpha,\tau$ depending only on $M_0$ and universal constants such that for all $\eps< \frac{1}{M_0}$ 
 the function $w_k$ given in \eqref{w_k} is a  weak supersolution of
$$
\mathcal{L}_2( w_k) \geq 0\quad
\text{ in } B_r\times (0,\tau) .
$$

%for ${A_0},\alpha$ depending on $M_0,C_0,{L}$, and for $r,\e,\tau$ depending on $M_0, C_0, L, \sigma_2$.

\end{proposition}

\begin{proof}

Let $\xi_k,\eta_k$ be from \eqref{e.5.7}, \eqref{e.5.8} respectively. As discussed before to prove the statement, it suffices to show that $\calG(\eta_k)\geq 0$ weakly in $ B_r\times (0,\tau)$.

% satisfies 
%\[\partial_t\eta_k-\Delta \eta_k^m-\nabla\cdot(\eta_k(\vec{b}(x,t)-\vec{b}(x+X(0,0;t),t)))\geq 0\quad
%\text{ weakly in } B_r\times (0,\tau) .\]

 %, where $\calG$ is the density version of the operator $\mathcal{L}_2$:

%{define an operator and show how it is related to $\mathcal{L}_2$.}

% satisfies 
%\[\partial_t\eta_k-\Delta \eta_k^m-\nabla\cdot(\eta_k(\vec{b}(x,t)-\vec{b}(x+X(0,0;t),t)))\geq 0\quad
%\text{ weakly in } B_r\times (0,\tau) .\]

%We denote the above operator as $\calG(\eta_k)$ {this is not even operator since you plugged in $\eta_k$ in it.}

\medskip

Below we  estimate each term in $\calG (\eta_k)$ {in $B_r\times (0,\tau)$} using $\xi_k$. We begin with some preliminary estimates on $\eta_k$. 

\medskip

Since $u_k$ is smooth and positive, $\xi_k$ is also smooth and positive. From the definition of the inf-convolution, it follows that $\eta_k$ is Lipschitz continuous. Since $\Delta u_k\geq -C_0$, direct computation yields that 
\begin{equation}
    \label{5.12}\Delta(\xi^m_k)\geq -\sigma C_0\xi_k \hbox{ for some } \sigma = \sigma(m)>0.
\end{equation}

\medskip

Let us set the constants
\begin{equation}
    \label{def Aalpha}
    A_0:=\sigma_3 M_0(1+C_0),\quad  
\alpha :=\sigma_3 M_0^2
\end{equation}
for some $\sigma_3\geq \sigma_2$ to be determined, and 
\begin{equation}
    \tau
 : = \min\left\{\frac{1}{2A_0},\,\frac{1}{2A_1},\, \frac{1}{\sigma_2 M_0},\,\frac{1}{5\alpha}\right\}.
    \label{def tau}
\end{equation}

By definition of ${\eta_k}$, there is $z(x)$ satisfying
\begin{equation}\label{compare y x}
    |z(x)-x|\leq |R_\eps|+r\eps\leq 2 M_0 r\eps,
\end{equation}
such that 
\[{\eta_k}(x,t)=g(t)\,\xi_k(z(x) ,{p_\eps}(t)),\]
where we use the notation $g(t):=e^{A_1\eps t}$. 

\medskip

It  follows from the definition of ${p_\eps}(t)$ in \eqref{t hat} that
\begin{equation}
    \label{5.14}
    p_\eps'(t)=(1+\sigma_2 M_0\eps)g(t)^{m-1}
\end{equation}
and
\begin{equation}
    \label{est st}
    0\leq{p_\eps}(t)-t\leq \sigma M_0 t\eps\leq \sigma \eps \quad\hbox{ for } 0<t<\tau.
\end{equation}

\medskip

We now proceed to estimating each terms in $\mathcal{G}(\eta_k)$, starting with $\partial_t {\eta_k}$. All estimates in the domain $B_r(0) \times (0,\tau)$. In the rest of the proof, for simplicity, $X(t):=X(0,0,;t)$, $p_\eps$, $\eta_k$ denotes the values of them at $(x,t)$, and $\xi_k,\partial_t \xi_k,\nabla \xi_k,\Delta \xi_k$ denotes the values of them evaluated at point $(z(x),p_\eps(t))$.

 \medskip
 
In \cite{kim2017singular}, $\partial_t {\eta_k}$ is computed in the viscosity sense. Since our $\eta_k$ is Lipschitz continuous, the same computation carries out almost everywhere in $ B_r\times (0,\tau)$. We have
\begin{equation}\label{partial}
    \partial_t {\eta_k}\geq { A_1 \eps }\, \eta_k-\partial_tR_\eps\,  |\nabla \eta_k|+(p_\eps')g\,\partial_t \xi_k.
\end{equation}
Applying \eqref{R ep}, \eqref{5.14} and the assumption that $\varphi\geq \frac{r}{M_0}$, \eqref{partial} implies
\begin{equation}\label{partial t w1}
    \partial_t\eta_k\geq { A_1 \eps } \,\eta_k+\frac{  \alpha r \eps}{M_0} |\nabla \eta_k|+(1+{{\sigma_2}}M_0 \eps)g^m\partial_t \xi_k.
 \end{equation}

From the assumptions on $\varphi$, $\|R_\eps\|_\infty\leq r M_0\eps$, $\|\nabla R_\eps\|_\infty\leq M_0\eps$. 
We now apply Lemma \ref{lem delta} with $h=\xi^{m}_k(\cdot+r\eps\mu,p_\eps)$ and $\psi=R_\eps(\cdot,t)$.  From \eqref{cond varphi} and \eqref{5.12}, the following holds in the sense of distribution:
\begin{equation}
    \label{5.1}
    \begin{aligned}
    -\Delta {\eta^m_k}&\geq -(1+\sigma_2\|R_\eps\|_\infty)g^m\Delta \xi^m_k-\sigma_2\|\nabla R_\eps\|_\infty C_0 \xi_k \\
    &\geq -(1+{{\sigma_2}}{M_0} \eps )g^m\Delta \xi^m_k-\sigma M_0 C_0 \eps \,\eta_k.
    \end{aligned} 
\end{equation}

Next we consider the terms coming from the drift. Due to Lemma \ref{lem nabla},
\[
    |\nabla {\eta_k}-g \nabla \xi_k|=|\nabla R_\eps ||g \nabla \xi_k|\leq M_0 \eps g|\nabla \xi_k|,
\]
 since $\eps<\frac{1}{M_0}$, we have
$
    |\nabla {\eta_k}-g \nabla \xi_k|\leq \sigma M_0 \eps |\nabla \eta_k|.
$
This implies that for $t\leq \tau$,
\begin{equation}\label{5.17}
     \left|\nabla \eta_k-(1+{{\sigma_2}}M_0\eps)g^m\nabla \xi_k\right|\leq \sigma M_0 \eps |\nabla \eta_k|.
\end{equation}

Next using the regularity of $\vec{b}$ and $|x|\leq r$, we have
\begin{equation}
    \label{grad est 1}
 \left|  \vec{b}(x+{X(p_\eps)},{p_\eps})-\vec{b}({X(p_\eps)},{p_\eps})\right| \leq \|D\vec{b}\|_\infty r\leq   \sigma  r,
\end{equation}
and, by \eqref{compare y x},
\begin{equation}
\label{grad est 3}
\left|\vec{b}(x+{X(p_\eps)},{p_\eps})-\vec{b}(z+{X(p_\eps)},{p_\eps})\right|\leq \sigma M_0r\eps.\end{equation} 
Then \eqref{5.17}-\eqref{grad est 3} imply
\begin{equation}
    \label{5.3}
\begin{aligned}
    & -\nabla \eta_k\cdot \left(\vec{b}(x+{X(p_\eps)},{p_\eps}) -\vec{b}({X(p_\eps)},{p_\eps})\right)\\
   \geq \, &-(1+{{\sigma_2}}M_0\eps)g^m\nabla \xi_k\cdot \left(\vec{b}(x+{X(p_\eps)},{p_\eps})-\vec{b}({X(p_\eps)},{p_\eps})\right)-\sigma M_0 r\eps |\nabla\eta_k|\\
  \geq\,  &-(1+{{\sigma_2}}M_0\eps)g^m\nabla \xi_k\cdot \left(\vec{b}(z+{X(p_\eps)},{p_\eps})-\vec{b}({X(p_\eps)},{p_\eps})\right)-\sigma M_0 r\eps |\nabla\eta_k|.
\end{aligned}
\end{equation}
{Parallel computations yield}
\begin{equation}
    \label{5.4}
    \begin{aligned}
   &-\eta_k\nabla\cdot\vec{b}(x+{X(p_\eps)})\\
   \geq\,& -(1+\sigma_2M_0\eps)g^m \xi_k \nabla\cdot\vec{b}(x+{X(p_\eps)})-\sigma\,\eta_k\, \left|g-(1+{{\sigma_2}}M_0\eps)g^m\right| \|D\vec{b}\|_\infty\\
    \geq\,& -(1+\sigma_2M_0\eps)g^m \xi_k \nabla\cdot\vec{b}(z+{X(p_\eps)})-\sigma M_0\eps\,\eta_k-\sigma\,\eta_k\, \|D^2\vec{b}\|_\infty M_0r\eps\\
\geq\,  & -(1+\sigma_2M_0\eps)g^m \xi_k \nabla\cdot\vec{b}(z+{X(p_\eps)})-\sigma M_0\eps\,\eta_k.   
    \end{aligned}
\end{equation}

Combining the estimates \eqref{partial t w1}, \eqref{5.1}, \eqref{5.3}  and \eqref{5.4}, we have
\begin{align*}
 \tilde{\calG}(\eta_k):= \, &\partial_t {\eta_k}-\Delta {\eta^m_k}-\nabla\left( {\eta_k}\cdot \left(\vec{b}(x+{X(p_\eps)},{p_\eps})-\vec{b}({X(p_\eps)},{p_\eps})\right)\right)\\
   \geq \, &A_1 \eps \,\eta_k+\frac{ \alpha r \eps }{M_0}|\nabla \eta_k|+(1+{{\sigma_2}}M_0\eps)g^m (\partial_t \xi_k- \Delta \xi^m_k)\\
    &\quad -(1+{{\sigma_2}}M_0\eps)g^m\,\nabla \left(\xi_k\cdot \left(\vec{b}(z+{X(p_\eps)},{p_\eps})-\vec{b}({X(p_\eps)},{p_\eps})\right)\right)\\
    &\quad -\sigma M_0(1+C_0)\eps \,\eta_k-\sigma M_0 r\eps |\nabla \eta_k|.
\end{align*}

Since $\calG(\xi_k)=0$ we obtain
\begin{align*}
   \widetilde{\calG} (\eta_k) &\geq   {A_1} \eps \,\eta_k+\frac{\alpha r\eps}{M_0}|\nabla \xi_k|+(1+\sigma_2 M_0\eps)g^m \calG(\xi_k(\cdot,\cdot))(z,{p_\eps})- \sigma M_0(1+C_0)\eps \,\eta_k- \sigma M_0 r\eps|\nabla \eta_k|    \\
    &= C_1\eps\, \eta_k+C_2 r\eps|\nabla \eta_k|,
\end{align*}
where
\begin{equation}\label{C34}
C_1:= A_1-\sigma M_0(1+C_0),\quad  
C_2:= \frac{\alpha}{M_0}-\sigma M_0.
\end{equation}

Finally we proceed from $\widetilde{\calG}(\eta_k)$ to $\calG (\eta_k)$:
\begin{equation}
    \label{L hat w}
\begin{aligned}
    \calG(\eta_k)&\geq \calG(\eta_k)-{\widetilde{\calG}}({\eta_k})+C_1\eps \,\eta_k+ C_2 r|\nabla \eta_k|\\
    &\geq C_1
\eps \,\eta_k+ C_2  r\eps|\nabla \eta_k|-\eta_k\left|\nabla\cdot \vec{b}(x+{X(p_\eps)},{p_\eps})-\nabla\cdot\vec{b}(x+{X(t)},t)\right|\\
&\quad -|\nabla \eta_k|\underbrace{\left|\vec{b}(x+{X(p_\eps)},{p_\eps})-\vec{b}({X(p_\eps)},{p_\eps})-(\vec{b}(x+{X(t)},t)-\vec{b}({X(t)},t))\right|}_{V_0:=}.
\end{aligned}
\end{equation}
Let us estimate $V_0$:
\begin{align*}
    V_0&=\left|\int_t^{{p_\eps}} \partial_s\vec{b}(x+{X(s)},s)-\partial_s\vec{b}({X(s)},s)ds\right|\\
    &\leq \int_t^{{p_\eps}}\left|\left((D\vec{b})(x+{X(s)},s)-(D\vec{b})({X(s)},s)\right)\vec{b}(X(s))\right|+\left|(\partial_t\vec{b})(x+{X(s)},s)-(\partial_t\vec{b})({X(s)},s)\right| ds\\
    &\leq \sigma|x|\int_t^{{p_\eps}}\left\|D^2\vec{b}\right\|_\infty\left\|\vec{b}\right\|_\infty+\left\|D\partial_t\vec{b}\right\|_\infty ds\leq \sigma  r\eps.
\end{align*}
Similarly,
\begin{align*}
    &\left|\nabla\cdot\vec{b}(x+{X(p_\eps)},{p_\eps})-\nabla\cdot\vec{b}(x+X(t),t)\right|\\
    \leq\,& \int_t^{{p_\eps}}\left|(D\nabla\cdot\vec{b})(x+{X(s)},s)\vec{b}(X(s))\right|+\left|(\partial_t\nabla\cdot\vec{b})(x+{X(s)},s)\right|ds\\
    \leq\,& \sigma\left(
    \left\|D^2\vec{b}\right\|_\infty\left\|\vec{b}\right\|_\infty+\left\|D\partial_t\vec{b}\right\|_\infty
    \right)\,\eps.
\end{align*}
Thus it follows from \eqref{L hat w}  that, if $C_1\geq \sigma $, $C_2\geq \sigma$,
\begin{equation}
    \label{5.24}
    \calG(\eta_k)\geq (C_1-\sigma r)\eps \,\eta_k+(C_2-\sigma )r\eps|\nabla \eta_k|\geq 0 \hbox { in } B_R\times (0,\tau).
\end{equation}
 In view of \eqref{C34}, $C_1, C_2\geq \sigma$ if 
 $\sigma_3$ in \eqref{def Aalpha} is chosen to be large enough depending only on universal constants.
Hence with this choice of $\sigma_3$ we have proved that $\calG(\eta_k)\geq 0$ in the sense of distribution in $B_r\times (0,\tau)$. From the Lipschitz continuity of $\eta_k$ we conclude that $\calG(\eta_k)\geq 0$ weakly in $B_r\times (0,\tau)$.

\medskip

Lastly
it is not hard to see that the choices of $A_0,\alpha,\tau$ are independent of $r$ and $k$.
\end{proof}

%\textcolor{Paul: the domain $\Sigma$ you use below is space-time. Here it has to be $U$. I put it back as $U$.}

\begin{corollary}\label{C.5.5}

Let $u$ be from Theorem \ref{prop nondeg}, and
let $v$ and  $w$ be given by \eqref{v} and \eqref{w} respectively. Suppose that the assumptions in Proposition \ref{lem supersolution} are satisfied. Then for any open set $U\subseteq B_r$, if $w\geq v$ on the parabolic boundary of $U\times (0,\tau)$,  then
\[w\geq v\quad \text{ in }U\times (0,\tau).\]
\end{corollary}
\begin{proof}
Let $\{u_k\}_k$ be the smooth approximations of $u$ and $u_k\geq u$. Let $ v_k(x,t)=u_k(x+X(t),t)$ and $ w_k$ be from \eqref{w_k}. It follows from the proposition that $\calL_2(w_k)\geq 0$ weakly in $B_r\times (0,\tau)$. We have $w_k\geq w$ due to the fact that $u_k\geq u$. Then by the assumption, $w_k\geq v$ on the parabolic boundary of $U\times (0,\tau)$. 
By comparison principle for $\calL_2$, we get $w_k\geq v$ in $U\times (0,\tau)$. Due to Lemma \ref{approximation}, $u_k$ converges locally uniformly to $u$, and so $w_k$ converges locally uniformly to $w$. We conclude by sending $k\to\infty$.
\end{proof}

\medskip

Now we are able to prove Theorem~\ref{prop nondeg}. 

\medskip

{\bf Proof of  Theorem~\ref{prop nondeg}.}
%By our assumption $u$ is monotone with respect to $W_{\theta,\mu}$. 
Let $\sigma_1$ be given in Lemma \ref{lem delta}, and
let $\Phi$ be the unique solution of  
\begin{equation*}
    \left\{\begin{aligned}
        & \Delta (\Phi^{-\sigma_1+1})=0 &\text{ in }&B_\frac{1}{2}\backslash B_{\sin \theta/10}\\
        & \Phi=A_{d,\theta} &\text{ on }&\partial  B_{\sin \theta/10}\\
        & \Phi=\frac{1}{2}\sin{\theta}& \text{ on }&\partial B_{\frac{1}{2}}.
    \end{aligned}
    \right.
\end{equation*}
Here $A_{d,\theta}$ is chosen sufficiently large so that
 
\begin{equation}\label{5.8}
    \Phi\left(y+\frac{{{\mu}}}{5}\right)\geq 3\quad \text{ for all }y\in B_{\frac{1}{10}}.  
\end{equation}
Then for some $M_0(\theta,d)\geq 1$
$$\frac{1}{M_0}\leq \Phi\leq M_0,\quad \|\nabla \Phi\|_\infty \leq M_0 \quad \hbox{ in } B_{\frac{1}{2}}.
$$
With this $M_0$, let ${A_0},\alpha,\tau$ be as given in  Proposition \ref{lem supersolution}.

\medskip
Fix any $(\hat{x},\hat{t})\in Q_1\cap\Gamma$ and let $C^*, h, \beta$ be from Theorem~\ref{strictly expanding} and $\tau$ be from \eqref{def tau}. We will show that the support of the solution strictly expands relatively to the streamlines at $(\hat{x},\hat{t})$. 

Let $\delta=\delta(\theta,C_0)>0$, which will be chosen as a constant satisfying \eqref{delta1} and \eqref{delta2}. Define
\begin{equation}\label{const t1}
    t_\delta: = \min\{\tau, h,\delta\},
\end{equation} 
and
\begin{equation}\label{5.26}
    r_\delta:=\min\left\{C_*t_\delta^\beta,\frac{1}{4}\right\}>0.
\end{equation} 
Due to Theorem~\ref{strictly expanding},
 \begin{equation}
     \label{eqn varrho 0}
u(x,\hat{t}-t_\delta)=0 \quad \text{ for } x\in B(X(\hat{x},\hat{t};-t_\delta),{r_\delta}).
\end{equation}
After translation, we assume $(X(\hat{x},\hat{t};-t_\delta),\hat{t}-t_\delta)$ to be the origin. Using the notation $X(t)=X(0,0;t)$, we have 

\[(X(t_\delta),t_\delta)=(X(X(\hat{x},\hat{t};-t_\delta)),\hat{t}-t_\delta;t_\delta),t_\delta)=(\hat{x},\hat{t})\in\Gamma(u).\]

Let $v$ be as given in \eqref{v}, and then $\mathcal{L}_2(v)=0$ weakly in $Q_{\frac{1}{2}}$, where $\calL_2$ is given in \eqref{cal L 2}.   %In view of \eqref{const t1}, we have \[\{(x+{X(t)},t), x\in {B_1},t\in [0,t_\delta]\}\subset Q_1.\]
It follows from \eqref{eqn varrho 0} that
\begin{equation}\label{ini zero}
    v(x,0)=0 \quad\text{ in }B_{r_\delta}.
\end{equation}
For $P:=-\frac{{r_\delta}}{5}{{\mu}}$, set $\varphi(x):={r_\delta}\Phi(\frac{x-P}{{r_\delta}})$.

Let $w$ be defined as in \eqref{w} with the above $\varphi$ and $r=r_\delta$:
\begin{align*}
w(x,t)&:=e^{A_0\eps t}\inf_{B(x,{{{\eps}}\varphi(x)(1-\alpha t)})}{u}(y+{r_\delta}{{\eps}}\mu + {X(p_\eps(t))} ,{p_\eps}(t))\\
&=e^{A_0\eps t}\inf_{B(x,{{{\eps}}\varphi(x)(1-\alpha t)})}{v}(y+{r_\delta}{{\eps}}\mu ,{p_\eps}(t)).
\end{align*}
%where $p_\eps(t)$ is given by \eqref{t hat}.
Next denote the cylindrical domain
$$
\Sigma:=(B(P, \frac{r_\delta}{2})\setminus B(P,{r_\theta})) \times [0,{t_\delta}]
$$
where $r_\theta:= \frac{{r_\delta}}{10}\sin\theta$.
We claim that
\begin{equation}\label{claim}
 w \geq v \hbox{ in } \Sigma.
\end{equation}

{Roughly speaking, \eqref{claim} states that the nondegeneracy property of $u$ propagates from the positive set to the free boundary, as the positive set expands out relative to the streamlines. }

The proof of \eqref{claim} will be given below. We first discuss its consequences. 

\begin{figure}[t]\centering\includegraphics[width=0.7\textwidth]{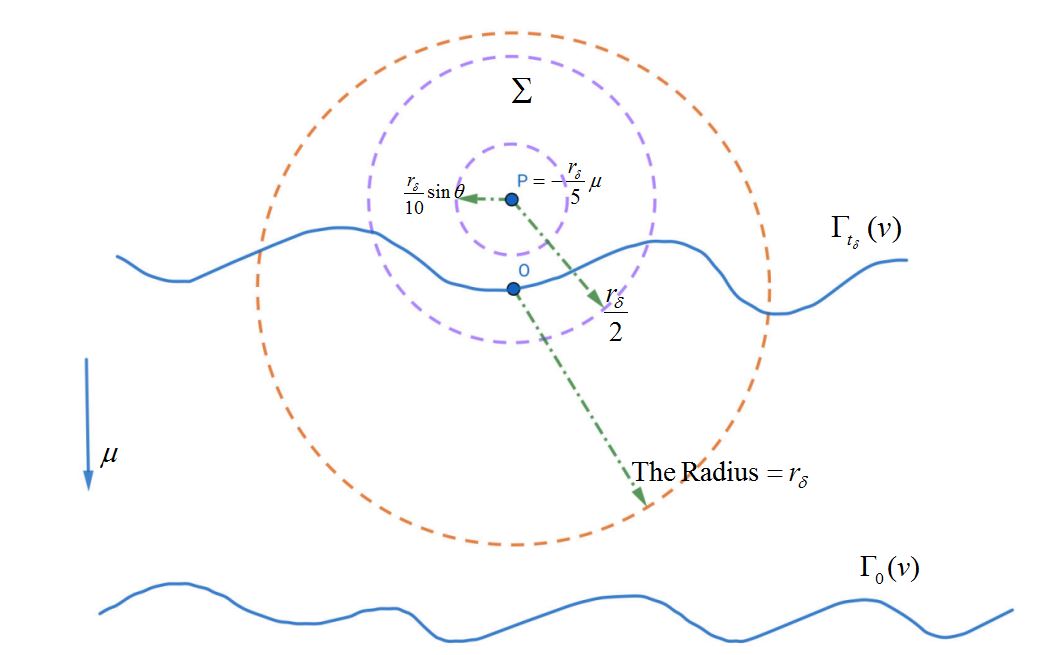}\label{figure1}\caption{}\end{figure}

\medskip

%Since $t_\delta$ is small (see \eqref{def tau}), we have $(1-\alpha t_\delta)\geq 4/5$. 
Using \eqref{def tau} and \eqref{5.8},
\[\varphi(x)=r_\delta\Phi\left(\frac{x}{r_\delta}+\frac{\mu}{5}\right)\geq 3{r_\delta}\quad \text{ for }x\in B_{\frac{r_\delta}{10}}(0).\]
From this, it follows that for all $|x|\leq \frac{r_\delta\eps}{5}\leq  \frac{r_\delta}{10}$,
 \begin{align*}
     -{r_\delta}\eps  {{\mu}}\in B\left(x, \frac{12}{5}{r_\delta}\eps \right)+{r_\delta}\eps  {{\mu}}\subseteq B(x,\eps\varphi(x) (1-\alpha t_\delta))+{r_\delta}\eps {{\mu}} .
 \end{align*}
 In the inclusion, we used that $\alpha t_\delta\leq \frac{1}{5}$.
Then using \eqref{claim} and the definition of $w$, we get for $|x|\leq \frac{r_\delta\eps}{5}$,
\begin{align*}
e^{A_0\eps t_\delta}\,{v}(-{r_\delta}\eps  {{\mu}},{p_\eps}(t_\delta))&\geq 
e^{A_0\eps t_\delta}\,\inf_{B(x,\eps\varphi(x)(1-\alpha t_\delta))}{v}(y+{r_\delta}\eps  {{\mu}},{p_\eps}(t_\delta))\\
&\geq {w}(x,t_\delta)\geq v(x,t_\delta).
\end{align*}
From \eqref{t hat} it follows that $p_\eps(t_\delta)=t_\delta+c\eps$ for some $c=c(t_\delta,\sigma)$ which is independent of $\eps$. Thus
\[u(-r_\delta\eps\mu+X(t_\delta+c\eps),t_\delta+c\eps)\geq e^{-A_0\eps t_\delta}\sup_{|x|\leq r_\delta\eps/5}u(x+X(t_\delta),t_\delta).\]
Recall that $(X(t_\delta),t_\delta)=(\hat{x},\hat{t})\in\Gamma(u)$ and $X(t_\delta+c\eps)=X(X(t_\delta),t_\delta;c\eps)$. We proved
\[u(-r_\delta\eps\mu+X(\hat{x},\hat{t};       c\eps),\hat{t}+c\eps)>0,\]
which implies
\[
u(X(\cdot,\cdot\,;c\eps)-r_\delta\eps{\mu},{\cdot\,+}c\eps)>0\quad \text{ on }  \Gamma\cap Q_1 .\]

\medskip

Now we proceed to prove our claim.

\medskip

%\noindent{\bf The Proof of \eqref{claim}:} 

{\bf Proof of \eqref{claim}}.
{Here we apply Corollary \ref{C.5.5} with the choice of $U:= B(P, \frac{r_\delta}{2})\setminus B(P,{r_\theta})$. To this end, it suffices to show that $w\geq v$ on the parabolic boundary of $\Sigma$.}

\medskip

First observe that  from \eqref{ini zero}, 
 $$
 {w}(x,0)\geq 0=v(x,0)\hbox{ in } B\l P,{\frac{{r_\delta}}{2}}\r.
 $$  
 Since $v(0,t_\delta)=u(X(t_\delta),t_\delta)=0$ and due to Lemma \ref{streamline},
\[{v}(0,t)=u(X({t}),t)=0 \text{ for }t\in [0,{t_\delta}].\]  
Due to the cone monotonicity assumption \eqref{e.5.1},
    \[ w\geq v=0 \text{ in } B(P, r_\theta) \subset B\left(P,{\frac{{r_\delta}}{5}\sin\theta}\right)\times [0,{t_\delta}].\]

Hence to show \eqref{claim}, it remains to show that $w\geq v$ on $\partial B(P,\frac{r_{\delta}}{2}) \times [0,t_{\delta}]$. 

\medskip

By definition of $\varphi$, we have $\varphi(x)=\frac{{r_\delta}}{2}\sin \theta$ on $\partial B(P,{{r_\delta}/2})$. From \eqref{def Aalpha}, we know $A_0\geq \sigma_2M_0$. It follows that for  $x\in\partial B(P,{r_\delta}/2)$,
%Below we will show that the right-hand side of   \eqref{eqn med 1} is no less than $ v(x,t)$ on $\partial B(P, \frac{{r_\delta}}{2})\times [0, t_\delta]$. 
\begin{equation}
    \label{eqn med 1}
\begin{aligned}
    {w}(x,t)&\geq e^{\sigma_2M_0{\eps}}\inf_{y\in B\left(x,{{r\eps}(1-\alpha t)\frac{{\sin\theta}}{2}}\right)}{v}(y+{r_\delta}{{\eps}} {{\mu}},{p_\eps(t)})\\
    &=e^{\sigma_2M_0{\eps}}\inf_{y\in B\left(x,{\frac{{r_\delta}\eps}{2}{\sin\theta}}\right)}{u}(y+{r_\delta}\eps {{\mu}}+{X(p_\eps(t))},{p_\eps(t)})\\
    &=:e^{\sigma_2M_0{\eps}}V_1(x,t).
\end{aligned}
\end{equation}
In view of \eqref{e.5.1}, we have
$$
 \inf_{B(x,{{{{r_\delta}\eps}}{\sin\theta}})}{u}(y+{{r_\delta}\eps} {{\mu}}+{X(t)},t)\geq v(x,t).
 $$ 
Thus it remains to show that
\begin{equation}\label{monotone}
e^{\sigma_2M_0{\eps}} V_1(\cdot,\cdot)\geq \inf_{B(\cdot,{{{{r_\delta}\eps}}{\sin\theta}})}{u}(y+{{r_\delta}\eps} {{\mu}}+{X(\cdot)},\cdot) \qquad\hbox{ on } \partial B(P, \frac{{r_\delta}}{2})\times [0, t_\delta].
 \end{equation}

Take any $(x,t)\in  \partial B(P, \frac{{r_\delta}}{2})\times [0, t_\delta]$, and denote 
\[z:=z(y,\eps)=y+{{r_\delta}\eps}{{\mu}}+X(t)\quad\text{ for any }y\in B\left(x,{\frac{{r_\delta}\eps}{2}{\sin\theta}}\right).\] 
With this notation we can rewrite $V_1(x,t)$ as
\begin{equation}
    \label{5.6}
    \begin{aligned}
\inf_{y\in B\left(x,{\frac{{r_\delta}\eps}{2}{\sin\theta}}\right) } u(z-{X(p_\eps(t))}+{X(t)},p_\eps(t)).
\end{aligned}
\end{equation}

By \eqref{est st} and \eqref{const t1}, we know 
\begin{equation}
    \label{5.29}
        {s_\eps}(t):={p_\eps}(t)-t\leq \sigma\delta\eps. 
\end{equation}
Then 
\begin{equation}
    \label{5.5}
\begin{aligned}
    |X(z,t;{s_\eps(t)})-z-{X(p_\eps(t))}+{X(t)}|&=|X(z,t;{s_\eps(t)})-X(z,t;0)-{X(X(t),t;s_\eps(t))}+{X(X(t),t;0)}|\\
    &=\left|\int_0^{s_\eps(t)} \vec{b}(X(z,t;h),h)-\vec{b}(X(X(t),t;h),h)dh\right|\\
    &\leq \int_0^{s_\eps(t)}\left(\|D\vec{b}\|_\infty|X(z,t;h)-X(X(t),t;h)|\right)dh.
\end{aligned}
\end{equation}
Note that, for some universal $\sigma$,
\begin{align*}
    |X(z,t;h)-X(X(t),t;h)|&\leq |X(z,t;0)-X(X(t),t;0)|+\sigma h\\
    &=|z-X(t)|+\sigma h\\
    &\leq \sigma r_\delta+\sigma h.
\end{align*}
Therefore, \eqref{5.5} and \eqref{5.29} imply that
\begin{align*}
     |X(z,t;s_\eps(t))-z-{X(p_\eps(t))}+{X(t)}|&\leq \sigma r_\delta {s_\eps(t)}+\sigma {s_\eps(t)}^2\nonumber\\
    & \leq \sigma(\delta r_\delta\eps+\delta^2 \eps^2 )\leq \frac{r_\delta\eps}{2}\sin\theta,
\end{align*}
where the last inequality holds if
\begin{equation}
    \label{delta1}
    \delta\leq \frac{\sin\theta}{4\sigma}\quad\text{ and }\quad\eps\leq \frac{r_\delta\sin\theta}{4\sigma\delta^2}. 
\end{equation}

Combining above estimate with \eqref{5.6}, it follows that
\[V_1(x,t)\geq \inf_{y\in B\left(x,{{{r_\delta}\eps}{\sin\theta}}\right) } u(X(z(y),t;{s_\eps(t)}),t+{s_\eps(t)}).\]
Due to \eqref{mono stream}, for ${C}:=(m-1)(C_0+\|\nabla\cdot\vec{b}\|_\infty)$,
\begin{align*}
\inf_{y\in B\left(x,{{{r_\delta}\eps}{\sin\theta}}\right) } u(X(z(y),t;{s_\eps(t)}),t+{s_\eps(t)})    \geq e^{-{{C}}{s_\eps(t)}}\inf_{y\in B(x,{{r_\delta}\eps}{\sin\theta}) }u(y+{{r_\delta}\eps}{{\mu}}+X(t),t).
\end{align*}
In view of \eqref{eqn med 1}, 
we derive 
\begin{equation}
\label{5.7}
\begin{aligned}
    {w}(x,t)
    \geq e^{\sigma_2M_0{\eps}} e^{-{C}s_\eps(t)}\inf_{B(x,{{{{r_\delta}\eps}}{\sin\theta}})}{u}(y+{{r_\delta}\eps} {{\mu}}+{X(t)},t).
\end{aligned}    
\end{equation}
Using \eqref{5.29} again shows
\begin{equation}
    \label{delta2}
e^{\sigma_2M_0{\eps}-{C}s_\eps(t)}\geq e^{\sigma_2M_0{\eps}-{C}\sigma\delta\eps}\geq 1\quad \text{ if  }\quad  
\delta\leq \frac{\sigma}{1+C_0}.
\end{equation}
Now after fixing $\delta=\delta(\theta,C_0)>0$ such that \eqref{delta1} and \eqref{delta2} hold, we can conclude with \eqref{monotone} and then the claim \eqref{claim}. 

\hfill $\Box$

%In the following theorem, we are going to show that the solution $u$ grows linearly near the free boundary, if known that free boundary is strictly expanding relatively to the streamlines.

\medskip

In view of the velocity law \eqref{with drift speed}, non-degeneracy follows once we know that the positive set of the solution is strictly expanding relatively to the streamlines. In the following theorem, we are going to show that indeed the solution $u$ grows linearly near the free boundary.

\begin{corollary}\label{cor nondeg} 
Under the conditions of Theorem \ref{prop nondeg}, there exist $\eps_0,\kappa_*>0$ depending only on constants in \eqref{C.5.1} such that, for all $\e\in (0,\e_0)$, 
\begin{equation}
    \label{nondeg0}
     u(x+\eps \mu,t)\geq \kappa_*\eps \quad \hbox{ for all } (x,t)\in \Gamma\cap Q_1.
     \end{equation}  
\end{corollary}

\begin{proof}
Let $c_0$ be from Lemma \ref{lem part1} and $C$ be from Theorem \ref{prop nondeg}. Define $\kappa:=\frac{c_0\,\sin^2\theta }{4C}$.
We first claim that for all $\eps>0$ sufficiently small 
\begin{equation}
     \label{contradiction delta star}
     \sup_{y\in B(x,{\eps})}u(y,t)\geq \kappa\eps\quad  \text{ for }(x,t)\in \Gamma\cap Q_1.
\end{equation}    
We argue by contradiction.  Suppose that the above claim is false. Then for any $\e_0>0$ there exist  $\e\in (0,\e_0]$ and $(\hat{x},\hat{t})\in \Gamma\cap Q_1$ such that \eqref{contradiction delta star} fails.

Set $t_1:=\hat{t}-C\eps$ and consider the map
$X(\cdot,t_1;C\eps):\mathbb{R}^d\to\mathbb{R}^d,$
which is an isomorphism when $\eps_0$ is small enough.
Since the {positive set of $u$} is strictly expanding relatively to the streamlines, we have
\[u(X(x,t_1;C\eps),\hat{t})>0\quad \text{ for }x\in B_1\cap \Gamma_{t_1}.\]
Using the cone monotonicity condition \eqref{e.5.1} and the fact that $u(\hat{x},\hat{t})=0$, it follows that $(\hat{x}+\mathbb{R}^+\mu) \cap \Gamma_{t_1}\neq \emptyset$. Therefore there exists $(x_1,t_1)\in\Gamma$ such that
\[ X(x_1,t_1;C\eps)=\hat{x}+C_1\eps\mu \quad\text{ for some }C_1>0.\]
Due to \eqref{e.5.1} again, we have
\begin{equation}
    \label{monotone cor 5.5}
   {d}( x_1-{c\eps}{\mu},\Gamma_{t_1})\geq c\eps\sin\theta\quad \text{ for all }c\geq 0.
\end{equation}
In view of Theorem~\ref{prop nondeg}, for all $\eps$ sufficiently small
\[u(X(x_1,t_1;C\eps)-\eps\mu,t_1+C\eps)>0.\]
Therefore, combining with the fact that
\[u(X(x_1,t_1;C\eps)-C_1\eps u,t_1+C\eps)=u(\hat{x},\hat{t})=0,\]
we obtain $C_1\geq 1$.

Next define
\[x_2:=X(\hat{x},\hat{t};-C\eps),\quad f(t):=X(\hat{x}+C_1\eps{\mu},\hat{t};t)-X(\hat{x},\hat{t};t).\] 
Due to \eqref{ode}, 
\[|f'(t)|\leq \|D_x\vec{b}\|_\infty|f(t)|=\sigma |f(t)|, \quad f(0)=C_1\eps\mu \hbox{ and }  f(-C\eps)=x_1-x_2.\]
Thus
\[
    |x_1-x_2-C_1\eps{\mu}|=|f(-{C}\eps)-f(0)|\leq  \sigma CC_1\eps^2.
\]
Using this, \eqref{monotone cor 5.5} and the fact that $C_1\geq 1$, if  $\eps\leq \eps_0$ is sufficiently small compared to $C$, it follows that
\[{d}(x_2,\Gamma_{t_1 })\geq \frac{C_1\eps\sin\theta}{2}\geq\frac{\eps\sin\theta}{2}=:R,\]
which yields
\begin{equation}\label{5.10}
    u(\cdot,t_1)=0 \text{ in }B(x_2,R).
\end{equation}

Note that $t_1+C\eps=\hat{t}$ and $X(x_2,t_1;C\eps)=\hat{x}$ from definition. Therefore the failure of \eqref{contradiction delta star} implies that

\begin{equation}\label{5.11}
\begin{aligned}
\oint_{B(X(x_2,t_1;C \eps),R)}u(x,\hat{t})dx&=\oint_{B(\hat{x},R)}u(x,\hat{t})dx\\
&\leq \kappa\eps=\frac{c_0R^2}{C\eps}.
\end{aligned}
\end{equation}
In the last equality, we used that $\kappa=\frac{c_0\,\sin^2\theta }{4C}$. 

With \eqref{5.10}-\eqref{5.11}, we are able to apply Lemma \ref{lem part1} to get
\[u(x,\hat{t})=0 \quad\text{ in }B(X(x_2,t_1;C\eps),R/6)=B(\hat{x},R/6),\]
which contradicts with the assumption that $(\hat{x},\hat{t})\in\Gamma$. We proved \eqref{contradiction delta star}. It can be seen from the proof that $\eps_0$ only depends on constants in \eqref{C.5.1}.

\medskip

Now we show \eqref{nondeg0}. Let us take $\gamma\in(0,1)$ to be small enough depending only on $\theta$ such that 
$B(\mu,\gamma)\subseteq W_{\theta,\mu}$, which implies that  for any $\eps\in (0,1)$,
\begin{equation}
    \label{inc1}
    \eps \mu\in \bigcap_{z\in B(0,\gamma\eps)}\{z+W_{\theta,\mu}\}.
\end{equation}

Fix any $(x,t)\in\Gamma\cap Q_1$, and set $\kappa_*:=\kappa\gamma$. By \eqref{contradiction delta star}, there exists $\eps_0>0$ such that
\begin{equation*}
\sup_{y\in B(x,{\gamma\eps})}u(y,t)\geq \kappa_*\eps \quad \hbox{ for any } \e\in (0,\eps_0].
\end{equation*} 
Therefore we can find $y\in {B(x,\gamma\eps)}$ that $u(y,t)\geq \kappa_*\eps$. It follows from \eqref{inc1} that $x+\eps\mu\in y+W_{\theta,\mu}$. Due to \eqref{e.5.1}, we conclude with
\begin{equation*}
u(x+\eps {\mu},t)\geq \kappa_*\eps \quad \hbox{ for any $(x,t)\in \Gamma\cap Q_1$ and $\eps\in(0,\eps_0]$}.
\end{equation*}

\end{proof}

\medskip

\section{Flatness Implies Smoothness}
\label{sec 6}

In this section we prove the following theorem.

\begin{theorem}\label{col C1alpha} Let $u$ be as given in Theorem \ref{prop nondeg}. If \eqref{cond ii} holds in $Q_1$, then $u$ is Lipschitz continuous and  $\Gamma\cap Q_{1/2}$ is a $d$-dimensional $C^{1,\alpha}$ surface for some $\alpha\in (0,1)$.
\end{theorem}

The cone monotonicity and \eqref{cond ii} provide sufficient monotonicity properties for the solution to rule out topological singularities and to localize the regularization phenomena driven by the diffusion in the interior of the domain. We follow the outline  for the zero drift built on \cite{C1alpha} and \cite{CVWlipschitz}, while we elaborate on the differences. Most notable difference is in establishing Proposition~\ref{condlem nondeg'}.

%In the proof we set $(x_0,t_0) = (0,0)$ for simplicity.  This result is of independent interest, since the propagation of regularity from the free boundary to positive level sets has not been investigated before in the setting of degenerate diffusions.

\medskip

%Throughout this section, by $C,L,\delta$ we denote constants that depend on the constants in \eqref{C.5.1}.

%\medskip

\begin{lemma}\label{condlem nondeg}
Under the conditions of Theorem \ref{col C1alpha}, $u$ is Lipschitz continuous in $Q_1$, and $\Gamma \cap\, Q_{{1}/{2}}$ is a $d$-dimensional Lipschitz continuous surface.

\end{lemma}

\begin{proof}
%Suppose \eqref{local non deg} holds, we only need to show that $|u_t|$ is bounded. By the equation\[u_t=(m-1)u\Delta u+|\nabla u|^2+\nabla u\cdot\vec{b}+(m-1)u\nabla\cdot\vec{b}.\]With the help of Lemma 3.2 \cite{C1alpha}, we obtain\[ |uD_{ij}u| \leq C,\quad \text{ for all }i,j=1,...,d.\]By boundedness of $\nabla u,\,\vec{b},\,\nabla\cdot\vec{b}$, we obtain boundedness of $|u_t|$ by the equation.

First let us prove that $u$ is Lipschitz continuous in $Q_1$. Since $u$ satisfies a parabolic equation locally uniformly in its positive set, $u$ is smooth in $\{u>0\}$. From the equation and $\Delta u\geq -C_0$, we obtain 

 \begin{equation}\label{fundamental00}
 u_t \geq |\nabla u|^2 - \sigma(C_0+1)u + \nabla u\cdot \vec{b} \hbox{ in } \{u>0\},
 \end{equation}
 where $\sigma$ is universal.
Above estimate combined with condition \eqref{cond ii} yields 
$$(A+\sigma)|\nabla u| + C(C_0,A,\sigma)\, u+A \geq |\nabla u|^2, $$
which turns into a bound on $|\nabla u|$ in $\{u>0\}$. 
From \eqref{cond ii}, we also get a bound on $|u_t|$. Notice the bounds are independent of the ellipticity constants of the equation satisfied by $u$. Indeed we have,
\begin{equation}
    \label{6.1 Lip}
    |\nabla u|+|u_t|\leq C \quad \text{ in } Q_1\cap\{u>0\}
\end{equation}
for some $C$ only depending on $A,C_0$ and universal constants. Since $u$ is continuous and nonnegative, it is not hard to see that the same estimate holds weakly in $Q_1$.

\medskip

Next we turn to the Lipschitz continuity of $\Gamma$, using the cone monotonicity and Lipschitz continuity of $u$. The spatial cone monotonicity of $u$ implies that for each $t\in(-1,1)$, $\Gamma_t$ is a Lipschitz continuous graph in $\mathbb{R}^d$. Thus it remains to show that for each $\tau\in(-1,1)$, $\Gamma_{t+\tau}\cap B_{\frac{1}{2}}$ is in a $C\tau$ neighbourhood of $\Gamma_{t}\cap B_1$ for some $C>0$. 
To this end it is enough to show the following:
for $(x,t)\in \Gamma \cap Q_{\frac{1}{2}}$ and for $\tau>0$ sufficiently small, we have
\begin{equation}\label{claim1}
d(x, \Omega_{t+\tau})\hbox{ and }   d(x, \{u(\cdot,t+\tau) = 0\}) \leq C\tau.
\end{equation}

To show \eqref{claim1}  let us fix $(x,t) \in \Gamma\cap Q_{\frac{1}{2}}$. Observe that from Lemma \ref{streamline}  there exists $C>0$ such that, if $\tau>0$ is small,
$$
d(x, \Omega_{t+\tau}(u)) \leq C\tau.
$$

Thus it remains to show the second inequality in \eqref{claim1}. Let $C_1>0$ be a sufficiently large constant to be chosen later. From the cone monotonicity
\[u(\cdot,t)=0\text{ in }B(y,R), \]
where $y:=x-C_1\tau\mu$ and  $R:=C_1\sin\theta\, \tau$. By the Lipschitz continuity of $u$, 
\begin{align*}
    \sup_{z\in B(X(y,t;\tau),R)}u(z,t+\tau)&\leq u(X(y,t;\tau),t)+C(R+\tau)\\
    &\leq u(y,t)+\sigma C\tau+C(1+C_1\sin\theta)\tau \quad \text{( since $|X(y,t;\tau)-y|\leq \|\vec{b}\|_\infty \tau$)}\\
    &\leq C\,C_1\,\tau,
\end{align*}
where $C$ depends on $Lip(u)$ and $\|\vec{b}\|_\infty$.
Thus, for $c_0$ given in Lemma \ref{lem part1},
\[\oint_{B(X(y,t;\tau),R)} u(z,t+\tau)dz\leq CC_1\tau\leq (c_0\,C_1^2\sin^2\theta)\,\tau=c_0\frac{R^2}{\tau},\]
where the last inequality holds if $C_1$ is large enough compared to $1/c_0, 1/\theta, Lip(u),\|\vec{b}\|_\infty$.
Remark \ref{remark part1} then yields for small $\tau$, \[u(x-C_1\tau\mu,t+\tau)=0\] and therefore \eqref{claim1} is proved.
\end{proof}

Now we start proving the $C^{1,\alpha}$ regularity of the free boundary. By considering $\tilde{u}(x,t):=2u(x_0+\frac{1}{2}x,t_0+\frac{1}{2}t)$ for any $(x_0,t_0)\in Q_{\frac{1}{2}}\cap \Gamma$, we can assume $(0,0)\in\Gamma$. And to prove the rest of Theorem \ref{col C1alpha}, it suffices to show that $\Gamma$ is $C^{1,\alpha}$ at point $(0,0)$.

The following proposition propagates the free boundary non-degeneracy  in Corollary~\ref{cor nondeg} to the nearby level sets.

\begin{proposition}\label{condlem nondeg'}
Under the conditions of Theorem \ref{col C1alpha}, there exist constants $0<\delta_1<\frac{1}{2}$ and $c_1>0$ such that
%\begin{equation}\label{local non deg}\frac{1}{C}\leq |\nabla u(x,t)|\leq C,\end{equation}
\begin{equation*}
    \nabla_{{{\mu}}} u(x,t)\geq c_1\quad \text{ {a.e.} in $Q_{\delta_1} \cap\Omega(u)$.}
\end{equation*} 

\end{proposition}

\begin{proof}
%$v$ is Lipschitz continuous in a neighbourhood of $(0,0)$, because $u$ is Lipschitz continuous. 

% We conclude that, for $(x,t)\in \Gamma$ and any $\e>0$ small, \begin{equation}\label{nondeg0}\phi(x,t):=(u(x+\eps {\mu},t)-u(x,t))/\eps \geq \kappa>0.\end{equation}

Fix a sufficiently small $\delta>0$ to be determined and pick $(\hat{x},\hat{t})\in \{u>0\}\cap Q_\delta$.
Let $h:={d}(\hat{x}, \Gamma_{\hat{t}})<\delta$. 
From Lemma 6.1, $\Gamma(u)$ is space-time Lipschitz continuous, and actually it can be written as the graph of $x_\nu=F_{u}(x^\perp,t)$ where $x_\nu:=x\cdot\nu$ and $x^\perp\in\{x\cdot\nu=0\}$. Let us denote the space-time Lipschitz constant of $F_u$ as $C$, and choose $C_2:= C+1$. Then

\[{d}(\hat{x},\Gamma_{\hat{t}-h})\leq ({C_2}-1)h.\]
Denote $(y,s)$ such that $s=\hat{t}-h, y\in \Gamma_s$ and
${d}(\hat{x},y)={d}(\hat{x},\Gamma_{s})\leq (C_2-1)h$. Thus $B(y,h)\subseteq B(\hat{x},{C_2} h)$. Also by Lipschitz continuity of $\Gamma_s$ in space, $\partial B(y,h)\cap \{u>0\}$ is of strictly positive measure $Ch$ with $C$ independent of $h$.

By the fundamental theorem of calculus and \eqref{nondeg0},
\[\oint_{B(y,h)\cap \{u>0\}}\nabla_\mu u(x,s)dx\geq\frac{\sigma}{h}\oint_{\partial B(y,h)\cap \{u>0\}}u(x,s)dx\geq {\kappa}\]
for some ${\kappa}>0$ only depending on $\kappa_*$ and $C_2$.

Let us define
\begin{equation}\label{notation omega}\Omega^r:=\{(x,t)\in\Omega, {d}((x,t),\partial\Omega)>r\}.\end{equation}
Fix $\gamma\in (0,\frac{1}{2})$ to be a small constant only depending on ${\kappa}$ such that
\[\oint_{B(y,h)\cap \Omega^{\gamma h}}\nabla_\mu u(x,s)dx\geq \frac{{\kappa}}{2}.\]
Therefore there exists a point  \[z\in B(y,h)\cap \Omega^{\gamma h}\subset B(\hat{x},{{C_2} h})\cap \Omega^{\gamma h}\] such that 
\begin{equation}
    \label{reference point}
    \nabla_\mu u(z,s)\geq \frac{{\kappa}}{2}.
\end{equation}

We will apply Harnack inequality to $\phi:= \nabla_\mu u$, using the fact that it solves a  locally uniform parabolic equation in the positive set of $u$.

Let us consider $\{u>0\}$ and then $u$ is $C^1$ inside the open region.
Differentiating \eqref{premain} in $\{u>0\}$, we can check that $\phi$  satisfies the following parabolic equation
\[ \phi_t=(m-1)\phi\Delta u+(m-1)u\Delta \phi+(2\nabla u + \vec{b})\cdot \nabla \phi +(m-1)\phi\nabla\cdot\vec{b}+f
\]
 in $\Omega^r$, where 
\[ f:=\nabla u\cdot \nabla_{\mu}\vec{b}+(m-1)u\nabla\cdot \nabla_{\mu}\vec{b}.\]
Since $u$ is Lipschitz continuous and $\vec{b}$ is smooth, $f$ is uniformly bounded. 
Then the new function
\[\tilde{\phi}:=\phi e^{C_3(t-s)} +\|f\|_\infty(t-s)\quad \text{ with } C_3:=(m-1)(C_0+\|\nabla\cdot\vec{b}\|_\infty) \]
satisfies 
\[ \tilde{\phi}_t\geq (m-1)u\Delta \tilde{\phi}+(2\nabla u + \vec{b})\cdot\nabla \tilde{\phi}.\]

Next let us define
\begin{align*}
    \label{def sigma}
    \Sigma_1^h&:=\Omega^{\gamma h}\cap\left(B(\hat{x},{{C_2} h})\times (-h+\hat{t},\hat{t}\,)\right),\\
    \Sigma_2^h&:=\Omega^{\gamma h/2}\cap\left(B(\hat{x},{2{C_2} h})\times (-2h+\hat{t},\hat{t}\,)\right),
\end{align*}
where $\Omega^r$ is as given in \eqref{notation omega}.
We have
\[(\hat{x},\hat{t}),\,(z,s)\in \Sigma_1^h\subseteq \Sigma_2^h.\]

For any $(x,t)\in\Sigma_2^h$ which is $\frac{\gamma h}{2}$ away from $\Gamma$, by the cone monotonicity and \eqref{nondeg0} we have
\begin{equation}\label{ineq u ch}
    u\geq \frac{\kappa_*\gamma h}{2} .
\end{equation}
Thus $w(x,t):=\tilde{\phi}(xh+\hat{x}, th+s)$ 
satisfies
\begin{equation}\label{eqn w sigma}
    w_t\geq (m-1)\frac{u}{h}\Delta w+(2\nabla u + \vec{b})\cdot\nabla w
    \quad\text{ in }  \Sigma_2 := (\Sigma_2^h-(\hat{x},s))/h.
\end{equation}
Also we denote \[\Sigma_1:=(\Sigma_1^h-(\hat{x},s))/h\subseteq \Sigma_2.\]
Notice that $\Sigma_1,\Sigma_2$ are domains with Lipschitz boundary with Lipschitz constant depending only on $C,\sigma$. Writing $\Sigma_i(t)=\{x\,|\,(x,t)\in \Sigma_i\}$ for $i=1,2$, we have
\[ \Sigma_2(t)+B_{\frac{\gamma}{2}}\subseteq \Sigma_1(t)\text{ for }t\in (-h+\hat{t},\hat{t}).\]%And each $t$ time slice of $\Sigma_1$ contains the $\frac{\gamma}{2}$-neighbourhood of the time slice of $\Sigma_2$.

Since $\frac{u}{h}\geq \frac{\kappa_*\gamma}{2}>0$ in $\Sigma_1$ due to \eqref{ineq u ch}, the operator in \eqref{eqn w sigma} is uniformly parabolic in $\Sigma_2$. 
Let us apply the Harnack inequality to $w$ in $\Sigma_1$ and write it in terms of $\phi$, 
to obtain
\[\phi(\hat{x},\hat{t})e^{C_3(\hat{t}-s)}+\|f\|_\infty(\hat{t}-s)\geq \frac{1}{C}\phi(z,s).
\] 
for some constant $C=C(\theta,\kappa_*,C_2)>0$, which is larger than $\frac{{\kappa}}{2C}$ due to \eqref{reference point}.

Since $\hat{t}-s=h\leq \delta$, further assuming $\delta $ to be small enough, we can get
$\phi(\hat{x},\hat{t})\geq \frac{{\kappa}}{4C}>0.$ Finally we conclude that $\nabla_\mu u\geq \frac{{\kappa}}{4C}>0$ in ${\Omega\cap Q_\delta}$.

\end{proof}

Next we show the strict monotonicity of $u$ along the streamlines.

\begin{lemma}\label{lem 6.3}
Let $u$ be given as in Proposition~\ref{condlem nondeg'}. Then there exist  $\delta_2\in (0,\delta_1)$ and $c_2>0$ such that, for $v(x,t):=u(x+X(t),t)$ with $X(t)=X(0,0;t)$, we have
\begin{equation*}
  v_t\geq c_2\quad \text{  in $Q_{\delta_2}\cap\{v>0\}$.}
\end{equation*} 

\end{lemma}

\begin{proof}
By definition, $v$ solves $\mathcal{L}_2(v)=0$, where $\mathcal{L}_2$ is as given in \eqref{cal L 2}.
By the equation, we have 
\begin{align*}
    \partial_t v
    &\geq -C_0(m-1) v+\frac{1}{2}|\nabla v|^2-4|\vec{b}(x+X(t))-\vec{b}(X(t))|^2-(m-1)v\|{\nabla}\vec{b}\|_\infty\\
    &\geq -\sigma C_0\delta
+\frac{c_1^2}{2}-4|x|^2\|{\nabla}\vec{b}\|_\infty^2-C\delta\\
    &\geq -\sigma C_0\delta +\frac{c_1^2}{2}-\sigma\delta^2-C\delta   \qquad\qquad\qquad \hbox{ in } Q_\delta,
    \end{align*}
{where the second inequality comes from the fact that $v\leq C\delta$ due to \eqref{6.1 Lip}, and  the third inequality follows from Proposition \ref{condlem nondeg'}.}

\medskip
   
Since $c_1$ is independent of $\delta$, the last quantity is positive if $\delta= \delta_2$ is small enough compared to $C_0, c_1,$ the Lipchitz constant of $u$ and universal constants. We thus conclude.
\end{proof}

Now we are ready to follow the  celebrated iteration procedure given in \cite{C1alpha}. Their argument describes the enlargement of cone of monotonicity as we zoom in near a free boundary point. More precise discussions are below.

\medskip

Our reference point is $(0,0)\in\Gamma$, and let $v$ be from Lemma \ref{lem 6.3}. For $\delta\in (0,\delta_2)$, define
\begin{equation}\label{udelta sec 7}v_\delta(x,t):=\frac{1}{\delta}v(\delta x,\delta t) , \quad \vec{b}_\delta(x,t):=\vec{b}(\delta x,\delta t),  \quad X_\delta:=\frac{1}{\delta}X(\delta t).
 \end{equation} 
Then $X_\delta$ is the streamline generated by $\vec{b}_\delta$ starting at $(0,0)$. We have that $v_\delta$ is a solution to $\calL_2(\cdot)=0$ with $\vec{b},X$ replaced by $\vec{b}_\delta,X_{\delta}$.
From Lemmas \ref{condlem nondeg} - \ref{lem 6.3}, we have for some $L>0$ independent of $\delta$ (depending on constants in \eqref{C.5.1}) such that
\begin{equation}
    0\leq v_{\delta}\leq L,\quad \frac{1}{L}\leq |\nabla v_{\delta}|,\, \nabla_{{{\mu}}}v_{\delta},\, \partial v_{\delta}\leq L,\quad \Delta v_{\delta}\geq -L\delta \quad \hbox{ in } Q_1 \label{cond v sec 7}.
\end{equation}
Denoting $\sigma$ as the $C^2$ norm of $\vec{b}$, we have
\begin{equation}
     \label{cond V sec 7}
 \|\vec{b}_{\delta}\|_\infty\leq \sigma,\quad \|{\nabla}\vec{b}_{\delta}\|_\infty+\|\partial_t\vec{b}_{\delta}\|_\infty\leq \sigma\delta,\quad \|D^2 \vec{b}_{\delta}\|_\infty+\|\nabla\partial_t \vec{b}_{\delta}\|_\infty\leq \sigma\delta^2.
\end{equation}

\medskip

Let $\widehat{W}_{\theta,{\nu}}$ be given as in \eqref{spacetime cone}. We say $v$ has the {\it cone of monotonicity} $\widehat{W}_{\theta,\nu}$ {in $Q_1$} if 
\begin{equation*}
\hat{\nabla}_p v\geq  0\quad\hbox{  in } Q_1\hbox{  for all }p\in \widehat{W}_{\theta,\nu}.
\end{equation*}
%\textcolor{(in the above we need $\geq 0$ instead of $\geq\frac{1}{2C_0}$. to me this definition is only useful in describing Caffarelli's results. I understand the importance of describing the previous result and then present the differences.)}
The following lemma, yielding the initial cone of monotonicity for $v_\delta$, can be proven using \eqref{cond v sec 7}- \eqref{cond V sec 7} with a  parallel proof to  Proposition 2.1 of \cite{C1alpha}. Let us denote the positive time direction as $e_{d+1}$.% and $\bar{\mu}=(\mu,0)\in \mathbb{R}^{d+1}$.

\begin{lemma}\label{prop 2.1}
Let $v_{\delta}$ be as given in  \eqref{udelta sec 7}. Then there exists $\mathcal{\theta}_0>0$ such that  
\[\hat{\nabla}_p v_\delta\geq  \frac{1}{2L}\quad\hbox{  in } Q_1\quad \hbox{  for all } p\in \widehat{W}_{\theta_0,{\mu}_0} \cap {\mathcal{S}^{d+1},}\]
where ${\mu}_0:=\frac{1}{\sqrt{2}}[(\mu,0)+e_{d+1}]$ and $L$ is as given in \eqref{cond v sec 7}.%$v$ has cone of monotonicity $\widehat{W}_{\theta_0, {\mu}_0}$ in $Q_1$. 
\end{lemma}

Now we begin our iteration procedure. Fix some $J(L)\in (0,1)$ to be chosen later,  define 
\begin{equation}\label{scale_k}
v_k(x,t):=\frac{1}{J^k}v_\delta(J^k x,J^k t) \quad\hbox{ for } k\in \mathbb{N}^+.
\end{equation}
Then $v_k$ satisfies
\begin{equation}\label{6.4}
\begin{aligned}
&    \partial_t{v_k} -(m-1){v_k}\Delta {v_k}-|\nabla {v_k}|^2-\nabla {v_k}\cdot (\vec{b}_k(x+X_k(t),t)-\vec{b}_k(X_k(t),t))\\
&\qquad\qquad -(m-1){v_k}\nabla\cdot \vec{b}_k(x+X_k(t),t)=0.
\end{aligned}
\end{equation}
where $\vec{b}_k(x,t):=\vec{b}_\delta(J^k x,J^k t), X_k(t):=\frac{1}{J^k}X_\delta(J^kt)$. 

\medskip

Due to  \eqref{cond v sec 7} - \eqref{cond V sec 7} the following holds  in $Q_1$:
\begin{itemize}
\item[$(A_k)$] $0\leq v_k\leq {L},$\quad  $\Delta v_k\geq -{L}\delta,$ \quad $ |\nabla v_k|+| \partial_t v_k|\leq {L}$;
\item[$(B_k)$] $\nabla_{{{\mu}}} v_k,$\quad $\partial_t v_k\geq\frac{1}{{L}}$;
\item[$(C_k)$] $\|\vec{b}_k\|_\infty\leq \sigma,$\quad $\|{\nabla}\vec{b}_k\|_\infty+\|\partial_t\vec{b}_k\|_\infty\leq \sigma\delta {J^k},$\quad $\|D^2 \vec{b}_k\|_\infty+\|\nabla\partial_t \vec{b}_k\|_\infty\leq \sigma\delta^2 {J^{2k}}$.% and $|\partial_t X_k|\leq \sigma$.
\end{itemize}

\medskip

The {main step} in the proof of Theorem \ref{col C1alpha} is to show the following property inductively.

%\textcolor{I changed this statement to be more clear, check to see if it is correct.}

\begin{itemize}
\item[$(D_k)$] there exist $s\in (0,1)$ and ${\mu}_k\in\mathbb{R}^{d+1}$ such that for $\theta_k:=\frac{\pi}{2}-s^k(\frac{\pi}{2}-\theta_0)$,
\begin{equation}\label{d_k}
\hat{\nabla}_p v_k\geq  \frac{1}{2L} {J^k} \quad \hbox{ in } Q_1 \quad \hbox{ for all } p\in \widehat{W}_{\theta_k,{\mu}_k} \cap \mathcal{S}^d.
\end{equation}
\end{itemize} 
Once establishing ($D_k$), it shows that the cone of monotonicity $\widehat{W}_{\theta_k, {\mu}_k}$ for $v_k$ has strictly increasing $\theta_k$, converging to $\pi/2$ as $k \to \infty$. The rate of its increasing angles leads to the $C^{1,\alpha}$ regularity of the free boundary.  

In \cite{C1alpha}, {\eqref{d_k} is stated with the weaker requirement $\hat{\nabla}_p v_k\geq  0$.} However for us the competition between diffusion and drift requires a stronger inductive property: see Remark 6.9. This extra observation follows from the enlargement of cones as well as the non-degeneracy of the solution. 

\medskip
 We will proceed with several lemmas that leads to the enlargement of cones in Proposition~\ref{prop 2.5}. The proofs of the lemmas will be postponed until after the proof of the Proposition.

\medskip

First we show that  some improvements on monotonicity can be obtained on the set $\{v_k=\e\}$.

% \textcolor{$(A_k)-(C_k)$ is built in the definition of $v_k$, so it is just $(D_k)$ that is being iterated. I made it clear in the lemma below.} \textcolor{yes}
 
\begin{lemma}\label{prop 2.2} [Enlargement of Cones]
{Let $v_k$ be as given in \eqref{scale_k}, and suppose that $v_k$ satisfies $(D_k)$.} For any $\eps \in(0,1)$, there exist positive constants $r\leq \frac{1}{10},\delta_0<\delta_2, C$ only depending on $\eps,L,\sigma$  such that  the following holds:

For any $\gamma\in (0,\eps)$, $\delta\in (0,\delta_0)$, $p\in \widehat{W}_{\theta_k, {\mu}_k}\cap\,\mathcal{S}^{d}$ and $\tau:=C\eps^{-1}\cos\langle \,p, \hat{\nabla}v_k({{\mu}},-2r) \,\rangle$,  we have 
$$
v_k\leq \eps\,\, \hbox{ in } Q_{2r}; \quad \hbox{ and } \quad v_k((x,t)+\gamma p)\geq (1+\tau\gamma)v_k \,\,\hbox{  on }(B_\frac{3}{4}\times(-2r,2r))\cap \{v=\eps\}.
$$

\end{lemma}

Next we show that this improvement can propagate to the zero level set of $v$.

\begin{lemma}\label{prop 2.3}
{Let $v_k$ be as given in Lemma~\ref{prop 2.2}}. Let $\delta(\eps,L),r(\eps,L),\tau(\eps,L)$   be as given  in Lemma \ref{prop 2.2}. 
Let $w$ be a supersolution of \eqref{6.4}, and suppose that $w\geq v_k$ in $Q_1$ and
 $$w\geq (1+\tau\gamma)v_k \hbox{ in }({B_\frac{1}{2}}\times(-2r,2r))\cap \{v_k=\eps\}.
 $$ Then, if $\eps$ is small enough (independently of $\delta, r,\tau$),
\[w\geq (1+\tau\gamma)v_k \text{ in }({B_\frac{1}{4}}\times(-2r,2r))\cap\{v_k\leq \eps\}.\]
\end{lemma}

Lastly we further improve the monotonicity in a smaller domain of size $r$.

\begin{lemma}
\label{prop 2.4}
Let $v_k,w, \tau $ be as in Lemma \ref{prop 2.3}. There exists a small $\kappa>0$ depending only on $L$ and universal constants such that the following holds. Consider any smooth function  $\phi: \R^d\to \R^+$ such that $\phi$ is supported in $B_{2r}$ and $\phi,|\nabla\phi|,|D^2\phi|\leq \kappa\tau\gamma$.  If $v_k\leq \eps$ in $Q_{2r}$ then we have
 \[w(x,t)\geq v_k(x+(t+2r)\phi(x)\mu,t) \text{ in }Q_{2r}. \]

\end{lemma}

\begin{remark}\label{remark}
In \cite{C1alpha} for the zero drift case, $w$ in the above lemmas {is chosen as a translation of $v_k$ to derive monotonicity properties of $v_k$}. Since our equation is not translation invariant, we instead choose $w$ of the form $v_k((x,t)+p)+E t$ with $E>0$. {To control the extra term $Et$ we rely on the inductive property $(D_k)$. The order between $v_k,w$ is still enough to derive the Proposition below.}
 \end{remark}

\medskip

Now we state the main proposition.% After establishing the proposition, our Theorem \ref{col C1alpha} follows from the proof of Theorem 1 in \cite{C1alpha}. 

\begin{proposition}\label{prop 2.5}[Improvement of Monotonicity]
{Let $v_\delta$ be as given in \eqref{udelta sec 7} with $\delta<\delta_0$ is as given in Lemma~\ref{prop 2.2}. Then there exist constants $J, s\in (0,1)$} independent of $k$ such that the following holds.  Suppose $(0,0)\in\Gamma$ and \eqref{cond v sec 7} - \eqref{cond V sec 7}. Then there exist a monotone family of cones $\widehat{W}_{\theta_k,{\mu}_k}$ with $\theta_k=\frac{\pi}{2}-s^k(\frac{\pi}{2}-\theta_0)$ such that
$$\hat{\nabla}_p v_\delta\geq (2L)^{-1} J^k \quad \hbox { in } Q_{J^k} \quad \hbox{ for all } p\in \widehat{W}_{\theta_k,{\mu}_k} \cap \mathcal{S}^d.
$$

\end{proposition}
%\textcolor{I don't understand the previous statement. I changed it.}

 The $C^{1,\alpha}$ regularity of $\Gamma$ at $(0,0)$ is a result of the relation $\theta_k = \theta_{k-1} + S(\pi/2 - \theta_{k-1})$ which describes quantitatively the enlargement of {cone of monotonicity} of solutions near the free boundary. Then Theorem \ref{col C1alpha} follows. We omit detailed discussion of this part since it is parallel to Theorem 1 in \cite{C1alpha}.

\begin{proof}
Fix a small $\eps>0$ such that the conclusion of Lemma \ref{prop 2.3} holds, and let $r,\delta_0$ be as given in Lemma \ref{prop 2.2}. Then $\eps,\delta_0,r$ only depend on $L$ and universal constants. Define $\tau$ as in Lemma \ref{prop 2.2}. Let $v_k$ be as in \eqref{scale_k}, 
 and set $\vec{b}_k,X_k$ as before and we take $J\leq r$ to be determined. 
It is straightforward that for all $k\geq 0$, $(A_k)-(C_k)$ hold.
When $k=0$, due to Lemma \ref{prop 2.1}, $(D_0)$ holds for $v=v_0$. 

\medskip

Let us suppose that $(D_k)$ holds for some $k\geq 0$ with ${\mu}_k,\theta_k\geq\theta_0$ i.e. the hypothesis of Lemmas \ref{prop 2.2}- \ref{prop 2.4} are satisfied.
We will show $(D_{k+1})$.

For any $\gamma\in(0,\eps)$ and a unit vector $p\in \widehat{W}_{\theta_{k},{\mu}_{k}}$, define 
$$
\widetilde{w}(x,t):=v_k((x,t)+\gamma p).
$$
Note that  $\widetilde{w}\geq v_k$ in $Q_1$ due to $D_k$.
Next, \eqref{6.4} implies that
\begin{align*}
\mathcal{L}_2 (\widetilde{w})\geq   -\gamma\left(  |\nabla \widetilde{w}| |\hat{\nabla}_p \vec{b}_k(x+X_k)| +(m-1)\widetilde{w} |\nabla\cdot \hat{\nabla}_p \vec{b}_k|\right) .
\end{align*}
By $(A_k)-(C_k)$ and the fact that $|\partial_t X_k|\leq |\vec{b}_k|\leq \sigma$, we have
\[\mathcal{L}_2 (\widetilde{w})\geq   -\gamma  (\sigma L\delta J^k)=:-\gamma E_k.\]
Then for $w:=\widetilde{w}+E_k(t+2r)$, we have $w\geq v_k$ in $Q_{2r}$. 

In view of Lemma \ref{prop 2.2}, $v_k\leq \eps$ in $Q_r$ and $w$ satisfies the hypothesis of Lemma \ref{prop 2.3}. Let $\tau$ be defined as in Lemma \ref{prop 2.2}, and let  $\kappa <\kappa_0$ be from Lemma \ref{prop 2.4}.
We select a smooth function $\phi:\mathbb{R}^d\to \R^+$ such that $\phi$ is supported in $B_{2r}$, and $\phi,|\nabla\phi|, |D^2\phi|\leq \kappa\tau\gamma$, and
\begin{equation}
    \label{lower b phi}
    \phi\geq \sigma r^2\kappa\tau\gamma\quad  \text{ in } B_r \,\,\text{ for some universal }\sigma.
\end{equation}
Clearly such $\phi$ exists.

It follows from Lemmas 6.6-6.8 that
\[w(x,t) \geq  v_k(x+(t+2r)\phi(x){{\mu}},t) \quad\hbox{ in } Q_{2r}.\]
By $(B_k)$ and \eqref{lower b phi}, for $c':= \frac{\sigma r^3\kappa}{L}$ we have
\begin{equation*}
    \label{diff w v}
\begin{aligned}w(x,t) 
\geq v_k(x,t)+\frac{t+2r}{{L}}\phi(x)\geq v_k(x,t)+c'\tau\gamma \hbox{ in } Q_r.
\end{aligned}
\end{equation*}
This implies that 
\begin{equation*}
\begin{aligned}
    \hat{\nabla}_p v_k(x,t)&=\lim_{\gamma\to 0}\frac{v_k((x,t)+\gamma p)-v_k(x,t)}{\gamma}\\
    &\geq \lim_{\gamma\to 0}\frac{w(x,t)-v_k(x,t)}{\gamma}-3E_k r\\
    &\geq c'\tau-\sigma L\delta J^k r \qquad \quad\qquad\quad\hbox{ in } Q_r\cap\{v_k>0\}.
\end{aligned}
\end{equation*}
Using the definition of $\tau$, we obtain
\begin{equation}
    \label{grad estimate} \hat{\nabla}_p v_k(x,t) = C_1 \cos\langle\, p, \hat{\nabla}v_k({{\mu}},-2r)\rangle-\sigma L\delta J^k r
\end{equation}
where $C_1:=c'C\eps^{-1}$ only depending on $L,\sigma$ (since $\eps$ is fixed).  

It follows from ($A_k$) and ($D_k$) that
\begin{equation}
    \label{cos p mu}
    \begin{aligned}
       \cos\langle\, p, \hat{\nabla}v_k({{\mu}},-2r)\rangle&=\frac{\hat{\nabla}_p v_k}{|\hat{\nabla}v_k|}({{\mu}},-2r)
       \geq \frac{1}{L}\hat{\nabla}_pv_k({{\mu}},-2r)\geq \frac{1}{2L^2}J^k.
    \end{aligned}
\end{equation}
Taking $\delta $ to be small enough only depending on $L$ and $\sigma$, \eqref{grad estimate} yields 
\begin{equation*}
    \hat{\nabla}_p v_k(x,t)\geq \frac{C_1}{2}\cos\langle\, p,\hat{\nabla} v_k({{\mu}},-2r)\rangle %\geq C_3 J^k  
    \quad \text{in } Q_r\cap \{v_k>0\}.
    \end{equation*}
%$$ 
Thus in $Q_r\cap \{v_k>0\}$,
\begin{equation}
    \label{x to mu}
    \cos\langle\, p,\hat{\nabla}v_k(x,t)\rangle=\frac{\hat{\nabla}_p v_k}{|\hat{\nabla}v_k|}(x,t) \geq \frac{C_1}{2L}\cos\langle\, p,\hat{\nabla} v_k({{\mu}},-2r)\rangle. %\quad \text{ with  $C_2:=\frac{C_1}{2L}$.}
\end{equation}

For $p\in \mathcal{S}^{d+1}$, set
\[\rho(p):=\frac{C_1}{8L}\cos\langle \,p, \hat{\nabla}v_k({{\mu}},-2r)\rangle.\]
For any $q\in B(p,\rho(p))$ we have
$   \sin\langle\, p,q\,\rangle\leq \rho(p)$
and thus 
\begin{equation*}
%    \label{est cos}
    \begin{aligned}
    \cos\langle\, q,\hat{\nabla}v_k(x,t)\rangle &\geq \cos\langle\, p,\hat{\nabla}v_k(x,t)\rangle-2\sin\langle\, p,q\,\rangle\\
    &\geq \frac{C_1}{2L}\cos\langle\, p,\hat{\nabla}v_k(\mu,-2r)\rangle-2\rho(p)\quad (\text{ by }\eqref{x to mu} )\\
    &=\frac{C_1}{4L}\cos\langle\, p,\hat{\nabla}v_k({{\mu}},-2r)\rangle.
\end{aligned}
\end{equation*}

In view of ($A_k$) and \eqref{cos p mu}, we get
\begin{equation*} \label{est cos Dk}
     \hat{\nabla}_q v_k(x,t)\geq \frac{C_1}{4L^2}\cos\langle\, p,\hat{\nabla}v_k({{\mu}},-2r)\rangle\geq \frac{C_1}{8L^4}J^k.
\end{equation*}
Since the above holds for all $q\in B(p,\rho(p))$,  there exists a larger cone $ \widehat{W}_{\theta_{k+1},{\mu}_{k+1}}$ for some ${\mu}_{k+1}\in\mathbb{R}^{d+1}$, $ S\in (0,1)$ and $\theta_{k+1}= \theta_k+S(\frac{1}{2}\pi-\theta_k)$ such that 
\[\hat{\nabla}_p v_k(x,t)\geq \frac{C_1}{8L^4} J^{k} \quad\text{ for all unit vector }p\in \widehat{W}_{\theta_{k+1},{\mu}_{k+1}} \text{ and } (x,t)\in Q_r.\]
Here $S$ is independent of $k$, because $\rho(p)$ only depends on the angle between $p$ and $\hat{\nabla}v_k({{\mu}},-2r)$. {From the iterative definition of $\theta_k$, we obtain $\theta_k=\frac{\pi}{2}-s^k(\frac{\pi}{2}-\theta_0)$ with $s=1-S$.}
We refer readers to \cite{C1alpha,caffarellipart} for more details.

Let $J:=\min\{{C_1}/{(4L^3)}, r\}$. Recalling $v_{k+1}(x,t)=\frac{1}{J}v_{k}(Jx,Jt)$, we obtain for all unit $p\in \widehat{W}_{\theta_{k+1},{\mu}_{k+1}}$
\[\hat{\nabla}_p v_{k+1}(x,t)=\hat{\nabla}_p v_k \geq \frac{C_1}{8L^4} J^{k}\geq \frac{1}{2L} J^{k+1} \hbox{ in } Q_1.\] We checked $(D_{k+1})$ and therefore by induction we conclude the proof of the theorem.

\end{proof}

Now we give the proofs of Lemmas 6.6-6.8. To simplify notations, we write $v:=v_k$,  $\vec{b}:=\vec{b}_k$ and  $X:=X_k$ in the following proofs. 

\medskip

{\bf Proof of Lemma \ref{prop 2.2}.}
%The proof mostly follows from the proof of Proposition 2.2 in \cite{C1alpha}. We only sketch the proof below.
First note that if $r\leq \frac{\eps}{2{L}}$ , then $v\leq \eps \text{ in } Q_{2r}$ from  ($A_k$) and the fact that $0\in\Gamma_0$.
Next observe that in $Q_1$, $g:=\hat{\nabla}_p v$ solves
\[g_t=(m-1)g\Delta v+2\nabla v\cdot \nabla g+(m-1)v\Delta g+\nabla g\cdot (\vec{b}(x+X)-\vec{b}(X))+(m-1)g\nabla\cdot \vec{b}\]
\begin{equation}\label{e.6.6}
    +\nabla v\cdot \hat{\nabla}_p \vec{b}(x+X)+(m-1)v\nabla\cdot \hat{\nabla}_p \vec{b}.
\end{equation}
By the condition $(A_k)(C_k)$,
\[|\nabla v\cdot \hat{\nabla}_p \vec{b}(x+X)|+|(m-1)v\nabla\cdot \hat{\nabla}_p \vec{b}|\leq \sigma L\delta {J^k}.\]

Now we apply Harnack's inequality to $g$, using \eqref{e.6.6},  in $({B_\frac{7}{8}}\times[-3r,3r])\cap \{v\geq\frac{1}{2}\eps \}$. As done in Proposition 2.2 in \cite{C1alpha}, if we restrict to a smaller region $({B_\frac{3}{4}}\times(-2r,2r)) \cap \{v\geq\eps\}$ for $r$ small enough (depending on $\eps$), there exist $C,C'$ (depending on $L,r,\eps$) such that
\[\hat{\nabla}_p v(x,t)\geq C\hat{\nabla}_p v({{\mu}},-2r)-C'\delta {J^k}.\]

By ($D_k$), we have $\hat{\nabla}_p v({{\mu}},-2r)\geq J^k$. Thus we can select $\delta$ small enough such that for some $C>0$% and all $J\in (0,1)$
\begin{equation}
    \label{est enlarge core}\hat{\nabla}_p v(x,t)\geq C\hat{\nabla}_p v({{\mu}},-2r) \text{ in } (B_\frac{3}{4}\times(-2r,2r))\cap\{u\geq \eps \}.
\end{equation}
To show the assertion, we need to show \begin{equation*}
\frac{v((x,t)+\gamma p)-v(x,t)}{\gamma}\geq \tau v(x,t)=\tau\eps
\end{equation*}
which holds by the definition of $\tau$ and \eqref{est enlarge core}.\hfill$\Box$

\medskip

{\bf Proof of Lemma \ref{prop 2.3}.}
 Let $f\in C^1(B_{1/2})$ be a non-negative function such that  
\[f=0 \hbox{ in }B_\frac{1}{4}; \quad f=\eps\hbox{ on }\partial B_{\frac{1}{2}};\quad |\nabla f|\leq 10\eps; \quad |\Delta f|\leq {10}\eps.
\]   For $\alpha\in (-2r,2r)$, define
\[\xi(x,t):=v(x,t)+\tau\gamma(v(x,t)+\eps(t+\alpha)-f(x))_+.\] We claim that $\xi$ is a subsolution in $\Sigma:=(B_{\frac{1}{2}}\times(-2r,-\alpha))\cap \{v\leq \eps\}$ if $\eps$ is small enough, independent of $r$. 
Let us follow \cite{C1alpha} and only point out the differences coming from the drift. We recall the operator $\calL_2$ defined in \eqref{cal L 2} and denote {the drift independent part} as $\widetilde{\mathcal{L}}$:
\begin{align}
\widetilde{\mathcal{L}}(\xi):=\xi_t-(m-1)\,\xi\,\Delta \xi-|\nabla \xi|^2,\label{def tL}
\end{align}
Let $g(s):=\tau\gamma s_+$ and thus $g'=\tau\gamma\, \chi_{\{s>0\}}$, $g''\geq 0$ in the sense of distribution. Below, we write $g=g(v+\eps(t+\alpha)-f)$.
Direct computations yield
\begin{align*}
&\xi_t=(v+g)_t=(1+g')v_t+\eps g',\\
    &\nabla \xi=\nabla(v+g)=(1+g')\nabla v-g'\nabla f.
\end{align*}% &\Delta \xi=(1+g')\Delta v-g'\Delta f+g''|\nabla (v-f)|^2,
  
Following the computations in Lemma 3.1 of \cite{C1alpha} and using $|\nabla v|\geq \frac{1}{L}$, we obtain
\[    \widetilde{\mathcal{L}}(\xi)\leq (1+g')\widetilde{\mathcal{L}}(v) -\left(\frac{1}{{L}^2}-C\eps\right)g'\quad \text{  with $C$ only depending on $L$ and $\sigma$}.\]

Since $\mathcal{L}_2(\xi)=\widetilde{\mathcal{L}}(\xi)-\nabla \xi\cdot (\vec{b}(x+X)-\vec{b}(X))-(m-1)\xi\nabla\cdot \vec{b}$, then
\begin{align*}
    \mathcal{L}_2(\xi) &\leq (1+g')\widetilde{\mathcal{L}}(v)-\left(\frac{1}{{L}^2}-C\eps\right)g'-\nabla \xi \cdot(\vec{b}(x+X)-\vec{b}(X))-(m-1)\xi\,\nabla\cdot \vec{b}\\
    %&=(1+g')\mathcal{L}_2v+(1+g')(\nabla v\cdot(\vec{b}(x+X)-\vec{b}(X))+(m-1)v\nabla\cdot \vec{b})\\
    %&\quad -(\nabla v+g'\nabla(v-f)) (\vec{b}(x+X)-\vec{b}(X))-(m-1)(v+g)\nabla\cdot \vec{b}-\left(\frac{1}{{L}^2}-C\eps\right)g'\\
    &= (1+g')\mathcal{L}_2(v)+g'\nabla f\cdot(\vec{b}(x+X)-\vec{b}(X))-(m-1)(g-g')\nabla \cdot \vec{b}-\left(\frac{1}{{L}^2}-C\eps\right)g'\\
        &= g'\nabla f\cdot(\vec{b}(x+X)-\vec{b}(X))-(m-1)g\nabla \cdot \vec{b}-\left(\frac{1}{{L}^2}-C\eps-(m-1)\nabla\cdot\vec{b}\right)g'.
\end{align*}
By $(D_k)$, we have $\|\vec{b}\|_\infty\leq \sigma, \|{\nabla}\vec{b}\|_\infty\leq \sigma\delta J^k$. Since we assumed $\delta\leq \eps$ and $J<1$, $\|\nabla \vec{b}\|_\infty\leq \sigma\eps$. Also for $(x,t)\in\Sigma$, $s:=v+\eps(t+\alpha)-f\leq\eps$ and hence $g(s)\leq \eps g'(s)$. We get 
\begin{align*}
     |g'\nabla f\cdot(\vec{b}(x+X)-\vec{b}(X))+(m-1)g
    \nabla \cdot \vec{b}|
    \leq \sigma\eps^2 g' \hbox{ in } Q_{1/2}.
\end{align*}
Thus $\mathcal{L}_2( \xi)\leq 0$ if $\eps $ is small enough.

The rest of the proof follows from the proof of Proposition 2.3 \cite{C1alpha}, where we compare $w$ and $\xi$ in $\Sigma$ to conclude that \begin{equation}
    w(x,-\alpha)\geq (1+\tau\gamma)v(x,-\alpha)
   \quad \text{ in }B_\frac{1}{4}\cap\{v\leq\eps\}
\end{equation}
for all $\alpha\in(-2r,2r)$.

\begin{comment}
Let us sketch the idea below: with the above, we are able to compare $w$ and $\omega$ and obtain
\[w\geq \omega\text{ in }(B_\frac{1}{2}\times (-2r,-\alpha))\cap\{v\leq\eps\}.\]
Since $f=0$ in $B_\frac{1}{4}$, 

\end{comment}

\hfill$\Box$

\medskip

{\bf Proof of Lemma \ref{prop 2.4}.} 
 Based on $(A_k)$,$(B_k)$ and the elliptic regularity estimate {applied to $v$},  one can argue as in Lemma 3.2 of  \cite{C1alpha} to conclude that 
\begin{equation}
\label{Dij}
    vD_{ij}v\geq -C_4,\quad \text{ for all }i,j=1,...,d \quad \hbox{ in } Q_{2r},
\end{equation}
where $C_4$ depends only on $L$, universal constants and the Lipschitz constant of $\Gamma(v)$. We will use this fact in the computation below.

\medskip

%For a given vector ${\mu }\in\mathcal{S}^{d-1}$, define $h$ such that
Define
\[h(x,t):=(1+\tau\gamma)v(x+(t+2r)\phi{\mu },t), \quad y:= x+(t+2r)\phi{\mu }.\] 
Note that $|y-x|\leq {\kappa}\tau\gamma$.  Lemma \ref{prop 2.3} implies that $w\geq h$ on the parabolic boundary of 
\[\Sigma:=({B_\frac{1}{4}}\times (-2r,2r))\cap \{v\leq \eps\}.\] We claim that $\mathcal{L}_2(h) \leq 0$ in $\Sigma$. Write $\tau':=\tau\gamma$. We have
\begin{align*}
    h_t &=(1+\tau')(v_t+v_{{\mu }}\phi),\\
    \nabla h&= (1+\tau')\left(\nabla v+v_{{\mu }}(t+2r)\nabla\phi\right),\\
   {\Delta h}&= {(1+\tau')\left(\Delta v+2(t+2r)\nabla v_{{\mu }}\cdot\nabla\phi+v_{{\mu }{\mu }}(t+2r)^2|\nabla\phi|^2+v_{{\mu }}(t+2r)\Delta\phi\right),}
  \end{align*}
 From \eqref{Dij} and the computations in Proposition 2.4 \cite{C1alpha}
\[    \widetilde{\mathcal{L}}(h) \leq (1+\tau')\widetilde{\mathcal{L}}(v) (y,t)-\tau'\left(\frac{1}{{L}}-C{\kappa}\right)
\]
where $\widetilde{\mathcal{L}}$ is given by \eqref{def tL} and $C$ depends only on $m,L,C_4,\sigma$. Thus
\begin{align*}
    \mathcal{L}_2(h)&\leq (1+\tau') \mathcal{\widetilde{L}}\,v (y,t)-\tau'\left(\frac{1}{{L}}-C{\kappa}\right)-\nabla h\cdot\left(\vec{b}(x+X)-\vec{b}(X)\right)-(m-1)h\nabla \cdot \vec{b}(x+X)\\
    &= (1+\tau')\mathcal{L}_2(v) (y,t)-\tau'\left(\frac{1}{{L}}-C{\kappa}\right)-(1+\tau')\nabla v\cdot\left(\vec{b}(x+X)-\vec{b}(y+X)\right)\\
    &\quad -(m-1)(1+\tau')v(y,t)\nabla \cdot \left(\vec{b}(x+X)-\vec{b}(y+X)\right)-(1+\tau')v_{{\mu }}(t+2r)\nabla\phi)\cdot\left(\vec{b}(x+X)-\vec{b}(X)\right)\\
    &\leq -\tau'\left(\frac{1}{{L}}-C{\kappa}\right)+(1+\tau')|\nabla v|\left\|D\vec{b}\right\|_\infty|x-y|+(m-1)(1+\tau') v\left\|D^2 \vec{b}\right\|_\infty |x-y|\\
    &\quad +(1+\tau')\left|v_{{\mu }}\right|(t+2r)\left|\nabla\phi\right| \,\left\|\nabla \vec{b}\right\|_\infty |x|.
\end{align*}
Now apply ($C_k$) and since $\delta\leq \eps$, we have $\left\|D\vec{b}\right\|_\infty\leq \sigma\eps, \left\|D^2\vec{b}\right\|_\infty\leq \sigma\eps^2$. Since $|\nabla\phi|\leq \kappa\tau'$, we obtain
\begin{align*}
\mathcal{L}_2( h)    %&\leq -\tau'\left(\frac{1}{{L}}-C{\kappa}\right)-C(1+\tau')\|{\nabla}\vec{b}\|_\infty|x-y|-C(1+\tau')r{\kappa}\tau'  -C\|D^2 \vec{b}\|_\infty |x-y|\\
    &\leq -\tau'\left(\frac{1}{{L}}-C{\kappa}\right)-\sigma L
    \eps{\kappa}\tau'-\sigma L \eps r{\kappa}\tau'  -\sigma L\eps^2  {\kappa}\tau'\\
    &\leq -\tau'\left(\frac{1}{L}-C\kappa-\sigma L{\kappa}\right)\leq 0  \quad \hbox{ in } \Sigma,
\end{align*}
if ${\kappa}$ is small enough.
By comparison principle {applied to $w$ and $h$ in $Q_{2r}$}  we can conclude that
\[w(x,t)\geq h(x,t) \geq  v(x+(t+2r)\phi(x)\mu,t) \text{ in }Q_{2r}. \]
\hfill$\Box$

%\begin{theorem}\label{col C1alpha} Let $u$ be a solution of \eqref{premain} and suppose $0\in \Gamma_0$. Assume conditions \eqref{cond i}\eqref{cond ii} and \[\|\vec{b}\|_{C^2}+\|\vec{b}_t\|_\infty<\infty\]hold in ${Q_2}$. Then there exists $\alpha\in (0,1)$ such that $\Gamma\cap Q_{1}$ is a $C^{1,\alpha}$ surface.\end{theorem}

\section{Discussion of traveling waves and potential singularities }\label{sec pres cor}

In this section we discuss evolution of solutions in two space dimensions, in several explicit scenario.

\subsection{A discussion on Traveling Waves}

For simplicity, we restrict to two space dimensions $d=2$. The drift is chosen as
\begin{equation}\label{drift_special}
\vec{b}({x_1}, x_2) := (\alpha(x_2), 0), \hbox{ where } \alpha \hbox{ is Lipschitz and bounded.}
\end{equation} 

 When $\alpha$ is periodic and $\max\{ \alpha\} <c$, it is shown in \cite{traveling} that {there exist} traveling wave solutions of the form $U(x+cte_1)$ for the corresponding pressure equation \eqref{dynamic eqn}, with the growth condition $\lim_{x_1\to\infty} \frac{U(x)}{x_1} = c$. While Lipschitz regularity of the solutions are established therein, the free boundary regularity and possibility of a corner remain open. 

Our regularity analysis cannot address the traveling waves themselves,  but we are able to say that such singularity, if at all, is of asymptotic nature. More precisely we show that dynamic solutions, used in \cite{numerics} to approximate the travelling waves,  stay smooth in any finite time interval.

\begin{theorem}\label{TW}
Let $u$ solve \eqref{premain} in $\R^2\times (0,\infty)$, with $\vec{b}$ given in \eqref{drift_special}, with the initial data
$u_0(x) = (x_1)_+$. Further impose that $\frac{u(x,t)}{x_1}  \to 1 \text{ as }x_1\to\infty.$ Then the following holds:
\begin{itemize}
\item[(a)] $u$ is uniformly Lipschitz continuous in $\mathbb{R}^2\times [0,\infty)$. 
\item[(b)]For any fixed $T>0$, there exists $\tau_0(T)>0$ such that for all $t\in [0,T]$ and $\tau\leq \tau_0$
\begin{equation*}
    \label{mono example}
    \partial_{x_1} u\pm \tau\partial_{x_2} u\geq 0.
\end{equation*}
\item[(c)]  $u$ is non-degenerate, and $\Gamma(u)$ is $C^{1,\alpha}$ in $\mathbb{R}^2\times [0,T]$.
\end{itemize}
\end{theorem}

\begin{proof}
 Let us rewrite \eqref{premain} with our choice of $\vec{b}$:
\begin{equation}\label{dynamic eqn}\partial_t u-(m-1)u\,\Delta u-|\nabla u|^2-\alpha( x_2)\, \partial_{x_1}u=0 .\end{equation} 
Define $\varphi(x,t):=(x_1+\sigma_1t)_+$ with $\sigma_1 : = \sup |\alpha| + 1$.  Then $\varphi$ is a supersolution of \eqref{dynamic eqn} with the same initial data as $u$, and thus $u \leq \varphi$. In particular, for any $\eps>0$
\begin{equation}\label{phi2 ini}
    u(x-\sigma_1\e e_1,\e)\leq \varphi(x-\sigma_1\e e_1,\eps)= (x_1)_+=u(x,0),
\end{equation}
where we denote the positive $x_1$ direction as $e_1$.

\medskip

For $\e>0$ let $u^\eps(x,t):=u(x-\sigma_1\eps e_1, t+\e)$. From \eqref{phi2 ini}, it follows that $u^\eps(\cdot,0)\leq u_0$. Since $u^\eps
$ also solves \eqref{dynamic eqn}, by comparison principle it follows that  $u^\eps\leq u$, and thus 
\begin{equation}
u_t-\sigma_1u_{x_1}\leq 0.
\end{equation}
Above inequality with \eqref{fundamental00} yields that $u$ is uniformly Lipschitz continuous in space and time.

\medskip

Next to show $(b)$,  for $\eps>0$ and $\sigma_2= \sup |\partial_{x_2}\vec{b}|$ we define
\[w(x,t):=\sup_{|y-x|\leq \eps e^{-{\sigma_2}t}}u(y-\eps e_1, t).\]
For each $x$, pick $y=y(x,t)$ that realizes the supremum.
As in the proof of Lemma \ref{lem nabla}, for a.e. $(x,t)\in \mathbb{R}^2\times (0,\infty)$ we have 
\[w_t(x,t)= (u_t-{ {\sigma_2} \eps } e^{-{\sigma_2}t}|\nabla u|)(y,t).
\]

Therefore for a.e. $(x,t)\in \mathbb{R}^2\times (0,\infty)$,
\begin{align*}
    &\quad w_t-(m-1)w\Delta w-|{\nabla}w|^2-\nabla w\cdot \vec{b}-(m-1)w\nabla \cdot \vec{b}\\
    &\leq -{ {\sigma_2} \eps } e^{-{\sigma_2}t}|{\nabla}w|+|{\nabla}w|\sup_{\vec{y}\in B(x,\eps e^{-{\sigma_2} t})}|\vec{b}({y}-\eps e_1)-\vec{b}({x})|\\
    &\leq (-{ {\sigma_2} \eps } e^{-{\sigma_2}t}+\eps e^{-{\sigma_2}t}\|\alpha'\|_\infty)|{\nabla}w| \leq 0,
\end{align*}
where for the second equality above we used the fact that $\vec{b}$ only depends on $x_2$. Thus $w$ is a subsolution. Since $w(\cdot,0)\leq u_0$, the comparison principle for \eqref{dynamic eqn} yields $w\leq u$. In particular we have
\[u(x,t)\geq \sup_{|y|\leq \eps e^{-{\sigma_2}T}}u(x+y-\eps e_1,t) \hbox{ for } 0\leq t\leq T,\]
which yields (b) with $\tau\leq \tan(\arcsin(e^{-{\sigma_2}T}))$. Since (a)-(b) imply \eqref{cond ii} and that $u$ is cone monotone, Proposition \ref{condlem nondeg'} and Theorem \ref{col C1alpha} yield $(c)$.

\end{proof}

\begin{remark}
Let us consider the travelling wave solution  $u(x,t)=U(x+cte_1)$ of \eqref{dynamic eqn} with smooth and periodic $\alpha$, studied in \cite{traveling} . It was shown there that, assuming non-degeneracy, the free boundary $\Gamma_0=\partial\{U(x)>0\}$ can be represented by a Lipschitz graph $x_1=f(x_2)$. 

Our analysis shows that under the same assumption the graph function $f$ is at least $C^{1,\alpha}$. Indeed $|\nabla U|$ is globally bounded due to Theorem 1 of \cite{traveling} and thus  \eqref{cond ii} holds for $u$. Now Theorem \ref{col C1alpha} applies to yield the desired regularity of $f$. This improvement suggests that singularity of the free boundary such as corner formulation could happen only when non-degeneracy fails.
\end{remark}

%Before stating more examples, we need the following technical lemma which is used for comparison.

\begin{comment}As for pressure variables, suppose $\varphi=\frac{m}{m-1}\psi^{m-1}$ for some $\psi$ satisfying the conditions of Lemma \ref{wk frm supersl}. Thus according to the lemma, to show that $\varphi$ is a supersolution to \eqref{premain} in $U$, we only need to show 
\[\mathcal{L}\varphi :=\varphi_t-(m-1)\varphi\Delta\varphi-|\nabla\varphi|^2-\nabla\varphi\cdot \vec{b}-(m-1)\varphi\nabla\cdot \vec{b}\geq 0\]
in the set $\{\varphi>0\}\cap U$.
\end{comment}

The rest of the section discusses examples of singular solutions that are not present in the zero drift problem. First we discuss global-time persistence and aggravation of corners.

\begin{theorem}\label{prop corner}
There exist solutions $u_1,u_2$ to \eqref{premain} in $Q$ with bounded smooth spatial vector fields and non-negative, Lipschitz initial data such that 
\begin{itemize}
    \item[1.] $u_1$ is stationary and $\Gamma(u_1)$ has a corner at the origin.
    \item[2.] For a finite time, there is a corner of shrinking angles on $\Gamma(u_2)$.
\end{itemize} 
\end{theorem}

\begin{proof}

%Suppose $V(0)=0$ and $V$ is independent of $t$. By Lemma \ref{streamline}, if $0\in \overline{\{u(\cdot,0)>0\}}$ then $0\in \overline{\{u(\cdot,t)>0\}}$ for all $t\geq 0$.We only need to construct a super solution to \eqref{premain} such that the support of the solution contains a corner with vertex $0$ and the corner preserves for all time or for a short time. 

%\medskip 
Write $(x,y)$ as the space coordinate.
Let
\[\vec{b}:=-\nabla \Phi(x,y)\text{ for some smooth function $\Phi$},\]
and then it can be checked directly that
\[u_1:=\max\{\Phi,0\}\]
 is a stationary solution to \eqref{premain}.
Notice $\Gamma_0(u_1)$ is the $0$-level set of $\Phi$ and we claim that if $\Phi$ is degenerate, the interface can be non-smooth.

 For example, we can take
\begin{equation*}
\Phi(x,y)=g(x)g(y)
    %\Phi(x,y)=\left\{\begin{aligned}&g(x)g(y) &\text{ if }x\geq 0,y\geq 0\\
%    &0 &\text{ otherwise}\end{aligned}\right. 
\end{equation*} 
where $g $ is a function on $\mathbb{R}$ that it is only positive in $(0,1)$.
 Then $\partial\{u_1>0\}$ is a square. In particular, $\partial\{u_1>0\}$ contains a Lipschitz corner at the origin.
% Actually we made $V=\nabla \varphi$. 

%Let $u_0=\varphi$ in ${{B_1}}$ and $u_0$ is smooth with compact support. Let $u$ solves the equation \eqref{premain} with initial data $u_0$. By comparison$\varphi\geq u$for all time and so \[\{u>0\}\subset \{\varphi>0\}\subset\{x>0,y>0\}.\]Since $V=0$ at the origin, the stream line started at the origin is stationary. By Lemma \ref{streamline}, \[0\in \overline{\{u>0\}}.\]We know the support of $u$ has a corner which preserved for all time.

\medskip

Next we show (2). Take $\vec{b}:=(ax,by)$ (for a moment) and
\begin{equation*}
\varphi(x,y,t):=\left\{\begin{aligned}&\lambda(t)(x^2-k(t)y^2)_+\quad &\text{ if }x>0,\\
&0 &\quad\text{ otherwise },
\end{aligned}\right.
\end{equation*}
where 
\[\lambda(t)=e^{{{\sigma_1}} t}, k(t)=k_0 e^t \text{ for some }{{\sigma_1}},k_0>0.\]
Then the $\Gamma_t(\varphi)$ contains a corner with vertex at the origin. %We only consider $k(0)=k_0\in (0,1)$ and so the angle of the corner is initially larger than $\pi/2$. For $k_0\geq 1$, the examples are similar.

Let us show that $\varphi$ is a supersolution to \eqref{main} for $t\in (0,1/{{\sigma_1}})$.
Due to Lemma \ref{lem comp supersl}, we only need to check  this for $x> k^{1/2}|y|$. 
\begin{align}
   \mathcal{L}\varphi &:=\varphi_t-(m-1)\varphi\Delta\varphi-|\nabla\varphi|^2-\nabla\varphi\cdot \vec{b}-(m-1)\varphi\nabla\cdot \vec{b}\nonumber\\
&= (x^2-k y^2)\lambda'-\lambda k'y^2-(m-1)\lambda^2(x^2-k y^2)(2-2k)-4\lambda^2x^2-4\lambda^2 k^2 y^2\nonumber\\
    \quad &-2a\lambda x^2+2bk\lambda y^2-(m-1)\lambda(x^2-k y^2)(a+b)\nonumber\\
%=& (x^2-k y^2)(\lambda'-\lambda^2(m-1)(2-2k)-\lambda(m-1)(a+b))-\lambda k' y^2\\   &\quad-\lambda^2(4|x|^2+4k^2|y|^2)+\lambda(-2a|x|^2+2bk|y|^2)\\
     &= (x^2-k  y^2)(\lambda'-\lambda^2(m-1)(2-2k)-\lambda(m-1)(a+b)-2a-4\lambda^2)\nonumber\\
     &\quad+\lambda y^2(2bk-k'-4\lambda k-4\lambda k^2-2ak ) \nonumber \\
     &\geq (x^2-k  y^2)\lambda\left({{\sigma_1}}-\sigma(\lambda,m,k_0,a,b)\right)+\lambda y^2k((2b-1)-(4\lambda +4\lambda k+2a)). \label{ex 1}
\end{align}
Now we fix $a$ and take $b$ such that
\[2b-1\geq 4\lambda +8\lambda k_0+2a\geq 4\lambda +4\lambda k(t)+2a,\]
if ${{\sigma_1}}\geq 10$ and $t\leq 1/{{\sigma_1}}$.
Next we further take ${{\sigma_1}}$ to be large enough such that, the first part of \eqref{ex 1} is also non-negative. 
We conclude that for $t\in (0,1/\sigma_1)$, $\varphi$ is indeed a supersolution and its support contains a corner with angles shrinking from $2\arctan (k_0^{-\frac{1}{2}}) $ to $2\arctan (k(t)^{-\frac{1}{2}}) $. 

%As a remark, seeing from here, the sign of $\nabla\cdot \vec{b}(=a+b)$ does not affect.

\medskip

Now consider a solution $u_2$ with initial data $u_0$ such that $u_0=\varphi(x,y,0)$ in ${{B_1}}$ and $u_0\leq \varphi(x,y,0)$. By comparison, $\varphi\geq u_2$ for all times and so 
\[\Omega_t(u_2)\subset \Omega_t(\varphi)\subset\{x>k^{1/2}(t)|y|\}.\]
Since $\vec{b}=0$ at the origin, the origin is a one-point streamline. By Lemma \ref{streamline}, $0\in \overline{\Omega_t(u_2)}$ for all $t\geq 0$.
Thus $\Gamma_t(u_2)$ has a shrinking corner for a short time. Lastly since $u_2$ is compactly supported, we can truncate $\vec{b}$ to be bounded which does not affect $u_2$ and its support.

\end{proof}

Next we consider
formation of corners and cusps over time.

\begin{theorem} \label{prop cusp}
There is a solution $u$ to \eqref{premain} in $Q$ with some bounded continuous vector field and non-negative, bounded and Lipshitz initial data $u_0$ such that:
\begin{itemize}
    \item[1. ] $\Gamma_0(u)$ is smooth;
    \item[2. ] $\Gamma_t(u)$ contains a corner/a cusp for a range of time. \end{itemize}
\end{theorem}

\begin{proof}

%Again we are going to take $\vec{b}=\vec{b}(x,y)$ with $\vec{b}(0,0)=0$. 
First we consider $\vec{b}:=-(x+|y|,\, y).$ We will construct a supersolution for this choice of $\vec{b}$. For some ${\sigma_0},\sigma_1,\eps>0$, set
$\lambda(t)={\sigma_0}\,e^{{\sigma_1}t},\, \alpha(t)=\eps t$ and
\[\varphi(x,y,t):=\lambda(t)x(x-\alpha(t)|y|)_+.\]
When $t=0$, the support of $\varphi$ is a half-plane, while for any $t>0$ there forms a corner on $\Gamma_t(\varphi)$.

In the positive set of $\varphi$ $(x>\alpha |y|)$, we have
\begin{align*}
   % &\varphi_t-(m-1)\varphi\Delta\varphi-|\nabla\varphi|^2-\nabla\varphi\cdot \vec{b}-(m-1)\varphi\nabla\cdot \vec{b}\nonumber\\
    \mathcal{L}\varphi&=\lambda'x(x-\alpha|y|)-\lambda\alpha' x|y|-(m-1)\lambda^2x(x-\alpha|y|)(2-\alpha x\delta_y)-\lambda^2\left|\left(2x-\alpha |y|,\alpha x\frac{y}{|y|}\right)\right|^2\\
    &\quad +\lambda\left(2x-\alpha |y|,\alpha x\frac{y}{|y|}\right)\cdot (x+|y|,y)+2(m-1)\lambda x(x-\alpha|y|).
    \end{align*}
Here $\delta_y$ is the Dirac mass of variable $y$.
Since $\delta_y\geq 0$, the above simplifies to
\begin{align*}
\geq & (x- \alpha|y|)(\lambda'x-2(m-1)\lambda^2 x +2(m-1)\lambda x)-\lambda \alpha' x|y|-\lambda^2|(x-\alpha|y|)+x|^2-\lambda^2\alpha^2 x^2\\
    &\quad +\lambda((x-\alpha|y|)+x)(x+|y|)-\lambda\alpha x|y|\\
\geq &(x- \alpha|y|)(\lambda'x-2m\lambda^2 x +2(m-1)\lambda x-\lambda^2(x-\alpha|y|))-\lambda^2 x^2-\lambda^2\alpha^2 x^2\\
    &\quad +\lambda x(x+|y|)-(\lambda\alpha +\lambda\alpha')x|y|.
\end{align*}   
Select 
${\sigma_1}=4m, {\sigma_0}\leq \frac{1}{2}e^{-4m}, \eps\leq 1/4$ and then
$\lambda'\geq 2(m-1)\lambda +2m\lambda^2$. Therefore for $t\in [0,1]$,
\begin{align*}
\mathcal{L}\varphi &\geq  -(\lambda^2+\lambda^2\alpha^2)(x-\alpha|y|)^2+\lambda  x^2+(\lambda-\lambda\alpha-\lambda\alpha')x|y|\\
    &\geq (\lambda-\lambda^2(1+\eps^2t^2))|x|^2+ \lambda(1- \eps-\eps t) x|y|\geq 0.
\end{align*}
In the last inequality we used that $\lambda\leq 1/2, \eps+\eps t\leq 1/2$.

Thus $\varphi$ is a supersolution in $\mathbb{R}^2\times [0,1]$. Now
 $u_0=\varphi(x,y,0)$ in ${B_1}$ and $u$ be a solution with initial data $u_0$. Then by comparison we conclude that a corner forms on $\Gamma_t(u)$ for $t>0$.

% be such that for some $h>0$\[u_0(x,y)=h\,\varphi(x,y, 0) \text{ for }(x,y)\in B({h}),\, u_0(x,y)\leq h\,\varphi(x,y, 0)\]and $u_0=0$ outside $B({2h})$. Let $u$ be a solution to \eqref{premain} with initial data $u_0$. By finite propagation property, we can select $h$ small enough such that $u=0$ on $(\partial {{B_1}})\times [0,1]  $. This can be done by constructing a suitable supersolution. We skip the details.Then by comparison $u\leq \varphi$ for $t\leq 1$. Hence
%\[\{u>0\}\subset\{\varphi>0\}.\]Apply Lemma \ref{streamline}, $0\in \{u>0\}$.

\medskip

Next we show the possibility of the formation of cusps. Consider
\[    \vec{b}:=(x\log x-10 x^{1-\delta},0).\]
which is continuous but not Lipschitz continuous at $x=0$. In particular in our barrier argument we will use approximations. For some ${\sigma_2}$ is large enough, let
\[\alpha(t):=1+\tau(\tau-t),\,\lambda:=e^{\sigma_2t}\]
and let $\tau,\eps,\delta>0$ be such that
\begin{equation}
    \label{exp assumption}
    1>\delta\geq \frac{2\tau^2}{1-\tau^2},\; \tau\leq 1/2,\; e^{2{\sigma_2}\tau}\leq 1.
\end{equation}
Set
\begin{equation*}
\varphi_\eps(x,y,t):=\left\{
\begin{aligned}&\lambda(t)(x^2-(|y|+\eps)^{2\alpha(t)})_+ \quad &\text{ if }x\geq 0,\\
&0\quad &\text{ otherwise}.
\end{aligned}\right.
\end{equation*}
Then as $\eps\to 0$, for $x\geq 0$, 
\[\varphi_\eps(x,y,t)\to\varphi(x,y,t):=\lambda(t)(x^2-|y|^{2\alpha(t)})_+.\]
Directly from the definition, the support of $\varphi$ is smooth when $\alpha>1$, while a cusp appears when $\alpha=1$ i.e. $t>\tau$.   Set the domain \[\Sigma_\e:=\bigcup_{t\in [0, 2\tau]}\left(\left(\frac{1}{2}\geq x\geq (|y|+\eps)^{\alpha(t)}\right)\times \{t\}\right).\] 
Let us check that $\varphi_\eps$ is a supersolution to \eqref{premain} in $\Sigma_\e$. Notice
\begin{align*}
    &\partial_y (|y|+\eps)^{2\alpha}=2\alpha (|y|+\eps)^{2\alpha-1}\frac{y}{|y|},\\
&\partial_{yy} (|y|+\eps)^{2\alpha}=2\alpha(2\alpha-1) (|y|+\eps)^{2(\alpha-1)}+2\alpha(|y|+\eps)^{2\alpha-1}\delta_y\geq 2\alpha(2\alpha-1) (|y|+\eps)^{2(\alpha-1)}.
\end{align*}
By direct computation, in $\Sigma$
\begin{align*}
    %&\partial_t\varphi_\eps-(m-1)\varphi_\eps\Delta\varphi_\eps-|\nabla\varphi_\eps|^2-\nabla\varphi_\eps\cdot \vec{b}-(m-1)\varphi_\eps\nabla\cdot \vec{b}\nonumber\\
    \mathcal{L}\varphi_\eps&\geq  (x^2- (|y|+\eps)^{2\alpha})(\lambda'-\lambda^2(m-1)(2-2\alpha(2\alpha-1)(|y|+\eps)^{2(\alpha-1)})-\lambda(m-1)\nabla\cdot \vec{b})\nonumber\\
    & -\lambda \alpha' (|y|+\eps)^{2\alpha}\log (|y|+\eps)^2 -\lambda^2(4|x|^2+4\alpha^2(|y|+\eps)^{4\alpha-2})-2\lambda\left(\left(x,-2\alpha (|y|+\eps)^{2\alpha-1}\frac{y}{|y|}\right)\cdot \vec{b}\right)\nonumber
\end{align*}
Note we can assume $\alpha\geq 1/2$ and $\nabla\cdot \vec{b}\leq \sigma$ for some universal $\sigma$ in $\Sigma$, and therefore the above
\begin{align*}
&\geq  (|x|^2- (|y|+\eps)^{2\alpha})(\lambda'-2\lambda^2(m-1)-\sigma\lambda(m-1))-\lambda \alpha'(|y|+\eps)^{2\alpha} \log (|y|+\eps)^2\nonumber\\
    &\quad  -\lambda^2(4(|y|+\eps)^{2\alpha}+4\alpha^2(|y|+\eps)^{4\alpha-2})+2\lambda(-x^2\log x+10 x^{2-\delta})\nonumber\\
   % \geq & (|x|^2- (|y|+\eps)^{2\alpha})(\lambda'-\lambda^2(m-1)(2-4\lambda^2)-\lambda(m-1)\nabla\cdot \vec{b})+\lambda \tau(|y|+\eps)^{2\alpha} \log (|y|+\eps)^2\nonumber\\
   % &\quad  -\lambda^2(4(|y|+\eps)^{2\alpha}+4\alpha^2(|y|+\eps)^{4\alpha-2})+2\lambda(-x^2\log x+10 x^{2-\delta})\nonumber\\
    &=: A_1+A_2+A_3+A_4.
\end{align*}

To have $A_1\geq 0$, we only need
\[\lambda=e^{{\sigma_2}t}\leq e^{2\sigma_2\tau}\leq 2\text{ and }\sigma_2\geq (\sigma+4 )(m-1).  \]
Using $r^2\log r$ is negative and decreasing for $r\in [0,\frac{1}{2}]$ and \eqref{exp assumption}, we have 
\begin{align*}
    A_2&=\lambda\tau (|y|+\eps)^{2\alpha} \log (|y|+\eps)^2 \quad &&\text{( $\alpha'=-\tau$)}\\
    &\geq 4\lambda\tau (|y|+\eps)^{2\alpha} \log (|y|+\eps)^\alpha \quad&&\text{( $\alpha\geq 1/2$)}\\
    &\geq 4\lambda\tau x^{2} \log x\geq 2\lambda x^2\log x \quad &&\text{( $x\geq(|y|+\eps)^\alpha$, $2\tau\leq 1$)}.
\end{align*}
%In the last inequality, we applied that $2\tau\leq 1$.
Also note by \eqref{exp assumption}, we have $\lambda\leq 1, \alpha\leq 1+\tau^2\leq 2$, $4\alpha-2\geq \alpha(2-\delta)$, So
\begin{align*}
A_3    = -4\lambda^2 (|y|+\eps)^{2\alpha}-4\lambda\alpha^2(|y|+\eps)^{4\alpha-2}\geq -4\lambda x^{2}-16\lambda^2 x^{{(2\alpha-1})/\alpha}\geq -20\lambda x^{2-\delta}.
\end{align*}
In all 
$\Sigma_{i=1}^4A_i\geq 0$. %4\lambda\tau x^2\log x-20\lambda x^{2-\delta}+2\lambda(-x^2\log x+10 x^{2-\delta})\geq 0.\]
We proved that $\varphi_\eps$ is a supersolution in $\Sigma_\e$, so by Lemma \ref{lem comp supersl}, it is a supersolution in $B_\frac{1}{2}\times [0,2\tau]$.

\medskip

%\textcolor{This is really a good point. Let us look at the weak formulation and consider ( for $\varrho\sim \varphi_\eps^{1/(m-1)}$)
%\[\iint -\varrho\eta_t+\nabla\varrho^m\nabla\eta+\varrho\vec{b}\cdot\nabla\eta\]
%in a $\delta$-nbhd of the boundary where $\eta$ is a test function supported in a $\delta$-nbhd of one free boundary point. We only need to check it converges to $0$ as $\delta\to 0$.
%This is true as long as $\varphi^{1/m-1},\nabla\varphi_\eps^{m/m-1}\to 0$ near the free boundary, which is indeed the case.}

Now for $h\in (0,1)$, we select $u^h_{0,\eps}$ to be smooth with initial data $u^h_{0,\eps}=h\,\varphi_\eps(\cdot,0)$ in $B_h$ and $u^h_{0,\eps}=0 \text{ in } B^c_{2h}$. Let $u^h_\eps$ solve \eqref{main} with vector field $\vec{b}$ and initial data $u^h_{0,\eps}$. 
By finite propagation property, we can take $h$ to be small enough such that for all $\eps\in (0,1)$ 
\[u^h_\eps(\cdot,t)=0 \quad\text{ on }(\partial B_{\frac{1}{2}})\times [0,2\tau].\]
By comparison (which is valid since $\vec{b}$ is smooth in $\Sigma_\e$), $u^h_\eps\leq \varphi_\eps$ in $B_{\frac{1}{2}}\times [0,2\tau]$. Now passing $\eps\to 0$ gives a solution $u^h$ with initial data $h\varphi(\cdot,0)$ in $B_h$
such that $u^h\leq \varphi$ for $t\in [0,2\tau]$. As before we conclude by the geometry of $\Omega(\varphi)$ and Lemma \ref{streamline} that a cusp appears for $\tau <t <2\tau$.

\medskip

\end{proof}

%\begin{thebibliography}{10}

%\end{thebibliography}

\newpage

\appendix

\section{Proof of Lemma \ref{lem comp supersl}}

Let us only consider the case when $U=\mathbb{R}^d$. The case of $U=B_1$ follows similarly. 

Fix one non-negative $\phi\in C_c^\infty(\mathbb{R}^d\times [0,T))$. Denote \[U_0:=\{\phi>0\}\cap \{\psi>0\}.\] 
For any $\eps>0$, take finitely many space time balls $U_i, i=1,...,n$ such that
\begin{itemize}
    \item[1. ] for each $ i\geq 1$, $|U_i|\leq \eps^d$ and $U_i$ is in the $\eps$-neighbourhood of $\Gamma(\psi)$,
    
    \item[2. ] $\{ U_i\}_{i=1,...,n}$ is an open cover of $\Gamma(\psi) \cap \{\phi>0\}.$
\end{itemize}
Since $\Gamma(\psi)$ is of dimension $d-1$, we can assume
\begin{equation}
    \label{appc n}
    n\lesssim 1/\eps^{d-1}.
\end{equation}
Take a partition of unity $\{\rho_i, i=0,...,n\}$ which is subordinate to the open cover $\{U_i\}_{i\geq 0}$. Then for $i\geq 1$, 
\begin{equation}
\label{appc rho}
    |\nabla \rho_i|+|\partial_t\rho_i|\lesssim 1/\eps.
\end{equation}  

By the assumption, $\psi$ is a supersolution in the interior of its positive set. And since $\eps$ can be arbitrarily small, to show \eqref{wk frm supersl} we only need to show
\begin{equation*}
 I_\eps:=\sum_{i=1}^{n(\eps)}\left(\int_0^T\int_{\mathbb{R}^d} \psi\,(\phi\rho_i)_t-(\nabla \psi^m+\psi\,\vec{b})\nabla (\phi\rho_i)\;dxdt - \int_{\mathbb{R}^d} \psi(0,x)\phi(0,x)\rho_idx\right)\to 0
\end{equation*}
as $\eps\to 0$.

By property 1 of $U_i$ and the regularity assumption on $\psi $, in all $U_i, i\geq 1$ we have 
\[\psi\leq C\eps^{\frac{1}{\alpha}},\quad |\nabla \psi^m|\leq C\psi^{m-\alpha}|\nabla \psi^\alpha|\leq  C\eps^{\frac{m-\alpha}{\alpha}}. \]
Now from \eqref{appc n}, \eqref{appc rho} and $\alpha<m$, it follows that
\begin{align*}
    |I_\eps|&\leq C\eps^{-d+1}\left(\iint_{U_i} \frac{1}{\eps} (\psi+|\nabla \psi^m|)\;dxdt +\int_{U_i\cap\{t=0\}} \psi(0,x)dx\right)\\
    &\leq C(\eps^{\frac{1}{\alpha}}+\eps^{\frac{m-\alpha}{\alpha}}+\eps)
\end{align*}
which indeed converges to $0$ as $\eps\to 0$.

\section{Sketch of the proof of Lemma \ref{lem delta}}

We follow the idea of Lemma 9 \cite{caffarellipart} and compute
\[\Delta f(0)=\overline{\lim}_{r\to 0}\left(\oint_{B_r}f(x)-f(0)dx\right).\]
Without loss of generality, suppose locally near the origin
\[f(x)=\inf_{|\nu|=1}h\left(x+\psi(x)\nu\right),\]
because otherwise $\Delta f(0)=0$.
Choosing an appropriate system of coordinates, we can have
\begin{align*}
    &f(0)=h(\psi(0)e_n);\\
    &\nabla\psi(0)=\alpha e_1+\beta e_n.
\end{align*}

We will evaluate $w$ by above by choosing $\nu(x)=\frac{\nu_*(x)}{|\nu_*(x)|}$ where
\[\nu_*(x):=e_n+\frac{\beta x_1-\alpha x_n}{\psi(0)}e_1+\frac{\gamma}{\psi(0)}\left(\Sigma_{i=2}^{d-1}\,x_i\,e_i\right)\]
where $\gamma$ satisfies
\[(1+\gamma)^2=(1+\beta)^2+\alpha^2.\]
With this choice of $\nu$, we define $y:=x+\psi(x)\nu(x)$ and so
$y(0)=\psi(0)e_n$.
After direct computations (also see \cite{caffarellipart}), we can write
\[y=Y_*(x)+\psi(0)e_n+o(|x|^2)\]
such that the first-order term, except the translation $\varphi(0)e_n$, satisfies
\[Y_*(x):=x+(\alpha x_1+\beta x_n)e_n+(\beta x_1-\alpha)e_1+\gamma\Sigma_{i=1}^d x_ie_i.\]  
Hence $Y_*(x)$ is a rigid rotation plus a dilation and we have
\begin{equation}
    \label{B.1}
     \left|\frac{D (Y_*-x)}{Dx}\right|\leq \sigma\|\nabla\psi\|_\infty.
\end{equation}

Then
\begin{align*}
    \oint_{B_r}f(x)-f(0)dx&\leq 
    \oint_{B_r}h(y(x))-h(y(0))dx\\
    &\leq 
    \oint_{B_r}h(y(x))-h(Y_*(x)+y(0))dx+
    \oint_{B_r}h(Y_*(x)+y(0))-h(y(0))dx. 
\end{align*}
By the condition on $\psi$ and the computations done in Lemma 9 \cite{caffarellipart}, the first term is non-positive. 

Since $h$ is smooth, the second term converges to 
\[\left(\left|\frac{D Y_*}{D x}\right|_{x=0}\right)^2(\Delta h)(y(0))\quad \text{ as }r\to 0.\]
Now using \eqref{B.1} and the assumption that $\Delta h\geq -C$ and $\|\nabla\psi\|_\infty\leq 1$, we get
\begin{align*}
    \oint_{B_r}f(x)-f(0)dx&
    \leq \oint_{B_r}h(Y_*(x)+y(0))-h(y(0))dx\\
    &\leq (1+\sigma \|\nabla\psi\|_\infty)(\Delta h) (y(0))+\sigma \|\nabla\psi\|_\infty C.
\end{align*}
Thus we finished the proof.

\section{Proof of Lemma \ref{lem nabla}}

Let us suppose $x=0$ and $f(0)=h(y)$ for a unique $y$. We only compute $\partial_1 f(0)=\partial_{x_1}f(0)$. If $\nabla h(y)=0$, it is not hard to see
\[\partial_1 f(0)=\partial_1 h(y)=0.\]
Next suppose $\nabla h(y)\ne 0$. We know that $h$ obtains its minimum over $B(0,\psi(0))$ at point $y\in\partial B(0,{\psi(0)})$. Let us assume
\[y=(y_1,y_2,0,...,0),\quad \text{ and thus }|y_1|^2+|y_2|^2=(\psi(0))^2.\] 

For smooth $h$, it is not hard to see that
\[\nabla h(y)=-k y\quad \text{with }k=\frac{|\nabla h|}{\psi(0)}.\] Near point $y$
\[h(x)-h(y)=-ky_1(x_1-y_1)-ky_2(x_2-y_2)+o(|x-y|).\]
To estimate $w((\delta,0,...,0))$, consider the leading terms:
\[A(\delta):=-ky_1(x_1-y_1)-ky_2(x_2-y_2)=-ky_1(x_1-\delta)-ky_2x_2+ky^2_1+ky^2_2-ky_1\delta.\]

By a standard argument, under the constrain
\[|x_1-\delta|^2+|x_2|^2+|x_3|^2+...+|x_n|^2\leq \psi(\delta,0...0)^2,\]
 $A(\delta)$ achieves its minimum at 
\[x_1=y_1\psi(\delta,0...0)/(y^2_1+y^2_2)^{\frac{1}{2}}+\delta,\,x_2=y_2\psi(\delta,0...0)/(y^2_1+y^2_2)^{\frac{1}{2}}\]
with value
\[-k\psi(\delta,0...0)(y^2_1+y^2_2)^{\frac{1}{2}}+ky^2_1+ky^2_2-ky_1\delta=-k\psi(\delta,0...0)\psi(0)+k\psi(0)^2-ky_1\delta.\]
Thus
\[\partial_1 f(0)=\lim_{\delta\to 0}A(\delta)/\delta=-k\psi(0)\,\partial_1 \psi(0)-ky_1.\]
Notice that $\partial_1 h(y)=-ky_1$. So we find
\[\partial_1 f(0)-\partial_1 h(y)=-k\psi(0)\,\partial_1 \psi(0)=-|\nabla h|\,\partial_1 \psi(0).\]
This leads to the conclusion.

%\bibliography{freeboundary}
%\bibliographystyle{plain}

\end{document}